\documentclass[12pt,a4paper,dvipsnames]{article}

\usepackage{etoolbox,fullpage,amssymb,amsmath,amsfonts,amsthm,paralist,graphicx,color,float, mathtools,tikz,xfrac}
\usepackage[utf8]{inputenc}
\usepackage[english]{babel}
\usepackage[colorlinks=true,linkcolor=BrickRed,citecolor=blue]{hyperref}
\usepackage[normalem]{ulem}
\usepackage{centernot}
\usepackage{ifthen,xkeyval, tikz, calc, graphicx}
\usepackage{xcolor}
\usepackage{framed}
\usepackage[textsize=scriptsize, backgroundcolor=blue!20!white,bordercolor=red]{todonotes}
\usepackage{dsfont}
\usepackage{mathrsfs}
\usepackage{comment}
\usepackage{stmaryrd}
\usepackage{enumitem}
\usepackage{caption}
\usepackage{subcaption}
\usepackage{todonotes}
\usepackage{thmtools}
\usetikzlibrary{plotmarks}
\usetikzlibrary{shapes.misc}
\usepackage{thmtools}
\usepackage{thm-restate}
\usepackage{cleveref}
\usepackage{orcidlink}

\newtheorem{theorem}{Theorem}[section]
\newtheorem{lemma}[theorem]{Lemma}
\newtheorem{prop}[theorem]{Proposition}
\newtheorem{corollary}[theorem]{Corollary}

\theoremstyle{definition}
\newtheorem{definition}[theorem]{Definition}
\newtheorem{example}{Example}[section]

\newtheorem{assumption}{Assumption}
\newtheorem{assumptionalt}{Assumption}[assumption]
\newenvironment{assumptionp}[1]{
  \renewcommand\theassumptionalt{#1}
  \assumptionalt
}{\endassumptionalt}

\theoremstyle{remark}
\newtheorem{remark}[theorem]{Remark}

\def\acknowledgementsname{Acknowledgments}
\newenvironment{acks}[1][\acknowledgementsname]{\noindent\textbf{#1.}\space\ignorespaces}{\par}

\definecolor{shadecolor}{named}{GreenYellow}

\makeatletter
\newcommand{\pushright}[1]{\ifmeasuring@#1\else\omit\hfill$\displaystyle#1$\fi\ignorespaces}
\newcommand{\pushleft}[1]{\ifmeasuring@#1\else\omit$\displaystyle#1$\hfill\fi\ignorespaces}
\makeatother

\newcommand{\pla}{\mathbb P_\lambda}
\newcommand{\ela}{\mathbb E_{\lambda}}
\newcommand{\Pcal}{\mathcal P}

\newcommand{\Dcal}{\mathcal{D}}
\newcommand{\Lcal}{\mathcal{L}}
\newcommand{\Ical}{\mathcal{I}}
\newcommand{\Jcal}{\mathcal{J}}
\newcommand{\Kcal}{\mathcal{K}}
\newcommand{\Vcal}{\mathcal{V}}
\newcommand{\E}{\mathbb E}
\newcommand{\Ecal}{\mathcal E}
\newcommand{\Bcal}{\mathcal B}

\newcommand{\R}{\mathbb R}
\newcommand{\X}{\mathbb X}
\newcommand{\Xcal}{\mathcal X}

\newcommand{\Rd}{\mathbb R^d}
\newcommand{\Z}{\mathbb Z}

\newcommand{\N}{\mathbb N}

\newcommand{\HypTwo}{{\mathbb{H}^2}}
\newcommand{\HypDim}{{\mathbb{H}^d}}
\newcommand{\dd}{\mathrm{d}} 
\newcommand{\C}{\mathscr {C}}
\newcommand{\Complex}{\mathbb C}

\newcommand{\Leb}{{\rm Leb}}

\DeclareMathOperator*{\esssup}{ess\,sup}
\DeclareMathOperator*{\essinf}{ess\,inf}

\newcommand{\LandauBigO}[1]{\mathcal{O}\left(#1\right)}

\newcommand{\ConstantTriangle}{C_\Delta}
\newcommand{\cbar}[1]{\overline{c}_{#1}}

\DeclareMathOperator{\arcosh}{arcosh}

\DeclareMathOperator{\artanh}{artanh}

\newcommand{\conn}[3]{#1 \longleftrightarrow #2\:\textrm {in}\: #3}
\newcommand{\adja}[3]{#1 \sim #2\:\textrm{in}\: #3}

\newcommand{\tlam}{\tau_\lambda}

\newcommand{\trilam}{\triangle_\lambda}

\newcommand{\Greenlam}{g_\lambda}
\newcommand{\OpGreenlam}{\mathcal{G}_\lambda}

\newcommand{\Id}{\mathds 1}
\newcommand{\Optlam}{\mathcal{T}_\lambda}
\newcommand{\OptlamCritL}{\mathcal{T}_{\lambda_\mathrm{c}(L),L}}

\newcommand{\Opconnf}{\varPhi}

\DeclarePairedDelimiter\abs{\lvert}{\rvert}
\DeclarePairedDelimiter\norm{\lVert}{\rVert}

\DeclarePairedDelimiterX{\inner}[2]{\langle}{\rangle}{#1, #2}

\newcommand{\EssIm}[1]{{\rm ess.Im}\left(#1\right)}

\newcommand{\dist}[1]{\mathrm{dist}_{\HypDim}\left(#1\right)}
\newcommand{\HSNorm}[1]{\left\lVert#1\right\rVert_{\rm HS}}

\newcommand{\orig}{o}
\newcommand{\origin}[1]{{o_{#1}}}
\newcommand{\e}{\mathrm{e}}

\newcommand{\connf}{\varphi}

\definecolor{darkorange}{RGB}{255,165,0}
\definecolor{altviolet}{RGB}{139,0,139}
\definecolor{turquoise}{RGB}{64,224,208}
\definecolor{lblue}{RGB}{173,216,230}
\definecolor{violet}{RGB}{238,130,238}
\definecolor{darkgreen}{RGB}{0,100,0}
\definecolor{lgreen}{RGB}{144,238,144}

\numberwithin{equation}{section}
\allowdisplaybreaks

\title{Non-Uniqueness Phase in Hyperbolic Marked Random Connection Models using the Spherical Transform}
\author{Matthew Dickson\footnote{University of British Columbia, Department of Mathematics, Vancouver, BC, Canada, V6T 1Z2; Email: dickson@math.ubc.ca; \orcidlink{0000-0002-8629-4796}~https://orcid.org/0000-0002-8629-4796}}
 
\date{}

\begin{document}
\maketitle

\vspace{-1em}

{\centering{ \today}\par}

\vskip-3em

\begin{abstract}
A non-uniqueness phase for infinite clusters is proven for a class of marked random connection models on the $d$-dimensional hyperbolic space, $\HypDim$, in a high volume-scaling regime. The approach taken in this paper utilizes the spherical transform on $\HypDim$ to diagonalize convolution by the adjacency function and the two-point function and bound their $L^2\to L^2$ operator norms. Under some circumstances, this spherical transform approach also provides bounds on the triangle diagram that allows for a derivation of certain mean-field critical exponents. In particular, the results are applied to some Boolean and weight-dependent hyperbolic random connection models. While most of the paper is concerned with the high volume-scaling regime, the existence of the non-uniqueness phase is also proven without this scaling for some random connection models whose resulting graphs are almost surely not locally finite.
\end{abstract}


\smallskip



\section{Introduction}

Random connection models (RCMs) are random graph models in which the vertex set is a Poisson point process on some ambient space (with an intensity parameter $\lambda>0$), and the edges then exist independently with a probability that is prescribed by an \emph{adjacency function}, $\connf$, that depends upon the positions of the two vertices in question. \emph{Marked} random connection models are RCMs for which the ambient space is the product of an infinite space (say Euclidean $\Rd$ or hyperbolic $\HypDim$) with a probability space $\Ecal$ of marks. The adjacency function naturally depends on these marks, and this ensures that not all vertices behave in the same way (as they might well in a purely very symmetric ambient space). The primary aim of this paper is to prove marked RCMs on $\HypDim\times \Ecal$ exhibit a \emph{non-uniqueness} phase, in which there are almost surely infinitely many infinite clusters. Specifically, we show in Theorem~\ref{thm:NonUniqueness} that in a certain ``stretched-out" regime the percolation critical intensity $\lambda_\mathrm{c}$ (above which infinite components exist) and the uniqueness critical intensity $\lambda_\mathrm{u}$ (above which there is a unique infinite component) satisfy
\begin{equation}
\label{eqn:non-uniqueness}
    \lambda_\mathrm{c} < \lambda_\mathrm{u}.
\end{equation}
Notably, this ``stretched-out" regime is not the ``spread-out" regime as seen in the lace expansion arguments of \cite{HarSla90,HeyHofLasMat19}.
As we shall see, the approach taken in this paper also helps to derive mean-field critical exponents for such models (see Theorem~\ref{thm:meanfield}). These results are also applied to Boolean disc models (with random radii) and weight-dependent RCMs on $\HypDim$. In particular, Proposition~\ref{prop:nonperturb} proves the existence of a non-uniqueness regime for some weight-dependent RCMs \emph{without} requiring the ``stretched-out" perturbation. Note that this proposition also proves a non-trivial property for some random graphs that are almost surely not locally finite.

While it is well-known that percolation models on amenable and Euclidean spaces have at most one infinite connected cluster (see \cite{Gri99,MeeRoy96} for example), by considering $\HypDim$ we can move beyond this and model non-amenable geometry. It is clear that Bernoulli bond percolation (BBP) on regular trees have a non-uniqueness phase, and \cite{benjamini1996percolation} conjectures that such a phase exists for BBP on any locally finite quasi-transitive non-amenable graph. This background is discussed further in Section~\ref{sec:background}.

There are three main ideas in this paper that allow us to show non-uniqueness. In a similar way to \cite{hutchcroft2019percolation} we first use the two-point (or connectedness) and adjacency functions to construct operators on $L^p$ spaces - in particular the Hilbert space $L^2$. This produces the critical intensity $\lambda_{2\to 2}$, at which the $L^2\to L^2$ operator norm of the two-point operator switches from being finite to infinite. Crucially, this operator critical intensity lies between the susceptibility critical intensity $\lambda_T$ ($=\lambda_{1\to 1}$ -- see Lemma~\ref{lem:generalOperatorBounds}) and the uniqueness critical intensity $\lambda_{\mathrm{u}}$:
\begin{equation}
    \lambda_{T}\leq \lambda_{2\to 2} \leq \lambda_{\mathrm{u}}.
\end{equation}
If one could show $\lambda_T=\lambda_{\mathrm{c}}$ and that one of these bounds is strict (it will be the lower bound in this argument), then a non-uniqueness phase would be proven. In practice we will find an upper bound on $\lambda_\mathrm{c}$ by comparing to a model on which the percolation and susceptibility thresholds are equal - we only prove $\lambda_T=\lambda_{\mathrm{c}}$ for models in which we also derive mean-field critical exponents. Furthermore, $\lambda_{2\to 2}$ can be bounded below by the reciprocal of the $L^2\to L^2$ operator norm of the adjacency operator. 

The second main idea is that this $L^2\to L^2$ operator norm of the adjacency operator can be written explicitly using the spherical transform of \cite{helgason1994geometric}. Specifically, in the case of \emph{unmarked} hyperbolic RCMs the  $L^2\to L^2$ operator norm of the adjacency operator is exactly
\begin{equation}
\label{eqn:SimpleSphericalTransform}
	\int_{\HypDim}\connf\left(\dist{x,\orig}\right)Q_d\left(\dist{x,\orig}\right) \mu\left(\dd x\right),
\end{equation}
where $\mathrm{dist}_{\HypDim}$ is the hyperbolic metric, $\mu$ is the hyperbolic measure, and the function $Q_d\colon \R_+\to \R_+$ is given by an explicit integral in \eqref{eqn:QdFunctionDefinition} (taking $\R_+=\left[0,\infty\right)$). By studying simple properties of $Q_d$, one can show that as $\connf$ becomes more stretched out, the integral \eqref{eqn:SimpleSphericalTransform} is much smaller than the total mass of $\connf$ (i.e. this same integral without $Q_d$). 

The third idea is that a geometric property of $\HypDim$ allows us to use a cluster of the RCM to dominate a branching process, which tells us that the susceptibility critical intensity is bounded above by a constant divided by the total mass of $\connf$. We will have therefore shown that if $\connf$ is sufficiently stretched out then the susceptibility critical intensity is strictly less than $\lambda_{2\to 2}$, which is less than or equal to the uniqueness critical intensity. The critical intensity $\lambda_{2\to 2}$ has the additional advantage that the triangle diagram can easily be shown to be finite for $\lambda<\lambda_{2\to 2}$, and standard methods then allow us to show mean field critical exponents hold.

\subsection{Outline of the Paper}

In Section~\ref{sec:RESULTS} we first present the model and main question more precisely, before giving the main non-uniqueness result of the paper in Theorem~\ref{thm:NonUniqueness}. Theorem~\ref{thm:meanfield} then gives conditions under which mean-field critical exponents can be proven. In Section~\ref{sec:SpecificModels} we apply these theorems to scaled Boolean disc models and weight-dependent hyperbolic RCMs.

Section~\ref{sec:prelims} contains some preliminary information, including more details on hyperbolic spaces and on constructing adjacency and two-point operators.

In Section~\ref{sec:Operators} we see what conclusions we can come to by using operators and operator norms without using significant properties of $\HypDim$. The main claim here is that we can bound $\lambda_{2\to 2}$ below using the $L^2\to L^2$ operator norm of the explicitly known adjacency operator. The $\lambda_{2\to 2}$ critical threshold also has a role in showing that the triangle diagram is finite.

In Section~\ref{sec:sphericaltransform} the spherical transform is used to evaluate the $L^2\to L^2$ operator norms. Here the main claim is that as $L\to\infty$ the spherical transform of the adjacency function and operator is dominated by the plain integral and the associated operator.

In Section~\ref{sec:NonUniqueness} it is demonstrated how the $L^2\to L^2$ critical threshold acts as a lower bound for the uniqueness threshold, while geometric considerations give an upper bound for the percolation critical threshold in terms of the plain integral of the adjacency function. For sufficiently large $L$, these bounds we have found are the required way around to demonstrate the existence of the non-uniqueness phase. Section~\ref{sec:ProofCritExponents} then proceeds to show that not only is the triangle diagram finite, but is small enough to use the results of \cite{caicedo2023criticalexponentsmarkedrandom}.

The applications of the main results for the Boolean disc model (in Section~\ref{sec:ProofBooleanModel}) and weight-dependent hyperbolic random connection models (in Section~\ref{sec:ProofWDRCM}) are proven in Section~\ref{sec:ProofSpecificModels}.

Interactions between scaling functions and adjacency functions are discussed in Appendix~\ref{app:ScalingFunctions}. The volume and length-linear scaling functions (introduced in Section~\ref{sec:PresentModelQuestion}) are shown to behave nicely for all adjacency functions, while all scaling functions are shown to behave nicely for some classes of adjacency function. However, some specific pairs of scaling and adjacency function are identified that behave unintuitively.

\section{Results}
\label{sec:RESULTS}
\subsection{Presenting the Model and the Questions}
\label{sec:PresentModelQuestion}

Let us now be more precise in our definition of marked hyperbolic RCMs.

\paragraph{Hyperbolic space.}
For $d\geq 2$, the $d$-dimensional hyperbolic space (denoted $\HypDim$) is the unique simply connected $d$-dimensional Riemannian manifold with constant sectional curvature $-1$. We will often use the Poincar{\'e} ball model (sometimes called \emph{conformal ball model}) of $\HypDim$. Let $\mathbb{B}=\left\{x\in\Rd\colon \abs*{x}<1\right\}$ be the open Euclidean radius ball in $\Rd$. The \emph{hyperbolic metric} on $\mathbb{B}$ then assigns length
\begin{equation}
    \ell\left(\gamma\right) := 2\int^1_0 \frac{\abs*{\gamma'(t)}}{1-\abs*{\gamma(t)}^2}\dd t
\end{equation}
to the curve $\gamma=\left\{\gamma(t)\right\}_{t\in\left[0,1\right]}$. In particular, this means that the origin in $\mathbb{B}$ (denoted $\orig$) and an arbitrary point $x\in \mathbb{B}$ are hyperbolic distance
\begin{equation}
\label{eqn:HypDistance}
    \dist{\orig,x} = 2\artanh \abs{x}
\end{equation}
apart, where $\abs{x}$ is the Euclidean distance between $\orig$ and $x$ in $\Rd$. For Borel measurable subsets $E\subset \mathbb{B}$, the \emph{hyperbolic measure} (denoted $\mu$) assigns mass
\begin{equation}
    \mu\left(E\right) := \int_E\frac{4}{\left(1-\abs*{x}^2\right)^2}\dd x.
\end{equation}
In Section~\ref{sec:prelims} below, the isometries of these spaces are described. In particular, they have a transitive family that means that the point identified by $\orig$ is arbitrary.

\paragraph{Marked hyperbolic random connection model.}
We first let $\Ecal$ be a Borel measure space with probability measure $\Pcal$. Then, given $\lambda>0$, the vertex set of the marked hyperbolic RCM, $\eta$, is distributed as a Poisson point process on $\HypDim\times \Ecal$ with intensity measure given by the product measure of the hyperbolic measure and the probability measure: $\lambda\nu := \lambda\mu\otimes\Pcal$. Two distinct vertices then form an edge independently of all other vertices and possible edges with a probability given by a measurable \emph{adjacency function} $\connf\colon \R_+\times\Ecal^2\to \left[0,1\right]$ that is symmetric under the transposition of the two mark arguments. Note that throughout we will assume that $\connf>0$ on a $\Leb\otimes \Pcal\otimes\Pcal$-positive set so that some edges do indeed occur. We write the event that two vertices $\mathbf{x},\mathbf{y}\in\eta\subset\HypDim\times\Ecal$ form an edge as $\mathbf{x}\sim\mathbf{y}$. Given $x,y\in\HypDim$ and $a,b\in\Ecal$,
\begin{equation}
\label{eqn:EdgeProbs}
    \mathbb{P}\left(\left(x,a\right)\sim\left(y,b\right)\right) = \connf\left(\dist{x,y};a,b\right).
\end{equation}
We then use $\xi$ to denote the whole random graph, and we use $\pla$ to denote the law of $\xi$ (and $\ela$ for the associated expectation). A more precise and complete description of construction of an RCM can be found in \cite{HeyHofLasMat19}, where the reader need only replace instances of $\Rd$ with the space $\HypDim\times \Ecal$ we use here.

Most of the results contained in this paper hold in a perturbative regime, where we make longer edges more likely. This is achieved by modifying the adjacency function with a scaling function. For all $L>0$ let $\sigma_L\colon\R_+\to \R_+$ be an increasing bijection such that
\begin{itemize}
    \item $\sigma_1(r)=r$,
    \item for all $r>0$, $L\mapsto \sigma_L(r)$ is increasing and $\lim_{L\to\infty}\sigma_L(r)=\infty$,
\end{itemize}
and define $\connf_L\colon\R_+\times\Ecal^2\to\left[0,1\right]$ by
    \begin{equation}
    \label{eqn:scaledAdjFnct}
        \connf_L\left(r;a,b\right) := \connf\left(\sigma_L^{-1}(r);a,b\right).
    \end{equation}
    We call $\sigma_L$ the scaling function and $\connf_L$ the scaled adjacency function. More generally, the presence of a subscript $L$ in our notation then indicates that the scaled adjacency function $\connf_L$ is being used in the place of the reference adjacency function $\connf$. 
    \begin{remark}
        Note that this is different to the concept of \emph{spread-out} models that appear in lace expansion arguments like \cite{HarSla90,HeyHofLasMat19}, because we do not multiply the adjacency function by a small factor in addition to changing the length scales. In these lace expansion arguments this small factor is required to ensure the lace expansion converges, but this is not required in the arguments here. If one included such a small factor (uniformly in the marks), then objects such as the operator norm $\norm*{\Opconnf}_{2\to 2}$ change in an exactly predictable way, and the monotonicity makes objects such as the triangle diagram $\trilam$ easy to control. This means that our argument can work with this small factor, but it is not at all necessary.
    \end{remark}

    We shall be particularly interested in a scaling function that linearly scales the volume of balls. For all dimensions $d\geq 1$ and radii $r\geq0$, define
    \begin{equation}
        \mathbf{V}_d(r):= \int^r_0\left(\sinh t\right)^{d-1}\dd t.
    \end{equation}
    Note that the function $\mathbf{V}_d\colon \R_+\to \R_+$ is a bijection. We can then define the \emph{volume-linear scaling function} to be
    \begin{equation}
        s_L(r) := \mathbf{V}_d^{-1}\left(L \mathbf{V}_d(r)\right)
    \end{equation}
    for all $d\geq 1$ and $r\geq 0$. The name arises because this scaling linearly transforms the volume of hyperbolic balls of any radius. Scaling functions are discussed further in Appendix~\ref{app:ScalingFunctions}. For example, the advantages of the volume-linear scaling function over the simpler \emph{length-linear} $\sigma_L(r)=Lr$ scaling function are identified, and some non-intuitive properties of scaling functions are described that arise from the space $\HypDim$.

\paragraph{Critical Behaviour.}
    
    Let $\orig$ designate the (arbitrary) origin of $\HypDim$ -- it will sometimes also be convenient to use the notation $\origin{a}=\left(\orig,a\right)\in\HypDim\times \Ecal$. For $\mathbf{x}\in\eta$ let $\C\left(\mathbf{x},\xi\right)$ denote the set of vertices in $\eta$ that are connected to $\mathbf{x}$ in $\xi$, and naturally $\#\C\left(\mathbf{x},\xi\right)$ then denotes the size of this set under the counting measure. The notation $\xi^\mathbf{x}$ indicates that $\xi$ has been augmented by a vertex at $\mathbf{x}$, and since $\eta$ is distributed as a Poisson point process this is equivalent to conditioning on $\mathbf{x}\in\eta$. More details of this procedure can be found in Section~\ref{sec:prelims}.

    For all $\lambda>0$ and $L>0$, we define the \emph{susceptibility} function $\chi_{\lambda,L}\colon \Ecal \to \left[0,\infty\right]$ and \emph{percolation probability} function $\theta_{\lambda,L}\colon \Ecal\to \left[0,1\right]$ by
\begin{align}
    \chi_{\lambda,L}(a)&:= \E_{\lambda,L}\left[\#\C\left(\origin{a},\xi^{\origin{a}}\right)\right],\\
    \theta_{\lambda,L}(a) &:= \mathbb{P}_{\lambda,L}\left(\#\C\left(\origin{a},\xi^{\origin{a}}\right) = \infty\right),
\end{align}
    respectively. We then define the \emph{susceptibility critical intensity} and \emph{percolation critical intensity} using the susceptibility and percolation probability respectively:
    \begin{align}
        \lambda_T(L) :=& \inf\left\{\lambda>0\colon \esssup_{a\in\Ecal}\chi_{\lambda,L}(a)=\infty\right\}, \label{eqn:criticalSusceptibility} \\
        \lambda_\mathrm{c}(L) :=& \inf\left\{\lambda>0\colon \esssup_{a\in\Ecal}\theta_{\lambda,L}(a)>0\right\}.
    \end{align}
    If we fix $L>0$, then we say that the critical exponents $\gamma=\gamma(L)$ and $\beta=\beta(L)$ exist in the bounded ratio sense if there exist $\lambda$-independent constants $C_1,C_2\in\left(0,\infty\right)$ and $\varepsilon>0$ such that
    \begin{equation}
        C_1\left(\lambda_{T}(L)-\lambda\right)^{-\gamma} \leq \norm*{\chi_{\lambda,L}}_p \leq C_2\left(\lambda_{T}(L)-\lambda\right)^{-\gamma}
    \end{equation}
    for all $\lambda\in\left(\lambda_{T}(L)-\varepsilon,\lambda_{T}(L)\right)$ and $p\in\left[1,\infty\right]$, and 
    \begin{equation}
        C_1\left(\lambda-\lambda_{\mathrm{c}}(L)\right)^\beta \leq \norm*{\theta_{\lambda,L}}_p\leq C_2\left(\lambda - \lambda_{\mathrm{c}}(L)\right)^\beta
    \end{equation}
    for all $\lambda\in\left(\lambda_{\mathrm{c}}(L),\lambda_{\mathrm{c}}(L)+\varepsilon\right)$ and $p\in\left[1,\infty\right]$. In these, $\norm*{\cdot}_p$ denotes the $L^p$ norm with respect to the measure $\Pcal$ on $\Ecal$. We also say that the cluster tail critical exponent $\delta=\delta(L)$ exists in the bounded ratio sense if there exist constants $C_1,C_2\in\left(0,\infty\right)$ such that 
    \begin{equation}
        C_1 n^{-\frac{1}{\delta}} \leq \mathbb{P}_{\lambda_\mathrm{c}(L),L}\left(\#\C\left(\origin{a},\xi^{\origin{a}}\right) \geq  n\right) \leq C_2 n^{-\frac{1}{\delta}}
    \end{equation}
    for all $n\in\N$ and $\Pcal$-almost every $a\in\Ecal$.

    In addition to the percolation critical intensity, we will discuss the \emph{uniqueness critical intensity} defined by
\begin{equation}
    \lambda_{\mathrm{u}}(L) := \inf\left\{\lambda>0\colon \mathbb{P}_{\lambda,L}\left(\exists! \text{ infinite cluster}\right)>0\right\}.
\end{equation}
If $\lambda<\lambda_\mathrm{c}(L)$, then there are almost-surely no infinite clusters, and therefore $\lambda_\mathrm{u}(L)\geq \lambda_\mathrm{c}(L)$. We will be interested in finding whether this inequality is strict. Note that the invariance of both the vertex process and adjacency function under the isometries means that all the distributions $\mathbb{P}_{\lambda,L}$ are invariant under the isometries. Therefore by standard ergodicity arguments $\mathbb{P}_{\lambda,L}\left(\exists! \text{ infinite cluster}\right)\in\left\{0,1\right\}$, and so $\lambda_{\mathrm{u}}(L)= \inf\left\{\lambda>0\colon \exists! \text{ infinite cluster } \mathbb{P}_{\lambda,L}\text{-almost surely}\right\}$.

\subsection{General Results}
\label{sec:GeneralResults}

We first state the two main general theorems, before applying them to notable specific cases in Section~\ref{sec:SpecificModels}. One of these general theorems proves that a non-uniqueness phase exists, and the other proves that critical exponents take their mean-field values. These will be proven in two types of regime. In one regime we require that there are only finitely many marks, which allows us more freedom in the scaling function chosen. In the other regime we allow for infinitely many marks, but to avoid excessive complications we require that the scaling function takes the nice volume-scaling form.

\begin{assumptionp}{F}[Finitely Many Marks]
\label{assump:finitelymany}
    There are only finitely many marks, and there exists $L_0>0$ such that $L\geq L_0$ implies that 
    \begin{equation}
        \label{eqn:pointwiseIntegrable}
        \max_{a,b\in\Ecal}\int^\infty_0\connf_L\left(r;a,b\right)r\exp\left(\frac{1}{2}\left(d-1\right) r\right)\dd r<\infty.
    \end{equation}
    Also assume that
    \begin{equation}
    \label{eqn:EigenValueRatio}
        \lim_{L\to\infty}\frac{\max_{a,b\in\Ecal}\int^R_0\connf_L\left(r;a,b\right) \left(\sinh r\right)^{d-1}\dd r}{\max_{a,b\in\Ecal}\int^\infty_0\connf_L\left(r;a,b\right) \left(\sinh r\right)^{d-1}\dd r} = 0
    \end{equation}
    for all $R<\infty$.
\end{assumptionp}
Note that under Assumption~\ref{assump:finitelymany}, without loss of generality we assume that $\Pcal\left(a\right)>0$ for all $a\in\Ecal$. In Appendix~\ref{app:ScalingFunctions}, it is shown that \eqref{eqn:EigenValueRatio} holds for all adjacency functions if the scaling function is volume-linear or length-linear, and that it holds for all scaling functions for some classes of adjacency function. However, Lemma~\ref{lem:BadExample} demonstrates that \eqref{eqn:EigenValueRatio} does not hold for a specific choice of scaling and adjacency function. Therefore \eqref{eqn:EigenValueRatio} is indeed necessary in Assumption~\ref{assump:finitelymany} for this proof to work.

\begin{assumptionp}{S}[Volume-Linear Scaling]
\label{assump:specialscale}
    The scaling function is volume-linear, 
    \begin{equation}
        \label{eqn:pointwiseIntegrableVolumeLinear}
        \int^\infty_0\connf\left(r;a,b\right)r\exp\left(\frac{1}{2}\left(d-1\right) r\right)\dd r<
    \infty
    \end{equation}
    for $\Pcal$-almost every $a,b\in\Ecal$, and
    \begin{equation}
    \label{eqn:L2normVolumeScale}
        \sup_{f\in L^2\left(\Ecal\right), f\ne 0}\frac{\int_\Ecal\abs*{\int_\Ecal\left(\int^\infty_0\connf\left(r;a,b\right)r\exp\left(\frac{1}{2}\left(d-1\right) r\right)\dd r\right) f(b)\Pcal\left(\dd b\right)}^2\Pcal\left(\dd a\right)}{\int_{\Ecal}\abs*{f(a)}^2\Pcal\left(\dd a\right)}<\infty.
    \end{equation}
\end{assumptionp}

The condition \eqref{eqn:pointwiseIntegrableVolumeLinear} corresponds to \eqref{eqn:pointwiseIntegrable}, taking into account the fact that the volume-linear scaling ensures that if \eqref{eqn:pointwiseIntegrableVolumeLinear} is finite for one scaling $L$ then it is finite for any $L>0$. In the operator notation we introduce in Section~\ref{sec:Operators}, the conditions \eqref{eqn:pointwiseIntegrable} and \eqref{eqn:L2normVolumeScale} ensure that the operator norms $\norm*{\Opconnf_L}_{2\to 2}<\infty$ for $L$ sufficiently large. The choice of the volume-linear scaling also means that no condition corresponding to \eqref{eqn:EigenValueRatio} is required (see Lemma~\ref{lem:VolumeLinearNice}).

\begin{theorem}\label{thm:NonUniqueness}
    If Assumption~\ref{assump:finitelymany} or Assumption~\ref{assump:specialscale} holds, then for sufficiently large parameter $L$ we have $\lambda_{\mathrm{u}}(L) > \lambda_{\mathrm{c}}(L)$.
\end{theorem}

The restriction of Assumptions~\ref{assump:finitelymany} and \ref{assump:specialscale} to \ref{assump:finitelymanyPlus} and \ref{assump:specialscalePlus} respectively in Theorem~\ref{thm:meanfield} results from inheriting conditions from \cite{caicedo2023criticalexponentsmarkedrandom} (recalled exactly in Section~\ref{sec:ProofCritExponents}). For succinctness, we introduce some notation. Given $a,b\in\Ecal$, $k\geq 1$, and $L>0$, let us define
\begin{align}
    D_L(a,b) &:= \int_{\HypDim}\connf_L\left(x;a,b\right)\mu\left(\dd x\right),\label{eqn:DegreeFunction}\\
    D_L^{(k)}(a,b) &:= \int_{\Ecal^{k-1}}\left(\prod^k_{j=1} D_L(c_{j-1},c_j)\right)\prod^{k-1}_{i=1}\Pcal\left(\dd c_i\right),\label{eqn:kDegreeFunction}
\end{align}
where $c_0=a$ and $c_k=b$. The functions $D(a,b)$ and $D^{(k)}(a,b)$ (with no subscript $L$) are constructed in exactly the same way with the reference function $\connf$ in the place of $\connf_L$.

\begin{assumptionp}{F+}\label{assump:finitelymanyPlus}
In addition to Assumption~\ref{assump:finitelymany}, the following inequalities hold:
    \begin{align}
        \limsup_{L\to\infty}\frac{\max_{a,b\in\Ecal} D_L(a,b)}{\min_{a\in\Ecal}\sum_{b\in\Ecal} D_L(a,b)\Pcal\left(b\right)}&<\infty,\label{eqn:maxmaxOverminsum}\\
        \liminf_{L\to\infty}\sup_{k\geq 1}\frac{\min_{a,b\in\Ecal} D^{(k)}_L(a,b)}{\left(\max_{a\in\Ecal}\sum_{b\in\Ecal} D_L(a,b)\Pcal\left(b\right)\right)^k} &>0.\label{eqn:kstepRatio}
    \end{align}
\end{assumptionp}

\begin{assumptionp}{S+}\label{assump:specialscalePlus}
In addition to Assumption~\ref{assump:specialscale}, the following inequalities hold for the reference adjacency function:
    \begin{align}
        \esssup_{a,b\in\Ecal} D(a,b)&<\infty,\label{eqn:bounddegreedensity}\\
        \essinf_{a\in\Ecal}\int_\Ecal D(a,b)\Pcal\left(\dd b\right)&>0\label{eqn:infIntegralDegree},\\
        \esssup_{a\in\Ecal}\sup_{k\geq 1}\essinf_{b\in\Ecal}D^{(k)}(a,b)&>0.\label{eqn:strongIrreducibilityCondition}
    \end{align}
\end{assumptionp}

\begin{theorem}
\label{thm:meanfield}
    If Assumption~\ref{assump:finitelymanyPlus} or Assumption~\ref{assump:specialscalePlus} holds, then for all $L$ sufficiently large the critical exponents exist and $\gamma(L)=1$, $\beta(L)=1$, and $\delta(L)=2$. Furthermore, $\lambda_T(L)=\lambda_\mathrm{c}(L)$.
\end{theorem}

The additions of Assumption~\ref{assump:specialscalePlus} and Assumption~\ref{assump:finitelymanyPlus} serve the same role, in that they allow us to use results from \cite{caicedo2023criticalexponentsmarkedrandom} (summarized as Proposition~\ref{prop:summaryCD24} in this paper). 

\begin{remark}
    The requirement that $L$ be sufficiently large is very important for the proofs of both Theorem~\ref{thm:NonUniqueness} and \ref{thm:meanfield}. While the analogy with Bernoulli bond percolation on locally finite quasi-transitive nonamenable (and, say, Gromov hyperbolic) graphs suggests that the non-uniqueness phase should be non-empty for at least a large class of RCMs on $\HypDim\times\Ecal$, the general argument presented here relies heavily on the perturbative parameter $L$. That said, Proposition~\ref{prop:nonperturb} in Section~\ref{sec:SpecificModels} does show non-uniqueness for some models if one is willing to drop the requirement that the resulting graph is locally finite.
\end{remark}

\subsection{Applications to Specific Models}
\label{sec:SpecificModels}

We now see how Theorems~\ref{thm:NonUniqueness} and \ref{thm:meanfield} can apply to various notable examples of random connection models. We also see in Proposition~\ref{prop:nonperturb} a result that is not strictly a corollary of these theorems, but can be derived by applying similar techniques. In particular, it is non-perturbative unlike the general theorems.

\paragraph{Scaled Boolean Disc Model on $\HypDim$.}
    Let $\Ecal\subset\left(0,\infty\right)$ and $\Pcal$ be a probability measure on $\Ecal$, and for all $r\geq 0$ and $a,b\in\Ecal$ let
    \begin{equation}
        \connf\left(r;a,b\right)=\Id_{\left\{r<a+b\right\}}.
    \end{equation}
    
    The RCM with this reference adjacency function is a Boolean disc model in that it can be interpreted by placing balls with random radii with distribution $\Pcal$ centred on the points of a Poisson point process (intensity $\lambda$) on $\HypDim$. Vertices in the RCM are then adjacent if the associated two balls overlap. 

    We refer to the model with the scaled adjacency function $\connf_L$ as the \emph{scaled} Boolean disc model. The distinction is emphasised because in general the scaled Boolean disc model does not have the ``placing balls" interpretation. If the scaling function is length-linear (i.e. $\sigma_L(r)=Lr$) then the interpretation survives: if the original ball had radius $R$, then the scaled ball has radius $LR$ and vertices share an edge if and only if their balls intersect. If $\Ecal$ is a singleton (without loss of generality $\Ecal=\left\{1\right\}$), then the interpretation survives again and the balls all have radius $\frac{1}{2}\sigma_L(2)$. However, if the support of $\Pcal$ is at least two elements of $\left(0,\infty\right)$, and the scaling function is not linear in $r$, then radii cannot be assigned to each mark so that vertices are adjacent in the RCM if and only if the balls intersect.

    If one insisted on retaining the interpretation of ``placing balls" on the vertices, one could consider a model where $\connf_L(r;a,b) := \Id_{\left\{r<\sigma_L(a)+\sigma_L(b)\right\}}$, for some scaling function $\sigma_L$. However, if $\sigma_L(a)+\sigma_L(b)\ne \sigma_L(a+b)$  for some $a,b\in\Ecal$ for sufficiently large $L$, then this is not in our framework.

\begin{restatable}{corollary}{BooleanCorollary}
\label{lem:BooleanCorollary}
    Consider the scaled Boolean Disc model on $\HypDim$. 
    \begin{enumerate}[label=(\alph*)]
        \item \label{enum:BooleanCorPart0} If $R^*:=\sup\Ecal<\infty$ and $\Pcal\left(\left\{R^*\right\}\right)>0$, then for sufficiently large $L$ we have $\lambda_u(L) > \lambda_c(L)$.
        \item \label{enum:BooleanCorPart1} If $\sigma_L$ is volume-linear and
    \begin{equation}
    \label{eqn:FiniteExpectedVolume}
        \int^\infty_0r^2\exp\left(\left(d-1\right)r\right)\Pcal\left(\dd r\right)<\infty,
    \end{equation}
    then for sufficiently large $L$ we have $\lambda_u(L) > \lambda_c(L)$.
        \item \label{enum:BooleanCorPart2} If 
        \begin{equation}
            \label{eqn:InfiniteExpectedVolume}\int^\infty_0\exp\left(\left(d-1\right)r\right)\Pcal\left(\dd r\right)=\infty,
        \end{equation}
        then for all $L\geq 1$ almost every vertex is in the same infinite cluster, and $\lambda_{\mathrm{u}}(L)=\lambda_{\mathrm{c}}(L)=0$.
        \item \label{enum:BooleanCorPart3} If $\sigma_L$ is volume-linear and $R^*<\infty$, then for sufficiently large $L$ we also have  $\gamma(L)=1$, $\beta(L)=1$, $\delta(L)=2$, and $\lambda_T(L)=\lambda_{\mathrm{c}}(L)$. 
    \end{enumerate}
\end{restatable}

Observe that parts \ref{enum:BooleanCorPart0} and \ref{enum:BooleanCorPart2} do not require a specific scaling function. Therefore these parts apply to models with the length-linear scaling function, on which the ``placing balls" interpretation survives.

Also note that apart from the extra factor of $r^2$ in \eqref{eqn:FiniteExpectedVolume}, the volume-linear scaling choice, and the requirement that $L$ be sufficiently large, the condition \eqref{eqn:FiniteExpectedVolume} is the complement of \eqref{eqn:InfiniteExpectedVolume}. This leaves only a small space between the two regimes uncovered.

Corollary~\ref{lem:BooleanCorollary} parts \ref{enum:BooleanCorPart0} and \ref{enum:BooleanCorPart1} are both generalisations of the non-uniqueness result of \cite{tykesson2007number} in dimensions $d\geq 3$ -- their result in these dimensions is recovered by taking $\Ecal=\left\{1\right\}$, for example. It does not recover their result on $\HypTwo$ because the corollary here still requires the perturbative ``sufficiently large $L$" condition for that part.

For Boolean disc models on $\R^d$, the result of \cite{chebunin2024uniqueness} proves that there is at most one infinite cluster, and \cite{hall1985continuum} proved that almost every point in $\R^d$ is covered by one of the balls if the radius distribution has infinite $d$ moment. The values of the critical exponents are less well understood. If the radius distribution has a finite $5d-3$ moment, then \cite{DumRaoTas20} proves a lower bound for the percolation function that corresponds to the exponent $\gamma=1$, and \cite{DewMui} extends this to lower bounds on the susceptibility function and critical cluster tail probabilities that correspond to the mean-field exponents $\beta=1$ and $\delta=2$. The $5d-3$ moment is not expected to be `true' and is instead a consequence of the technique used. If we further restrict to radii distributions with \emph{bounded support}, \cite{DicHey2022triangle} shows that the triangle diagram is finite for sufficiently large $d$ (a similar approach should work for $d>6$ and a sufficiently `spread-out' radius distribution), and then \cite{caicedo2023criticalexponentsmarkedrandom} can be used to derive matching upper bounds (up to a factor of a constant) for the percolation and susceptibility functions and the critical cluster tail probabilities.

\paragraph{Weight-dependent hyperbolic RCMs.}
    Let $\Ecal=\left(0,1\right)$ with Lebesgue measure: $\Pcal=\Leb\left(0,1\right)$. Let the profile function $\rho\colon \R_+\to \left[0,1\right]$ be non-increasing and the kernel function $\kappa\colon \left(0,1\right)^2\to \left(0,\infty\right)$ be measurable and non-increasing in both arguments. We define a weight-dependent random connection model on $\HypDim$ by its having the adjacency function
    \begin{equation}
        \connf\left(r;a,b\right) = \rho\left(s^{-1}_{\kappa(a,b)}\left(r\right)\right)
    \end{equation}
    for all $r\geq 0$ and $a,b\in\left(0,1\right)$. This model has not been studied in this generality before, but is a natural analogy of the weight-dependent Euclidean RCMs discussed in literature such as \cite{KOMJATHY20201309,GraLucMor,GraHeyMonMor,HofstadHoornMaitra_2023,jorritsma2024large} (variously also called: \emph{general geometric inhomogeneous}, \emph{spatial inhomogeneous}, and \emph{kernel-based spatial} random graphs). In both the hyperbolic and Euclidean cases the form of the input to the profile function ensures that density of edges between marks is just given by the product of the mass of the profile function and a function of the kernel function evaluated on those marks. In our case, the use of the volume-linear scaling function means that the density of $a$-$b$ edges is given by
\begin{equation}
	\int_{\HypDim}\connf_L\left(\dist{x,\orig};a,b\right)\mu\left(\dd x\right) = L\kappa\left(a,b\right)\int_{\HypDim}\rho\left(\dist{x,\orig}\right)\mu\left(\dd x\right).
\end{equation}
    In particular this allows for the possibility that under appropriate kernel functions, the degree of a vertex with uniformly chosen mark asymptotically follows a Pareto distribution rather than a Poisson distribution as it would in the unmarked case. The homomorphism structure of $s_L$ (see \eqref{eqn:scalingHomomorphism}) means that the volume-linear scaling parameter $L$ is equivalent to a parameter multiplying the kernel by $L$. In the Euclidean literature this is often called the edge density parameter, and is denoted by the character $\beta$. In this paper $\beta$ is reserved for the percolation critical exponent.

    We now highlight five specific weight-dependent RCMs by specifying their respective kernel function. In each case the parameter $\zeta>0$, while $a\wedge b:= \min\left\{a,b\right\}$ and $a\vee b:= \max\left\{a,b\right\}$.
    \begin{itemize}
                \item The \emph{product kernel} is defined by
        \begin{equation}
        \label{eqn:prodkernel}
            \kappa^{\textrm{prod}}(a,b) := \left(ab\right)^{-\zeta}.
        \end{equation}
    
        \item The \emph{strong kernel} is defined by
        \begin{equation}
        \label{eqn:strongkernel}
            \kappa^{\textrm{strong}}(a,b) := \left(a\wedge b\right)^{-\zeta}.
        \end{equation}
        \item The \emph{sum kernel} is a related kernel function defined by
        \begin{equation}
            \kappa^{\textrm{sum}}(a,b) := a^{-\zeta} + b^{-\zeta}.
        \end{equation}
        The sum and strong kernels are related because $\kappa^{\textrm{strong}}\leq \kappa^{\textrm{sum}} \leq 2 \kappa^{\textrm{strong}}$, and therefore show qualitatively the same behaviour. Observe that in this interpretation of the weight-dependent hyperbolic RCM, the sum kernel does not produce a Boolean disc model (specifically the ``placing balls on each vertex" interpretation), because $s_L(r_1)+s_L(r_2)\ne s_L(r_1+r_2)$ in general. It is also a different model to the \emph{scaled} Boolean disc model described above, because the ``sum'' appears in the scaling function here, while it appeared in the reference adjacency function for the scaled Boolean disc model.

        \item The \emph{weak kernel} is defined by
        \begin{equation}
            \kappa^{\textrm{weak}}(a,b) := \left(a\vee b\right)^{-1-\zeta}.
        \end{equation}

        \item The \emph{preferential attachment kernel} is defined by
        \begin{equation}
        \label{eqn:prefattachkernel}
            \kappa^{\textrm{pa}}(a,b) := \left(a\vee b\right)^{-1+\zeta}\left(a\wedge b\right)^{-\zeta}.
        \end{equation}
    \end{itemize}
    In \cite{GraLucMor,GraHeyMonMor} the parameter $\zeta$ was denoted by $\gamma$, but this latter character will be reserved for the susceptibility critical exponent in this paper. As the following Corollaries will show, we cannot actually say anything about the behaviour of the weak and preferential attachment models, but we include them to show the limitations of the approach (which possibly suggests at a different behaviour for these models).

    In the following results, let $\Kcal$ denote the linear operator acting on $L^2\left(\Ecal\right)$ by
    \begin{equation}
        \left(\Kcal f\right)(a) = \int_{\Ecal}\kappa\left(a,b\right)f(b)\Pcal\left(\dd b\right),
    \end{equation}
    and let $\norm*{\Kcal}_{2\to 2}$ denote the operator norm:
    \begin{equation}
        \norm*{\Kcal}_{2\to 2} = \sup\left\{\frac{\norm{\Kcal f}_2}{\norm{f}_2}\colon f\in L^2(\Ecal),f\ne 0\right\}.
    \end{equation}

\begin{corollary}
\label{cor:WDRCMGeneral}
    Consider a weight-dependent random connection model on $\HypDim$ with volume-linear scaling. If
    \begin{equation}
    \label{eqn:WDRCMprofileCondition}
        \int^\infty_0\rho\left(r\right) r\exp\left(\frac{1}{2}\left(d-1\right)\right)\dd r <\infty
    \end{equation}
    and $\norm*{\Kcal}_{2\to 2}<\infty$, then for sufficiently large $L$ we have $\lambda_u(L) > \lambda_c(L)$.
\end{corollary}

\begin{corollary}
\label{cor:WDRCMSpecific}
    For the product, strong, and sum kernels,
    \begin{equation}
        \zeta<\frac{1}{2} \iff \norm*{\Kcal}_{2\to 2}<\infty.
    \end{equation}
    For the weak and preferential attachment kernels, $\norm*{\Kcal}_{2\to 2}=\infty$ for all $\zeta>0$.

    Therefore for weight-dependent random connection models on $\HypDim$ with product, strong, and sum kernels with parameter $\zeta<\frac{1}{2}$, if the scaling is volume-linear and the profile function satisfies \eqref{eqn:WDRCMprofileCondition} we have $\lambda_{\mathrm{u}}(L) > \lambda_{\mathrm{c}}(L)$ for $L$ sufficiently large.
\end{corollary}

While the following result is not strictly a corollary of the theorems above, by using the bound on the uniqueness threshold developed here we can produce a non-perturbative result for weight-dependent hyperbolic RCMs with moderately heavy-tailed profile functions.
\begin{restatable}{prop}{nonperturb}
\label{prop:nonperturb}
    Consider a weight-dependent random connection model on $\HypDim$. If
\begin{enumerate}[label=(\alph*)]
	\item  \begin{equation}\label{eqn:InfiniteDegree}
	    \int^\infty_0\rho(r)\exp\left(\left(d-1\right)r\right)\dd r=\infty,
	\end{equation}
	\item  \begin{equation}\label{eqn:ProfileL2bound}
	        \int^\infty_0\rho(r)r\exp\left(\frac{1}{2}\left(d-1\right)r\right)\dd r<\infty,
	\end{equation}
	\item and
 \begin{equation}
     \norm*{\Kcal}_{2\to 2}<\infty,
 \end{equation}
\end{enumerate}
 then $\lambda_T=\lambda_\mathrm{c}=0$ and $\lambda_\mathrm{u}>0$. Furthermore, $\theta_\lambda\left(a\right)=1$ for all $\lambda>0$ for a $\Pcal$-positive measure of $a\in\Ecal$, and therefore the critical exponent $\beta=0$. 
\end{restatable}
 If $\lambda_T=0$, then it does not make sense to talk about our definition of the critical exponent $\gamma$.

 \begin{remark}
     When talking about Bernoulli bond (or site) percolation it is usually assumed that the original graph is locally finite. Similarly, when talking about (marked or unmarked) RCMs there is usually an integral condition on $\connf$ that ensures that the degrees of each vertex are almost surely finite (and so the resulting graph is locally finite). One reason this is done is that without these conditions one immediately gets $\lambda_\mathrm{c} = 0$ (and $p_\mathrm{c}=0$). If one wishes to study behaviour at or around $\lambda_\mathrm{c}$, then many properties then become trivial. Proposition~\ref{prop:nonperturb} shows that uniqueness of the infinite cluster is not one of these properties. The condition \eqref{eqn:InfiniteDegree} ensures that every vertex almost surely has infinite degree (forcing $\lambda_\mathrm{c}=0$), but the proposition identifies conditions under which there are still infinitely many infinite clusters. That is, the resulting graph is not locally finite and yet there is still structure worth studying. 
 \end{remark}

\begin{remark}
    It is worth commenting on the difficulty that arises when trying to prove that the critical exponents take their mean-field values for the weight-dependent models. In each of the cases described by the kernels \eqref{eqn:prodkernel}-\eqref{eqn:prefattachkernel}, the expected degree of a vertex can be made arbitrarily large by taking the mark to be arbitrarily close to $0$. The argument presented here aims to show that the triangle diagram, $\triangle_\lambda$ (see Definition~\ref{defn:TriangleDiagram}), is small enough at criticality to use the result of \cite{caicedo2023criticalexponentsmarkedrandom} to prove that the critical exponents take their mean-field values. However, the presence of the $\esssup$ in the expression for $\triangle_\lambda$ means that $\triangle_\lambda=\infty$ for these weight-dependent models. If we were to expect that some types of weight-dependent random connection models could exhibit mean-field behaviour, then it is apparent that this particular form of the triangle condition is too strong. The presence of the $\esssup$ in Assumption~\ref{assump:specialscalePlus} has similar issues.
\end{remark}

\subsection{Background}
\label{sec:background}
 Bernoulli bond percolation on $\Z^d$ with nearest neighbour edges is the archetypal percolation model and naturally has a very large collection of literature. A standard ergodicity and trifurcation argument proves that there is almost surely at most one infinite connected cluster (see \cite{grimmett1989percolation}). Furthermore, by coupling with Bernoulli bond percolation on a Bethe lattice, \cite{AizNew84} proved that if the triangle condition holds then the susceptibility critical exponent exists and takes its mean-field value. Similarly, \cite{AizBar87,AizBar91} showed that other critical exponents take their mean-field values under the triangle condition. The matter of whether the triangle condition held was considered in \cite{HarSla90}, in which a lace expansion argument was used to show that the triangle condition holds for so-called `spread-out' models when $d>6$, and in the nearest neighbour model for $d>d^*$ for some $d^*\geq 6$. Currently the best upper bound for this upper critical dimension is $d^*\leq 10$ by \cite{FitHof17}.

The questions of Bernoulli bond percolation can also be asked on graphs other than the usual $\Z^d$. \cite{benjamini1996percolation} provides an early study of Bernoulli bond percolation on locally finite Cayley graphs, quasi-transitive graphs, and planar graphs. In particular, they conjectured that for Bernoulli bond percolation on locally finite nonamenable quasi-transitive graphs there is a non-empty interval of bond probabilities that produces infinitely many infinite clusters almost surely. For nonamenable finitely generated groups, \cite{pak2000non} proved every nonamenable group has a Cayley graph such that the associated Bernoulli bond percolation model has a non-empty non-uniqueness phase, and \cite{nachmias2012non} proved the existence of this phase if the Cayley graph has sufficiently high girth (i.e. sufficiently long shortest cycle). Likewise,  \cite{benjamini2001percolation} and \cite{hutchcroft2019percolation} proved that a non-uniqueness phase exists for locally finite nonamenable transitive planar single-ended graphs and locally finite nonamenable quasi-transitive locally finite \emph{Gromov hyperbolic} graphs respectively. In particular, \cite{hutchcroft2019percolation} used an operator description of the two-point function and a geometric result that they called a `\emph{hyperbolic magic lemma}' that they derived from the so-called `\emph{magic lemma}' of \cite{benjamini2011recurrence} to prove this non-uniqueness. 

For percolation on finite graphs, the question of the uniqueness of the infinite component becomes a question of the relative sizes of the two largest components. In \cite{bollobas2007phase}, this question was answered for inhomogeneous random graphs, in which each edge in the complete graph of $n$ vertices is given a mark from a finite measure space and the edge between two vertices is open independently with a probability given by a kernel function (that varies with $n$ to close in on critical behaviour). Under the assumption that the kernel is ``irreducible" and an assumption that ensures that the expected number of edges is ``what it should be," they prove that the largest component is of order $n$ with high probability, while the second largest component is $o\left(n\right)$.

Regarding critical exponents, \cite{Schon01} showed that the triangle condition holds for locally finite graphs with sufficiently high Cheeger constant, and used this to prove that various mean-field critical exponents are attained. \cite{Schon02} was then able to do this for nonamenable transitive locally finite single-ended graphs \emph{without} resorting to the triangle condition, instead using the dual graph to find a separating barrier between clusters. For certain regular tessellations of two and three dimensional hyperbolic spaces, \cite{madras2010trees} adapted this approach by using the geometric result that a positive fraction of the vertices in a finite cluster will be on the boundary of the convex hull of the cluster. In addition to the non-uniqueness phase, \cite{nachmias2012non,hutchcroft2019percolation} both derived mean-field critical exponents in their respective regimes by bounding the triangle diagram. For \cite{hutchcroft2019percolation} this came naturally from their operator description.

For random connection models (RCMs) on $\Rd\times\Ecal$ many of the results that applied to Bernoulli bond percolation on $\Z^d$ can be adapted. \cite{MeeRoy96} considered unmarked RCMs with rotation invariant and decreasing adjacency functions, showing that there was no non-uniqueness regime. This was generalised to \emph{marked} RCMs (also removing the requirement of being rotation invariant and decreasing) in \cite{chebunin2024uniqueness}, by using the amenability of $\Rd$ to show that clusters in such models on $\Rd\times\Ecal$ are \emph{deletion tolerant}. Then under a natural irreducibility assumption they proved that this property is equivalent to having no non-uniqueness phase. 

The lace expansion argument was adapted by \cite{HeyHofLasMat19} to show that (an appropriate version) of the triangle condition holds for `spread-out' RCMs on $\Rd$ when $d>6$, and otherwise for $d>d^*$ for some $d^*\geq 6$. In contrast to the lattice version, no effort has been made to optimize $d^*$ in this context. They then demonstrated the utility of this triangle condition by using it to prove that the susceptibility critical exponent takes its mean-field value. For marked RCMs on $\Rd\times\Ecal$, the lace expansion argument was performed by \cite{DicHey2022triangle}, in which a version of the triangle condition was proven to hold for sufficiently high dimensions. This was then used by \cite{caicedo2023criticalexponentsmarkedrandom} to show that certain critical exponents take mean-field values in sufficiently high dimensions.

Let us now review RCMs on $\HypDim$. For the specific case of the Boolean Disc Model with a fixed radius (variously also called Poisson Boolean model, Gilbert disc model, etc.) on $\HypDim$, \cite{tykesson2007number} proved that there is a non-uniqueness phase if $d=2$ or if $d\geq 3$ and the radius of the discs is sufficiently large. This was achieved by finding infinite geodesic rays from a point that are contained in the vacant set of the vertices' discs. In \cite{dickson2024hyperbolicrandomconnectionmodels}, the geometric arguments of \cite{madras2010trees} have been adapted to show that a wide variety of RCMs on $\HypTwo$ and $\mathbb{H}^3$ exhibit mean-field critical exponents. Like for \cite{madras2010trees}, the argument could not be made to work for $d\geq 4$.

The study of RCMs on $\HypDim$ (or $\HypDim\times\Ecal$) is not only relevant for itself, but may also have applications for RCMs on $\Rd\times\Ecal$. It is described in \cite{Hofstad2024RGCNV2} how a (Boolean disc) RCM on $\HypTwo$ can be interpreted as a one-dimensional product \emph{geometric inhomogeneous random graph} (a form of RCM on $\R\times\Ecal$). Hyperbolic embeddings of complex networks have also attracted much attention in themselves: \cite{boguna2010sustaining} models the Internet by embedding it in $\HypTwo$. 

The novelty in this work is that not only do we consider marked hyperbolic RCMs (i.e. RCMs on the ambient space $\HypDim\times\Ecal$) for any $d\geq 2$, but we show non-uniqueness (Theorem~\ref{thm:NonUniqueness}) and mean-field critical exponents (Theorem~\ref{thm:meanfield}) by using spherical transforms -- a technique with no obvious analog in the Bernoulli percolation or RCM literature. The ability of this approach to deal with models that almost surely produce non-locally finite graphs is also demonstrated in Proposition~\ref{prop:nonperturb}. The non-local finiteness of the resulting graph is emphasised here because all the references above are only concerned with models that produce locally finite graphs (by having the original graph being locally finite or by having an integrable adjacency function). This came from natural assumptions for them because they wished to analyse the percolation transition (and so needed it to be non-zero), but in our case we are studying a super-critical behaviour that makes sense even if the graph ends up not being locally finite. Note that Proposition~\ref{prop:nonperturb} also provides a \emph{non-perturbative} result - an advantage over the main Theorems~\ref{thm:NonUniqueness} and \ref{thm:meanfield}. Nevertheless these main theorems are perturbative like the works of \cite{pak2000non,Schon01,nachmias2012non}, \cite{tykesson2007number} for $d\geq 3$, and \cite{HarSla90,HeyHofLasMat19} in considering spread-out models.

\section{Preliminaries}
\label{sec:prelims}

Here are some extra details and remarks on some of the objects that appeared in the results.

\paragraph{Hyperbolic space.}

The (hyperbolic-)isometries of $\mathbb{B}$ can be characterized by M{\"o}bius transformations (that is, finite compositions of reflections). Every isometry of $\mathbb{B}$ extends to a unique M{\"o}bius transformation of $\mathbb{B}$, and every M{\"o}bius transformation restricts to an isometry of $\mathbb{B}$ (see \cite{ratcliffe1994foundations}). If we denote $\partial \mathbb{B}:= \left\{x\in\Rd\colon \abs*{x}=1\right\}$, then the orientation-preserving isometries can be classified by the number of fixed points of the associated M{\"o}bius transformation in $\mathbb{B}$ and in $\partial\mathbb{B}$. These isometries are sometimes called
\begin{itemize}
    \item \emph{rotations} if they fix a point in $\mathbb{B}$,
    \item \emph{horolations} if they fix no point in $\mathbb{B}$ but a unique point in $\partial\mathbb{B}$,
    \item \emph{translations} if they fix no point in $\mathbb{B}$ but fix two points in $\partial\mathbb{B}$.
\end{itemize}
These classes are often called `elliptic', `parabolic', and `hyperbolic' respectively, but we will prefer the former terms to avoid overloading the term `hyperbolic.' Note that while in $\Rd$ rotations cannot be expressed as a composition of translations, in $\HypDim$ all orientation-preserving isometries can be expressed as the composition of two translations (see for example \cite[Exercise~29.13]{martin2012foundations}).

The decision to make the edge probability only depend on the spatial distance between the vertices (see \eqref{eqn:EdgeProbs}) and not the full possibilities of the unordered pair $\left\{x,y\right\}$ is no great restriction. If we want our edge law to be \emph{spatially} homogeneous (i.e. there is no special spatial position), then we require that $\connf(x,a,y,b) = \connf(t(x),a,t(y),b)$ for all $x,y\in\HypDim$, $a,b\in \Ecal$, and translations $t$. However, since all rotations and horolations can be expressed as the product of two translations, we then know that $\connf(x,a,y,b) = \connf(i(x),a,i(y),b)$ for all $x,y\in\HypDim$, $a,b\in \Ecal$, and orientation-preserving isometry $i$. Therefore for all $r\geq 0$, $a,b\in\Ecal$, and $x,y\in\left\{u\in\HypDim\colon \dist{u,\orig}=r\right\}$, we have $\connf(x,a,\orig,b)=\connf(y,a,\orig,b)$ and the edge probabilities are of the form $\connf\left(\dist{x,y};a,b\right)$ in \eqref{eqn:EdgeProbs}.

The use of the hyperbolic measure $\mu$ ensures that the vertex distribution is invariant under the isometries, and since the positions of the vertices only influence the edge probability through their metric distance the distribution $\mathbb{P}_\lambda$ of the whole graph $\xi$ is also invariant under these isometries. Note that we have no such assumed symmetries in the marks other than the trivial transposition of the marks (which is required since the event $\mathbf{x}\sim \mathbf{y}$ is itself symmetric).

The symmetries of the hyperbolic spaces can also be viewed through identifying with the coset space
\begin{equation}
    SO\left(d,1\right) / O(d).
\end{equation}
Specifically $\HypDim$ can be realised as the symmetric space of the simple Lie group $SO(d,1)$. It is this description as a symmetric space that allows us to use spherical transforms (see \cite{helgason1994geometric}).

\paragraph{Connections and Augmentations.}
    Given two distinct vertices $\mathbf{x},\mathbf{y}\in\eta$, we say that $\mathbf{x}$ \emph{is connected to} $\mathbf{y}$ in $\xi$ (written ``$\conn{\mathbf{x}}{\mathbf{y}}{\xi}$") if there is a finite sequence of distinct vertices $\mathbf{x}=\mathbf{v}_0,\mathbf{v}_1,\ldots,\mathbf{v}_{k-1},\mathbf{v}_k=\mathbf{y}\in\eta$ such that $\mathbf{v}_i\sim \mathbf{v}_{i+1}$ for all $i=0,\ldots,k-1$.

    For $\mathbf{x}\in\HypDim\times\Ecal$, we use $\eta^\mathbf{x}$ and $\xi^{\mathbf{x}}$ to denote the vertex set and graph that have been augmented by $\mathbf{x}$. For $\eta^\mathbf{x}$ this simply means $\eta^\mathbf{x}:=\eta\cup\left\{\mathbf{x}\right\}$. For $\xi^\mathbf{x}$ this also includes the (random) edges between $\mathbf{x}$ and the vertices already present in $\eta$. This can naturally be generalized to produce $\eta^{\mathbf{x}_1,\ldots,\mathbf{x}_k}$ and $\xi^{\mathbf{x}_1,\ldots,\mathbf{x}_k}$ for any finite $k$. A more precise and complete description of this augmenting procedure can be found in \cite{HeyHofLasMat19}. This augmenting procedure allows us to define the \emph{two-point function} as
    \begin{equation}
        \tlam\left(\mathbf{x},\mathbf{y}\right):= \mathbb{P}_\lambda\left(\conn{\mathbf{x}}{\mathbf{y}}{\xi^{\mathbf{x},\mathbf{y}}}\right).
    \end{equation}
    Since $\pla$ is $\HypDim$-isometry invariant, so is $\tlam$. Therefore we will often write $\tlam\left(\left(x,a\right),\left(y,b\right)\right) = \tlam\left(\dist{x,y};a,b\right)$ and treat $\tlam$ as a $\R_+\times\Ecal^2\to\left[0,1\right]$ function like $\connf$. When we are talking about the two-point function of a model using a scaled adjacency function, $\connf_L$, we will denote it by $\tau_{\lambda,L}$. Note that this does \emph{not} mean that $\tau_{\lambda,L}$ is just $\tlam$ with the distance scaled by some $\sigma_L$, though.

\paragraph{Operators}
Let $\left(\X,\mathcal{X}\right)$ be a Borel space with associated non-atomic $\sigma$-finite measure $\nu$. Let $\connf\colon \X\times\X\to\left[0,1\right]$ be $\Xcal\times\Xcal$-measurable, and let $\pla$ denote the law of the RCM on $\X$ with intensity measure $\lambda\nu$ and adjacency function $\connf$.

Let $\psi\colon \X\times\X\to\R_+$ be $\Xcal\times\Xcal$-measurable. Then define $\Dcal(\psi)\subset \R^\X$ to be the set of $f\in\R^\X$ such that $\int_{\X}\psi(\mathbf{x},\mathbf{y})\abs{f(\mathbf{y})}\nu(\dd \mathbf{y})<\infty$ for $\nu$-almost every $\mathbf{x}\in\X$, so that
\begin{equation}
\label{eqn:OperatorDefn}
    \left(\Psi f\right)(\mathbf{x}) := \int_{\X}\psi(\mathbf{x},\mathbf{y})f(\mathbf{y})\nu\left(\dd \mathbf{y}\right)
\end{equation}
defines a linear operator $\Psi\colon \Dcal(\psi)\to \R^\X$. For each $p,q\in\left[1,\infty\right]$, the $L^p\to L^q$ operator norm of $\Psi$ is defined by
\begin{equation}
\label{eqn:OperatorNormDefn}
    \norm{\Psi}_{p\to q} := \begin{cases}
        +\infty &\text{if }L^p(\X)\not\subset \Dcal(\psi)\\
        \sup\left\{\frac{\norm{\Psi f}_q}{\norm{f}_p}\colon f\in L^p(\X),f\ne 0\right\} &\text{otherwise.}
    \end{cases}
\end{equation}

Let $\Optlam$ and $\Opconnf$ be the linear operators constructed in this way from $\tlam$ and $\connf$ respectively. We can define the operator critical thresholds to be
\begin{equation}
    \lambda_{p\to q} := \inf\left\{\lambda>0\colon \norm{\Optlam}_{p\to q} = \infty\right\}
\end{equation}
for each $p,q\in\left[1,\infty\right]$. In Lemma~\ref{lem:qtoqintensity_bound}, these are bounded by comparing the two-point function $\tlam$ to the Green's function of a branching process with offspring density $\lambda\connf$:
\begin{equation}
    \Greenlam(\mathbf{x},\mathbf{y}) := \sum^\infty_{n=1}\lambda^{n-1}\connf^{\star n}\left(\mathbf{x},\mathbf{y}\right),
\end{equation}
where $\star$ denotes convolution. That is, given functions $f_1,f_2\colon \X\times\X\to \R_+$ define
\begin{equation}
    f_1\star f_2(\mathbf{x},\mathbf{y}) := \int_\X f_1(\mathbf{x},\mathbf{u})f_2(\mathbf{u},\mathbf{y})\nu\left(\dd \mathbf{u}\right),
\end{equation}
and similarly $f^{\star n}_1\left(\mathbf{x},\mathbf{y}\right)$ by performing this iteratively. Also observe that while $\Opconnf^{n}$ is defined as the iterated application of the operator $\Opconnf$, it is also equal to the operator constructed from the function $\connf^{\star n}$.

\section{Operator Bounds}
\label{sec:Operators}

In this section we will prove the lemmata that do not depend on the hyperbolic structure of the spatial component of the ambient space. The scaling $\sigma_L$ will also not be relevant in this section, so $L$ is omitted from the notation.

Recall $\left(\X,\mathcal{X}\right)$ denotes a Borel space with associated non-atomic $\sigma$-finite measure $\nu$.

\begin{lemma}
\label{lem:qtoqintensity_bound}
For all $p\in\left[1,\infty\right]$,
    \begin{equation}
        \lambda_{p\to p} \geq \frac{1}{\norm{\Opconnf}_{p\to p}}.
    \end{equation}
\end{lemma}

\begin{proof}
    We first note the relation that for $\lambda>0$ and $\nu$-almost all $\mathbf{x},\mathbf{y}\in\X$,
    \begin{equation}\label{eqn:twopointVsGreen}
        \tlam\left(\mathbf{x},\mathbf{y}\right) \leq \Greenlam\left(\mathbf{x},\mathbf{y}\right).
    \end{equation}
    This was established in the proof of \cite[Lemma~2.2]{caicedo2023criticalexponentsmarkedrandom} using a `method of generations' approach. The neighbours of the starting vertex $\mathbf{y}$ are distributed as a Poisson point process with intensity measure $\lambda\connf\left(\mathbf{x},\mathbf{y}\right)\nu\left(\dd \mathbf{x}\right)$. These vertices are the `$1^{st}$ generation.' The `$2^{nd}$ generation' are the vertices that are adjacent to the $1^{st}$ generation, but not adjacent to $\mathbf{y}$. These are also distributed according to a Poisson point process, but the intensity measure is now a thinned measure whose density with respect to $\nu$ is less than or equal to $\lambda^2\connf^{\star 2}\left(\mathbf{x},\mathbf{y}\right)$. This can then be repeated for any generation. By Mecke's formula (Lemma~\ref{lem:Mecke} below) and monotone convergence, we then find that for any measurable $B\subset \X$,
    \begin{equation}
        \lambda\int_B\tlam\left(\mathbf{x},\mathbf{y}\right)\nu\left(\dd \mathbf{x}\right) = \E_\lambda\left[\#\left\{\mathbf{u}\in\eta\cap B\colon \conn{\mathbf{y}}{\mathbf{u}}{\xi^{\mathbf{y}}}\right\}\right] \leq \sum^\infty_{n=1}\lambda^n\int_B\connf^{\star n}\left(\mathbf{x},\mathbf{y}\right)\nu\left(\dd \mathbf{x}\right).
    \end{equation}
    The relation \eqref{eqn:twopointVsGreen} then follows.

    To arrive at the result, we make the elementary observation that if $\psi_1\geq\psi_2\geq 0$ $\nu$-everywhere, then their corresponding operators follow the partial ordering $\norm*{\Psi_1}_{p\to p}\geq \norm*{\Psi_2}_{p\to p}$ for every $p\in\left[1,+\infty\right]$. Then \eqref{eqn:twopointVsGreen}, the triangle inequality, and the sub-multiplicativity of the operator norm imply that
    \begin{equation}
    \label{eqn:TwoPointBoundGreen}
        \norm{\Optlam}_{p\to p} \leq \norm{\OpGreenlam}_{p\to p}  \leq \sum^{\infty}_{n=1}\lambda^{n-1}\norm{\Opconnf^n}_{p\to p} \leq \sum^{\infty}_{n=1}\lambda^{n-1}\norm{\Opconnf}_{p\to p}^n.
    \end{equation}
    This upper bound is finite if $\lambda\norm{\Opconnf}_{p\to p}<1$, and so the result follows.
\end{proof}

\begin{lemma}[Mecke's Formula]\label{lem:Mecke}
    Given $m\in\N$ and a measurable non-negative $f$,
\begin{equation}
    \E_\lambda \left[ \sum_{\vec {\mathbf{x}} \in \eta^{(m)}} f(\xi, \vec{\mathbf{x}})\right] = \lambda^m \int
				\E_\lambda\left[ f\left(\xi^{\mathbf{x}_1, \ldots, \mathbf{x}_m}, \vec {\mathbf{x}}\right)\right] \nu^{\otimes m}\left(\dd \vec{\mathbf{x}}\right),  \label{eq:prelim:mecke_n}
\end{equation}
where $\vec{\mathbf{x}}=(\mathbf{x}_1,\ldots,\mathbf{x}_m)$, $\eta^{(m)}=\{(\mathbf{x}_1,\ldots,\mathbf{x}_m)\colon x_i \in \eta, \mathbf{x}_i \neq \mathbf{x}_j \text{ for } i \neq j\}$, and $\nu^{\otimes m}$ is the $m$-product measure of $\nu$ on $\X^m$.
\end{lemma}
\begin{proof}
    A proof and discussion of Mecke's formula can be found in \cite[Chapter~4]{LasPen17}.
\end{proof}

For the following lemma, we can generalize the definition of $\lambda_T$ in \eqref{eqn:criticalSusceptibility} to our more general space. For $\mathbf{x}\in\X$,
\begin{align}
    \chi_{\lambda}(\mathbf{x})&= \E_{\lambda}\left[\#\C\left(\mathbf{x},\xi^{\mathbf{x}}\right)\right], \label{eqn:susceptibilityGeneral}\\
    \lambda_T &= \inf\left\{\lambda>0\colon \esssup_{\mathbf{x}\in\X}\chi_{\lambda}(\mathbf{x})=\infty\right\}.
\end{align}
In the case $\X=\HypDim\times\Ecal$, the transitive symmetry of the model in $\HypDim$ ensures that we recover the previous definition. The norm $\HSNorm{\cdot}$ in the following lemma is called the \emph{Hilbert-Schmidt} norm.

\begin{lemma}
\label{lem:generalOperatorBounds}
    For all linear operators $\Psi\colon \Dcal(\psi)\to \R^\X$ defined as in \eqref{eqn:OperatorDefn},
    \begin{align}
        \norm*{\Psi}_{1\to 1} &= \esssup_{\mathbf{y}\in\X}\int_\X\abs*{\psi\left(\mathbf{x},\mathbf{y}\right)}\nu\left(\dd \mathbf{x}\right),\\
        \norm*{\Psi}_{2\to 2} &\leq \HSNorm{\Psi}:= \left(\int_\X\int_\X\abs*{\psi\left(\mathbf{x},\mathbf{y}\right)}^2\nu\left(\dd \mathbf{x}\right)\nu\left(\dd \mathbf{y}\right)\right)^\frac{1}{2}.
    \end{align}
    In particular, this means
    \begin{equation}
        \lambda_T = \lambda_{1\to 1}.
    \end{equation}
\end{lemma}
\begin{proof}
    The first two are standard results. The first follows immediately from considering test functions approximating delta functions, while the second follows by applying a Cauchy-Schwarz inequality to $\norm*{\Psi f}_{2}$ in \eqref{eqn:OperatorNormDefn}.

    The equality of the critical intensities then follows from Mecke's formula:
    \begin{multline}
        \esssup_{\mathbf{y}\in\X}\chi_{\lambda}(\mathbf{y}) = 1+ \esssup_{\mathbf{y}\in\X}\E_{\lambda}\left[\sum_{\mathbf{x}\in\eta}\Id_{\left\{\conn{\mathbf{x}}{\mathbf{y}}{\xi^{\mathbf{y}}}\right\}}\right] \\= 1+ \lambda\esssup_{\mathbf{y}\in\X}\int_\X\tlam\left(\mathbf{x},\mathbf{y}\right)\nu\left(\dd \mathbf{x}\right) = 1+ \lambda\norm*{\Optlam}_{1\to 1}.
    \end{multline}
\end{proof}

A crucial object in our derivation of mean-field critical exponents is the triangle diagram.
\begin{definition}
\label{defn:TriangleDiagram}
    For $\lambda\geq0$, the \emph{triangle diagram} is defined as
    \begin{equation}
        \trilam := \lambda^2\esssup_{\mathbf{x},\mathbf{y}\in\X}\tlam^{\star 3}\left(\mathbf{x},\mathbf{y}\right).
    \end{equation}
\end{definition}

\begin{lemma}
\label{lem:TriangleFinite}
    If $\esssup_{\mathbf{x}\in\X}\int_{\X}\connf\left(\mathbf{x},\mathbf{y}\right)^2\nu\left(\dd \mathbf{y}\right)<\infty$ and $\lambda<\lambda_{2\to 2}$, then $\trilam<\infty$.
\end{lemma}

\begin{proof}
    Here we want to bound $\trilam$ using the operator norm $\norm*{\Optlam}_{2\to 2}$, which we know to be finite because $\lambda<\lambda_{2\to 2}$. The issue is that while we can approximate the essential suprema in the definition of $\trilam$ using $\Optlam$ and suitable test functions, the $2$-norms of these test functions may blow up as we improve the approximation. We avoid this issue by using factors of $\connf$ to mollify this blow-up.

    In \cite[Lemma~F.1]{DicHey2022triangle} it was proven using Mecke's formula that 
    \begin{align}
        \tlam\left(\mathbf{x},\mathbf{y}\right) &\leq \connf\left(\mathbf{x},\mathbf{y}\right) + \lambda\connf\star\tlam\left(\mathbf{x},\mathbf{y}\right)\\
        \tlam\left(\mathbf{x},\mathbf{y}\right) &\leq \connf\left(\mathbf{x},\mathbf{y}\right) + \lambda\tlam\star \connf\left(\mathbf{x},\mathbf{y}\right)
    \end{align}
    both hold for all $\lambda>0$ and $\mathbf{x},\mathbf{y}\in\X$. By then applying these inequalities to the first and last factors in $\tlam^{\star 3}$, we get
    \begin{equation}
        \tlam^{\star 3}\left(\mathbf{x},\mathbf{y}\right) \leq \connf\star\tlam\star\connf\left(\mathbf{x},\mathbf{y}\right) + 2\lambda\connf\star\tlam^{\star 2}\star\connf\left(\mathbf{x},\mathbf{y}\right) + \lambda^2\connf\star\tlam^{\star 3}\star\connf\left(\mathbf{x},\mathbf{y}\right)
    \end{equation}
    for all $\lambda>0$ and $\mathbf{x},\mathbf{y}\in\X$.

    By applying \cite[Lemma~F.2]{DicHey2022triangle} to our case, we find that for all $\varepsilon>0$ and $M\in\EssIm{\tlam^{\star 3}}$ (i.e. the essential image), there exist $\nu$-positive and finite sets $E_1,E_2\subset \X$ such that
    \begin{equation}
        f_i:= \frac{1}{\nu\left(E_i\right)}\Id_{E_i}
    \end{equation}
    for $i=1,2$ satisfy
    \begin{equation}
        \abs*{M - \inner{f_1}{\Optlam^3f_2}}\leq \varepsilon.
    \end{equation}
    Here $\inner{f}{g}:=\int_\X f\left(\mathbf{x}\right)\overline{g\left(\mathbf{x}\right)}\nu\left(\dd \mathbf{x}\right)$ denotes the inner product of two functions $f,g\in L^2\left(\X\right)$.

    Now since $f_1,f_2\geq 0$ everywhere, we arrive at
    \begin{equation}
        \inner{f_1}{\Optlam^3f_2} \leq \inner{f_1}{\left(\Opconnf\Optlam\Opconnf + 2\lambda\Opconnf\Optlam^2\Opconnf + \lambda^2\Opconnf\Optlam^3\Opconnf\right)f_2}.
    \end{equation}
    Observe that if $\norm*{\Opconnf}_{2\to 2}=\infty$ then $\norm*{\Optlam}_{2\to 2}=\infty$ for all $\lambda>0$ and therefore $\lambda_{2\to 2}=0$. Therefore under our assumptions $\Opconnf$ is a bounded operator, and therefore the symmetry of $\connf$ implies that $\Opconnf\colon L^2\left(\X\right)\to L^2\left(\X\right)$ is self-adjoint. This self-adjointness, the Cauchy-Schwarz inequality, the triangle inequality, and the sub-multiplicativity of the operator norm lead to
    \begin{equation}
        \inner{f_1}{\Optlam^3f_2} \leq \left(\norm*{\Optlam}_{2\to 2} + 2\lambda\norm*{\Optlam}_{2\to 2}^2 + \lambda^2\norm*{\Optlam}_{2\to 2}^3\right)\norm*{\Opconnf f_1}_2 \norm*{\Opconnf f_2}_2.
    \end{equation}
    Then we can write for $i=1,2$ that
    \begin{multline}
        \norm*{\Opconnf f_i}_2^2 = \int_\X\int_\X\int_\X f_i(\mathbf{x})\connf(\mathbf{x},\mathbf{y})\connf(\mathbf{y},\mathbf{z})f_i(\mathbf{z})\nu(\dd \mathbf{z})\nu(\dd \mathbf{y})\nu(\dd \mathbf{x}) \\\leq \norm*{f_i}^2_1\esssup_{\mathbf{x},\mathbf{z}\in\X}\connf^{\star 2}(\mathbf{x},\mathbf{z}) \leq \esssup_{\mathbf{x}\in\X}\int_{\X}\connf(\mathbf{x},\mathbf{y})^2\nu\left(\dd \mathbf{y}\right)<\infty.
    \end{multline}
    Finally since $\lambda<\lambda_{2\to 2}$, we know that $\norm*{\Optlam}_{2\to 2}<\infty$ and therefore $\trilam<\infty$.
\end{proof}

\section{Spherical Transform}
\label{sec:sphericaltransform}

We now restrict our attention to $\X=\HypDim$, and consider transforms of radial functions. Since we will be requiring this rotational symmetry, it will be convenient to write explicit expressions in terms of the Poincar{\`e} disc model. Letting
\begin{equation}
    \mathbb{B} := \left\{z\in\Rd\colon \abs*{z}<1\right\},
\end{equation}
the hyperbolic volume element on $\mathbb{B}$ is given by
\begin{equation}
    \mathbf{d}z = \frac{4}{\left(1-\abs*{z}^2\right)^2}\dd z.
\end{equation}
We also let $\mathbb{S}^{d-1} = \left\{z\in\Rd\colon \abs*{z}=1\right\}$ denote the Euclidean sphere.

\begin{definition}
Let $C^\infty_c\left(\mathbb{B}\right)$ denote the set of smooth functions on $\mathbb{B}$ with compact support (with respect to the hyperbolic metric), and $C^\infty_{c,\natural}\left(\mathbb{B}\right)$ denote the subset of such functions that are symmetric with respect to rotations about the origin $\orig$.

Given a radial function $g\colon \mathbb{B}\to \R_+$, converting into polar coordinates (with a minor abuse of notation) gives the expression
\begin{equation}
    \int_\mathbb{B} g(z)\mathbf{d}z = \mathfrak{S}_{d-1}\int^1_0 g(\varrho)\frac{4\varrho^{d-1}}{\left(1-\varrho^2\right)^2}\dd \varrho,
\end{equation}
where $\mathfrak{S}_{d-1}$ denotes the (Lebesgue) $\left(d-1\right)$-volume of the unit-radius sphere ${\mathbb{S}^{d-1}}$.

For $g\in C^\infty_{c}\left(\mathbb{B}\right)$, $s\in\R$ and $b\in {\mathbb{S}^{d-1}}$, we define the \emph{spherical transform} to be
    \begin{equation}
        \widetilde{g}(s,b) := \int_\mathbb{B}g(z)\e^{\left(-is+\frac{d-1}{2}\right)A(z,b)}\frac{4}{\left(1-\abs{z}^2\right)^2}\dd z,
    \end{equation}
    where
    \begin{equation}
        A(z,b) := \log \frac{1-\abs{z}^2}{\abs{z-b}^2},
    \end{equation}
    is the \emph{composite distance} between $\orig$ and the horocycle with normal $b$ and passing through $z$.

    \end{definition}
    \begin{remark}
    The spherical transform for $\HypTwo$ is given explicitly in \cite{helgason2000groups}. The general expression for symmetric spaces can be found in \cite{helgason1994geometric}, from which the expression of the $\HypDim$ spherical transform for $d\geq 3$ can be deduced.
\end{remark}
If $g\in C^\infty_{c,\natural}\left(\mathbb{B}\right)$, this expression for the spherical transform is $b$-invariant and reduces to
    \begin{align}
        \widetilde{g}(s) &:= \mathfrak{S}_{d-2}\int^1_0 g(\varrho)\frac{4\varrho^{d-1}}{\left(1-\varrho^2\right)^2}\left(\int^{\pi}_0\left(\frac{1-\varrho^2}{1+\varrho^2-2\varrho\cos\theta}\right)^{\frac{d-1}{2}-is}\left(\sin \theta\right)^{d-2}\dd \theta\right)\dd \varrho\\
        &= \mathfrak{S}_{d-1}\int^1_0 g(\varrho)Q^{\mathbb{B}}_d(\varrho;s)\frac{4\varrho^{d-1}}{\left(1-\varrho^2\right)^2}\dd \varrho,
    \end{align}
where, for $\varrho\in\left[0,1\right)$, $d\geq 2$ and $s\in\R$, we define
\begin{equation}
\label{eqn:Qd_Definition}
    Q^{\mathbb{B}}_d(\varrho;s):= \frac{\mathfrak{S}_{d-2}}{\mathfrak{S}_{d-1}}\int^{\pi}_0\left(\frac{1-\varrho^2}{1+\varrho^2-2\varrho\cos\theta}\right)^{\frac{d-1}{2}-is}\left(\sin \theta\right)^{d-2}\dd \theta.
\end{equation}
Note that for $s=0$ the integrand is real and strictly positive for all $\varrho\in\left[0,1\right)$ and $\theta\in\left(0,\pi\right)$, and therefore $Q^{\mathbb{B}}_d(\varrho;0)>0$ for all $\varrho\in\left[0,1\right)$. Furthermore, 
\begin{equation}
\label{eqn:QdatZero}
    Q^{\mathbb{B}}_d(0;s) = \frac{\mathfrak{S}_{d-2}}{\mathfrak{S}_{d-1}}\int^\pi_{0}\left(\sin \theta\right)^{d-2}\dd \theta = 1
\end{equation}
for all $d\geq 2$ and $s\in\R$.

It will sometimes be convenient to describe $Q^{\mathbb{B}}_d(\varrho;0)$ in terms of the hyperbolic distance rather than the Euclidean distance on $\mathbb{B}$. To this end, let us use \eqref{eqn:HypDistance} to define $Q_d\colon \R_+\to\R_+$ by
\begin{equation}
\label{eqn:QdFunctionDefinition}
    Q_d\left(r\right):= Q^{\mathbb{B}}_d\left(\tanh\frac{r}{2};0\right).
\end{equation}

While the above definition of $\widetilde{g}(s,b)$ and $\widetilde{g}(s)$ was given for $g\in C^\infty_{c}\left(\mathbb{B}\right)$ and $C^\infty_{c,\natural}\left(\mathbb{B}\right)$ functions respectively, these definition can naturally be extended to
\begin{align}
    \mathfrak{D}\left(\mathbb{B}\right) &:= \left\{g\colon \mathbb{B}\to \R \:\left\vert\: \int_\mathbb{B}\abs*{g(z)}\e^{\frac{d-1}{2}A(z,b)}\frac{1}{\left(1-\abs{z}^2\right)^2}\dd z<\infty\right.\right\}\\
    \mathfrak{D}_\natural\left(\mathbb{B}\right) &:= \left\{g\in\mathfrak{D}\left(\mathbb{B}\right) \:\left\vert\: g \:\text{is invariant under rotations about}\: \orig\right.\right\}
\end{align}
by using dominated convergence. An alternative expression for $\mathfrak{D}_\natural\left(\mathbb{B}\right)$ is 
\begin{multline}
\label{eqn:frakDnatural}
    \mathfrak{D}_\natural\left(\mathbb{B}\right) = \left\{g\colon \mathbb{B}\to \R \:\left\vert\:g \:\text{is invariant under rotations about}\: \orig, \right.\right.\\\left.\left.\:\text{and}\: \int^1_0\abs*{g\left(\varrho\right)}Q^\mathbb{B}_d\left(\varrho;0\right)\frac{\varrho^{d-1}}{\left(1-\varrho^2\right)^2}\dd \varrho<\infty\right.\right\}.
\end{multline}

\begin{lemma}
\label{lem:limitQd}
    For all $d\geq 2$, $\lim_{\varrho\nearrow 1}Q^{\mathbb{B}}_d(\varrho;0)=0$. In particular, there exist $c_d\in\left(0,\infty\right)$ such that
    \begin{equation}
        Q^{\mathbb{B}}_d(\varrho;0) \leq c_d\left(1-\varrho\right)^{\frac{d-1}{2}}\left(1\vee\abs*{\log\left(1-\varrho\right)}\right).
    \end{equation}
    Equivalently, there exist $c'_d\in\left(0,\infty\right)$ such that
    \begin{equation}
        Q_d(r) \leq
            c'_d \left(1\vee r\right) \exp\left(-\frac{1}{2}\left(d-1\right)r\right).
    \end{equation}
\end{lemma}
\begin{proof}
    First observe that by a double angle formula and appropriate substitutions,
    \begin{align}
        Q^{\mathbb{B}}_d(\varrho;0) & = \frac{\mathfrak{S}_{d-2}}{\mathfrak{S}_{d-1}}\left(\frac{1-\varrho^2}{1+\varrho^2}\right)^\frac{d-1}{2}\int^{\pi}_0\left(1-\frac{2\varrho}{1+\varrho^2}\cos\theta\right)^{-\frac{d-1}{2}}\left(\sin \theta\right)^{d-2}\dd \theta \nonumber\\
        & =\frac{\mathfrak{S}_{d-2}}{\mathfrak{S}_{d-1}}\left(\frac{1-\varrho^2}{\left(1-\varrho\right)^2}\right)^\frac{d-1}{2}\int^{\frac{\pi}{2}}_0\left(1+\frac{4\varrho}{\left(1-\varrho\right)^2}\left(\sin t\right)^2\right)^{-\frac{d-1}{2}}\left(\sin 2t\right)^{d-2}\dd t \nonumber\\
        & =\frac{\mathfrak{S}_{d-2}}{\mathfrak{S}_{d-1}}\left(\frac{1-\varrho^2}{\left(1+\varrho\right)^2}\right)^\frac{d-1}{2}\int^{\frac{\pi}{2}}_0\left(1-\frac{4\varrho}{\left(1+\varrho\right)^2}\left(\sin t\right)^2\right)^{-\frac{d-1}{2}}\left(\sin 2t\right)^{d-2}\dd t \nonumber\\
        & =\frac{\mathfrak{S}_{d-2}}{\mathfrak{S}_{d-1}}\left(\frac{1-\varrho}{1+\varrho}\right)^\frac{d-1}{2}J_d\left(\frac{4\varrho}{\left(1+\varrho\right)^2}\right),
    \end{align}
where
    \begin{equation}
        J_d\left(m\right) := \int^{\frac{\pi}{2}}_0\left(1-m\left(\sin t\right)^2\right)^{-\frac{d-1}{2}}\left(\sin 2t\right)^{d-2}\dd t.
    \end{equation}
    Note that $J_2\left(m\right)$ is the elliptic integral of the first kind. Also note that $J_3(m)$ can be evaluated by integration by substitution to get
    \begin{equation}
        J_3(m) = -\frac{2}{m}\log\left(1-m\right),
    \end{equation}
    and (since $\mathfrak{S}_1=2\pi$, $\mathfrak{S}_2=4\pi$ and $\varrho=\tanh\frac{r}{2}$) this produces
    \begin{equation}
        Q^{\mathbb{B}}_3\left(\varrho;0\right) = \frac{1-\varrho^2}{2\varrho}\log\frac{1+\varrho}{1-\varrho}, \qquad Q_3(r) = \frac{r}{\sinh r}.
    \end{equation}
    Similarly, by integration by substitution and long division it is possible to evaluate $Q^\mathbb{B}_d$ for any odd $d\geq 3$ using rational functions and logarithms. However because of the difficulty in getting generic expressions out of the long division step we will proceed by a different approach. For $d=2,3,4,5$, $Q^{\mathbb{B}}_d\left(\varrho;0\right)$ has been plotted in Figure~\ref{fig:QdFunctions2345}.

    \begin{figure}
    \centering
    \includegraphics[width=\columnwidth]{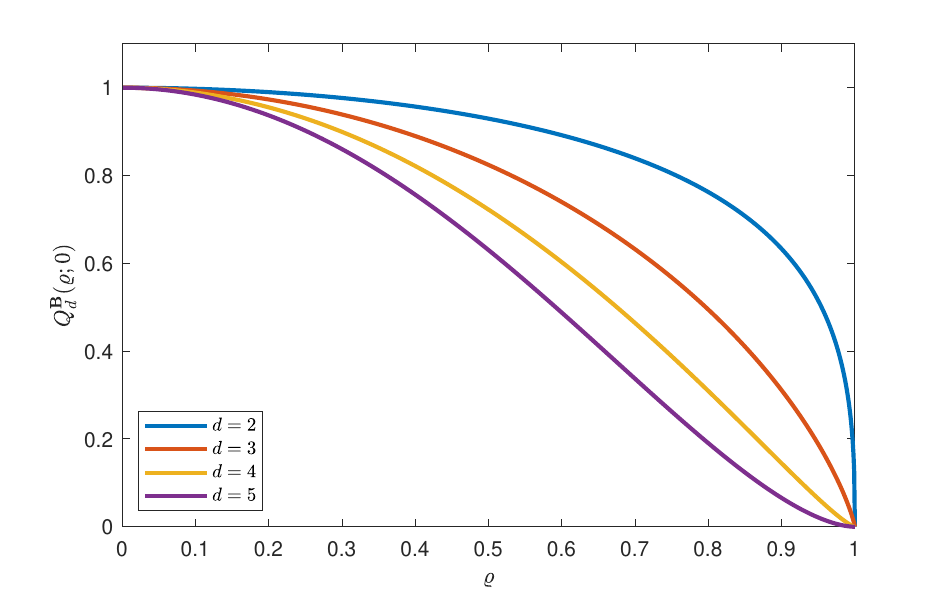}
    \caption{Plot of $Q^\mathbb{B}_d\left(\varrho;0\right)$ for $d=2,3,4,5$ using MATLAB.}
    \label{fig:QdFunctions2345}
\end{figure}

    For $\varepsilon\in\left(0,1\right)$, by using the substitution $y=1-\left(1-\varepsilon\right)\left(\sin t\right)^2$,
    \begin{multline}
        J_d\left(1-\varepsilon\right) = \int^{\frac{\pi}{2}}_0\left(1-\left(1-\varepsilon\right)\left(\sin t\right)^2\right)^{-\frac{d-1}{2}}\left(\sin 2t\right)^{d-2}\dd t \\
        = \left(1-\varepsilon\right)^{2-d}2^{d-3}\int^1_{\varepsilon}y^{-\frac{d-1}{2}}\left(1-y\right)^{\frac{d-3}{2}}\left(y-\varepsilon\right)^\frac{d-3}{2}\dd y.
    \end{multline}
    There exists a sequence $\left\{c_k\right\}_{k\in\N}$ such that $\left(y-\varepsilon\right)^{\frac{d-3}{2}}= \sum^\infty_{k=0}c_k\varepsilon^k y^{\frac{d-3}{2}-k}$, and only the integral resulting from the first of these terms will diverge as $\varepsilon\searrow 0$. Since $c_0=1$, 
    \begin{equation}
        J_d\left(1-\varepsilon\right) = \left(1-\varepsilon\right)^{2-d}2^{d-3}\int^1_{\varepsilon}y^{-1}\left(1-y\right)^{\frac{d-3}{2}}\dd y + \LandauBigO{1}
    \end{equation}
    as $\varepsilon\searrow 0$. Therefore there exists a constant $C_d\in\left(0,\infty\right)$ such that
    \begin{equation}
        J_d\left(1-\varepsilon\right) \sim -C_d \log \varepsilon
    \end{equation}
    as $\varepsilon\searrow 0$.

    Therefore there exists a constant $C'_d\in\left(0,\infty\right)$ such that
    \begin{equation}
        Q^{\mathbb{B}}_d(\varrho;0) \sim -C'_d\left(1-\varrho\right)^{\frac{d-1}{2}}\log\left(1-\varrho\right)
    \end{equation}
    as $\varrho\nearrow 1$. From the expression \eqref{eqn:Qd_Definition}, it is clear $Q^{\mathbb{B}}_d(\varrho;0)$ is continuous for $\varrho\in\left[0,1\right)$, and \eqref{eqn:QdatZero} then implies that the upper bound result holds.

    The bounds for $Q_d$ then follow from using \eqref{eqn:HypDistance} to get
    \begin{equation}
        1-\varrho \sim 2 \exp\left(-r\right).
    \end{equation}
\end{proof}

We now use the spherical transform to `partially diagonalize' $L^2\left(\HypDim\times\Ecal\times\Ecal\right)\to L^2\left(\HypDim\times\Ecal\times\Ecal\right)$ linear operators that are invariant under rotations about $\orig$ in $\HypDim$. The following lemma gives the three properties that we require from the spherical transform for our argument.
\begin{lemma}
\label{lem:diagonalisePlancherel}
    There exists a $d$-dependent real constant $w>0$ and a $d$-dependent function $\mathbf{c}\colon \R\to\Complex$ such that for $f\in \mathfrak{D}\left(\mathbb{B}\right)$, and $z\in\mathbb{B}$,
    \begin{equation}
        f(z) = \frac{1}{w} \int_\R\int_{{\mathbb{S}^{d-1}}}\e^{\left(+is+\frac{d-1}{2}\right)A(z,b)}\widetilde{f}(s,b)\abs{\mathbf{c}(s)}^{-2}\dd b\dd s.
    \end{equation}

    For $f_1\in \mathfrak{D}_\natural\left(\mathbb{B}\right)$ and $f_2\in \mathfrak{D}\left(\mathbb{B}\right)$,
    \begin{equation}\label{eqn:diagonalise}
        \left(f_1\star f_2\right)^{\sim}(s,b) = \widetilde{f_1}(s)\widetilde{f_2}(s,b).
    \end{equation}
    
    For $f_1,f_2\in \mathfrak{D}\left(\mathbb{B}\right)$ the following Plancherel result holds:
    \begin{equation}
        \int_{\mathbb{B}}f_1(z)f_2(z) \mathbf{d}z = \frac{1}{w}\int_\R\int_{{\mathbb{S}^{d-1}}}\widetilde{f_1}(s,b)\overline{\widetilde{f_2}(s,b)}\abs{\mathbf{c}(s)}^{-2}\dd b\dd s.
    \end{equation}
\end{lemma}

The constant $w$ is the order of the Weyl group and the function $\mathbf{c}(s)$ is known as the Harish-Chandra $\mathbf{c}$-function. We will not require expressions for these in our argument, but their definitions can be found in \cite[Chapter~1, p.75 \& p.108]{helgason1994geometric}.

\begin{proof}
    For $d=2$, this is proven in \cite[Intro.: Thm~4.2, Thm~4.6]{helgason2000groups}, and the $d\geq 3$ cases follow from the general symmetric space case dealt with in \cite[Ch.~III: Thm~1.3, Lem.~1.4, Thm~1.5]{helgason1994geometric}. Note that while \cite{helgason1994geometric,helgason2000groups} express these properties for functions $f_1,f_2$ in the space of continuous functions with compact (with respect to the $\HypDim$ metric) support, these can be extended to $\mathfrak{D}\left(\mathbb{B}\right)$ in the usual manner. Also be warned that these references define convolution in the reverse order to the convention we take here (convolution on $\HypDim$ is in general not commutative).
\end{proof}

In what follows, let $\psi\colon\left(\HypDim\times\Ecal\right)\times \left(\HypDim\times\Ecal\right)\to \left[0,1\right]$ be measurable and define the associated operator $\Psi\colon \Dcal\left(\psi\right)\to \R^{\HypDim\times \Ecal}$ in the sense of \eqref{eqn:OperatorDefn}. Also suppose that $\psi$ is isometry invariant in the $\HypDim$ arguments, so that $\psi\left(x,a,y,b\right)=\psi\left(\dist{x,y};a,b\right)$ for $x,y\in\HypDim$ and $a,b\in\Ecal$. Observe then that the spherical transform is defined for  $r\mapsto \psi\left(r;a,b\right)$ if and only if
\begin{equation}
    \label{eqn:SphericalTransformDefined}\int^\infty_0\psi\left(r;a,b\right)Q_d\left(r\right)\e^{\left(d-1\right)r}\dd r<\infty,
\end{equation}
where we have taken the integrability condition in \eqref{eqn:frakDnatural} and expressed an equivalent condition in hyperbolic distance coordinates (and using $\psi\in\left[0,1\right]$). If \eqref{eqn:SphericalTransformDefined} holds for $\Pcal$-almost every $a,b\in\Ecal$, then we can define an operator-valued function $s\mapsto\widetilde{\Psi}(s)$. For each $s\in\R$ we can define an operator $\widetilde{\Psi}(s)\colon \Dcal\left(\widetilde{\psi}(s;\cdot,\cdot)\right)\to \R^{\HypDim\times\Ecal}$ using the spherical transform-ed functions $\left(a,b\right)\mapsto\widetilde{\psi}(s;a,b)$.

\begin{lemma}
\label{lem:TwoToTwoNormBound}
If \eqref{eqn:SphericalTransformDefined} holds for $\Pcal$-almost every $a,b\in\Ecal$ and $\norm*{\widetilde{\Psi}\left(0\right)}_{2\to 2}<\infty$,
    \begin{equation}
        \norm*{\Psi}_{2\to 2} = \esssup_{s\in\R}\norm*{\widetilde{\Psi}\left(s\right)}_{2\to 2} = \norm*{\widetilde{\Psi}\left(0\right)}_{2\to 2}.
    \end{equation}
\end{lemma}

\begin{proof}
    The bound \eqref{eqn:SphericalTransformDefined} holding for $\Pcal$-almost every $a,b\in\Ecal$ indicates that the spherical transform is defined for these $a,b\in\Ecal$, and we can use Lemma~\ref{lem:diagonalisePlancherel} for those values. The first equality holds by the same argument as in \cite[Lemma~3.6]{DicHey2022triangle}. Approximately speaking, for each $s\in\R$ find a unit $g^{(s)}\in L^2\left(\Ecal\right)$ such that $\norm*{\widetilde{\Psi}(s) g^{(s)}}_2$ approximates the operator norm of $\widetilde{\Psi}(s)$. Then by suitably choosing some $\widetilde{f}\colon \R\to \Complex$ one can use the Plancherel and inverse transform parts of Lemma~\ref{lem:diagonalisePlancherel} to construct a unit $h\in L^2\left(\HypDim\times \Ecal\right)$ such that $\norm*{\Psi h}_2$ approximates some part of the essential spectrum of $\Psi$. This shows $\norm*{\Psi}_{2\to 2} \geq \esssup_{s\in\R}\norm*{\widetilde{\Psi}\left(s\right)}_{2\to 2}$. To show the reverse inequality, first take some unit $h\in L^2\left(\HypDim\times\Ecal\right)$ that approximates the operator norm of $\Psi$, and then for all $s\in\R$ define $g^{(s)}$ by fixing $a\in\Ecal$ and taking the spherical transform of $h(x,a)$ (i.e. $g^{(s)}(a) = \widetilde{h}(s,a)$). The convolution operator on $L^2\left(\Ecal\right)$ that uses $g^{(s)}$ can be `diagonalised' into a multiplication operator on some other space by a unitary transformation. This unitarity, and the Plancherel and multiplication parts of Lemma~\ref{lem:diagonalisePlancherel} then mean that $g^{(s)}$ is a unit vector in $L^2\left(\Ecal\right)$ such that $\norm*{\widetilde{\Psi}(s)g^{(s)}}_2$ approximates some part of the essential spectrum of $\widetilde{\Psi}(s)$.

    For the inequality $\esssup_{s\in\R}\norm*{\widetilde{\Psi}\left(s\right)}_{2\to 2} \leq  \norm*{\widetilde{\Psi}\left(0\right)}_{2\to 2}$, note that
    \begin{align}
        &\abs*{\widetilde{\psi}\left(s;a,b\right)}\nonumber\\
        &\hspace{1cm}=\abs*{\mathfrak{S}_{d-2}\int^1_0 \psi(\varrho;a,b)\frac{4\varrho^{d-1}}{\left(1-\varrho^2\right)^2}\left(\int^{\pi}_0\left(\frac{1-\varrho^2}{1+\varrho^2-2\varrho\cos\theta}\right)^{\frac{d-1}{2}-is}\left(\sin \theta\right)^{d-2}\dd \theta\right)\dd \varrho} \nonumber\\
        &\hspace{1cm}\leq \mathfrak{S}_{d-2}\int^1_0 \psi(\varrho;a,b)\frac{4\varrho^{d-1}}{\left(1-\varrho^2\right)^2}\left(\int^{\pi}_0\abs*{\left(\frac{1-\varrho^2}{1+\varrho^2-2\varrho\cos\theta}\right)^{\frac{d-1}{2}-is}\left(\sin \theta\right)^{d-2}}\dd \theta\right)\dd \varrho\nonumber\\
        &\hspace{1cm}= \mathfrak{S}_{d-2}\int^1_0 \psi(\varrho;a,b)\frac{4\varrho^{d-1}}{\left(1-\varrho^2\right)^2}\left(\int^{\pi}_0\left(\frac{1-\varrho^2}{1+\varrho^2-2\varrho\cos\theta}\right)^{\frac{d-1}{2}}\left(\sin \theta\right)^{d-2}\dd \theta\right)\dd \varrho\nonumber\\
        &\hspace{1cm}= \widetilde{\psi}\left(0;a,b\right)
    \end{align}
    uniformly in $a,b\in\Ecal$ and $s\in\R$. The inequality then holds because of the monotonicity of the operator norm.

    The last equality then uses $\norm*{\widetilde{\Psi}\left(0\right)}_{2\to 2}<\infty$ and a dominated convergence argument to show that $s\to \norm*{\widetilde{\Psi}\left(s\right)}_{2\to 2}$ is continuous at $s=0$.
\end{proof}

Recall the definition of $D_L(a,b)$ from \eqref{eqn:DegreeFunction}. We can use this function to define a linear operator on $\Dcal\left(D_L\right)\subset \R^\Ecal$ as in \eqref{eqn:OperatorDefn}, and let $\norm*{D_L}_{p\to q}$ denote the $L^p(\Ecal)\to L^q(\Ecal)$ operator norm of this operator.

\begin{lemma}
\label{lem:RatioofNorms}
Suppose Assumption~\ref{assump:finitelymany} or Assumption~\ref{assump:specialscale} holds.
Then
    \begin{equation}
        \lim_{L\to\infty}\frac{\norm*{\Opconnf_L}_{2\to 2}}{\norm*{D_L}_{2\to 2}} = 0.
    \end{equation}
\end{lemma}

\begin{proof}
    Let $R\in\left(0,1\right)$. Then Lemma~\ref{lem:limitQd} indicates that there exists a constant $C=C(d)<\infty$ such that $Q^{\mathbb{B}}_d\left(\varrho;0\right)\leq C$ and thus
    \begin{align}
        &\int^1_0 \connf_L(\varrho;a,b)Q^{\mathbb{B}}_d\left(\varrho;0\right)\frac{4\varrho^{d-1}}{\left(1-\varrho^2\right)^2}\dd \varrho \nonumber \\
        &\hspace{2cm}\leq C\int^R_0 \connf_L(\varrho;a,b)\frac{4\varrho^{d-1}}{\left(1-\varrho^2\right)^2}\dd \varrho + \int^1_R \connf_L(\varrho;a,b)\frac{4\varrho^{d-1}}{\left(1-\varrho^2\right)^2}\dd \varrho \esssup_{t>R}Q^{\mathbb{B}}_d(t;0)\nonumber\\
        &\hspace{2cm}\leq C\int^R_0 \connf_L(\varrho;a,b)\frac{4\varrho^{d-1}}{\left(1-\varrho^2\right)^2}\dd \varrho + D_L(a,b) \esssup_{t>R}Q^{\mathbb{B}}_d(t;0).
    \end{align}
    Now define 
    \begin{equation}
        D^{(\leq R)}_L\left(a,b\right):= \mathfrak{S}_{d-1}\int^R_0 \connf_L(\varrho;a,b)\frac{4\varrho^{d-1}}{\left(1-\varrho^2\right)^2}\dd \varrho
    \end{equation}
    and the associated linear operator $\Dcal\left(D^{(\leq R)}_L\right)\to \R^\Ecal$. Assumptions~\ref{assump:finitelymany} and \ref{assump:specialscale} each show that $\int^\infty_0\connf_L\left(r;a,b\right)Q_d\left(r\right)\e^{\left(d-1\right)r}\dd r<\infty$ for $\Pcal$-almost every $a,b\in\Ecal$ and $ \norm*{\widetilde{\Opconnf}_L(0)}_{2\to 2}<\infty$ in their regimes, and therefore we can apply Lemma~\ref{lem:TwoToTwoNormBound}, the triangle inequality, and the homogeneity of the norm to get
    \begin{equation}
        \norm*{\Opconnf_L}_{2\to 2} = \norm*{\widetilde{\Opconnf}_L(0)}_{2\to 2} \leq C\norm*{D^{(\leq R)}_L}_{2\to 2} + \norm*{D_L}_{2\to 2}\esssup_{t>R}Q^{\mathbb{B}}_d(t;0).
    \end{equation}

    Then under Assumption~\ref{assump:finitelymany}, the equivalence of norms and \eqref{eqn:EigenValueRatio}  (and translating into Poincar{\`e} disc coordinates) implies that there exists a sequence $\left\{C'_L\right\}_{L}$ such that $C'_L\to 0$ and $\norm*{D^{(\leq R)}_L}_{2\to 2}\leq C'_L \norm*{D_L}_{2\to 2}$. Under Assumption~\ref{assump:specialscale}, the scaling function is volume-linear and therefore
    \begin{align}
        \norm*{D^{(\leq R)}_L}_{2\to 2} &\leq L \mathfrak{S}_{d-1}\mathbf{V}_d\left(s^{-1}_L(R)\right) = \mathfrak{S}_{d-1}\mathbf{V}_d\left(R\right)\\
        \norm*{D_L}_{2\to 2} & = L\norm*{D}_{2\to 2}.
    \end{align}
    That is, under Assumption~\ref{assump:specialscale} there also exists a sequence $\left\{C'_L\right\}_{L}$ such that $C'_L\to 0$ and $\norm*{D^{(\leq R)}_L}_{2\to 2}\leq C'_L \norm*{D_L}_{2\to 2}$. Either way,
    \begin{equation}
        \norm*{\Opconnf_L}_{2\to 2} \leq\norm*{D_L}_{2\to 2}\left(C'_L+\esssup_{t>R}Q^{\mathbb{B}}_d(t;0)\right).
    \end{equation}

    Therefore
    \begin{equation}
        \limsup_{L\to\infty}\frac{\norm*{\Opconnf_L}_{2\to 2}}{\norm*{D_L}_{2\to 2}} \leq \esssup_{t>R}Q^{\mathbb{B}}_d(t;0).
    \end{equation}
    From Lemma~\ref{lem:limitQd}, this bound can be made arbitrarily small by taking $R\nearrow 1$ and the result follows.
\end{proof}

When we are working under Assumption~\ref{assump:specialscale}, Lemma~\ref{lem:RatioofNorms} doesn't give enough information when $\norm*{D}_{2\to 2}=\infty$, and we will need the following refinement.

\begin{lemma}
\label{lem:NormSublinear}
    If Assumption~\ref{assump:specialscale} holds, then $\norm*{\Opconnf_L}_{2\to 2} =o\left( L\right)$ as $L\to\infty$.
\end{lemma}

\begin{proof}
    Let $a,b\in\Ecal$. Then by a change of variables
    \begin{multline}
        \int^\infty_0\connf_L(r;a,b)Q_d(r)\left(\sinh r\right)^{d-1}\dd r = L \int^\infty_0\connf(r;a,b) Q_d\left(s_L\left(r\right)\right)\left(\sinh r\right)^{d-1}\dd r \\
         \leq L \int^\infty_0  \connf(r;a,b) f\left(s_L(r)\right)\left(\sinh r\right)^{d-1}\dd r,
    \end{multline}
    where $f(r):= \sup_{t\geq r}Q_d\left(t\right)$. $f$ is non-increasing and inherits the $r\to\infty$ asymptotics from those derived for $Q_d$ in the proof of Lemma~\ref{lem:limitQd}. In particular Assumption~\ref{assump:specialscale} implies that
    \begin{equation}
        \mathfrak{S}_{d-1}\int^\infty_0  \connf(r;a,b) f\left(r\right)\left(\sinh r\right)^{d-1}\dd r<\infty
    \end{equation}
    for almost every $a,b\in\Ecal$, and the operator with this kernel (denote this $\Opconnf^{(f)}$) has finite $L^2\left(\Ecal\right)\to L^2\left(\Ecal\right)$ operator norm.

    Now we explore the asymptotics of $s_L(r)$ as $r\to\infty$ for fixed $L$. Clearly as $r\to\infty$
    \begin{align}
        \mathbf{V}_d\left(r\right) &\sim \frac{1}{d-1}\frac{1}{2^{d-1}}\e^{\left(d-1\right)r},\\
        \mathbf{V}_d^{-1}\left(r\right) &\sim \log 2 + \frac{1}{d-1}\log\left(d-1\right) + \frac{1}{d-1}\log r,
    \end{align}
    and therefore
    \begin{equation}
        s_L(r)\sim r+ \frac{1}{d-1}\log L.
    \end{equation}
    Therefore from the proof of Lemma~\ref{lem:limitQd} we have
    \begin{equation}
    \label{eqn:largerratio}
        \frac{f\left(s_L\left(r\right)\right)}{f\left(r\right)} \sim \frac{s_L(r)}{r}\exp\left(\frac{1}{2}\left(d-1\right)\left(r - s_L(r)\right)\right) \sim \frac{1}{\sqrt{L}}
    \end{equation}
    as $r\to \infty$ for fixed $L$. On the other hand, if we fix $r>0$ and let $L\to \infty$ we get $s_L(r)\to\infty$ and therefore
    \begin{equation}
        \lim_{L\to\infty}\frac{f\left(s_L\left(r\right)\right)}{f\left(r\right)} = 0
    \end{equation}
    for all $r>0$. Furthermore, since $L\mapsto s_L(r)$ is increasing, $L\mapsto \tfrac{f\left(s_L\left(r\right)\right)}{f\left(r\right)}$ is non-increasing. For $L,\varepsilon>0$, define
    \begin{equation}
        R_{L,\varepsilon}:= \sup\left\{r>0\colon \frac{f\left(s_L\left(r\right)\right)}{f\left(r\right)} \geq \varepsilon\right\},
    \end{equation}
    so that the map $L\mapsto R_{L,\varepsilon}$ is non-increasing.

    Now given $\varepsilon>0$, fix $R_{\varepsilon}:= R_{4\varepsilon^{-2},\varepsilon}$ (\eqref{eqn:largerratio} implies $R_\varepsilon<\infty$). Therefore for all $a,b\in\Ecal$
    \begin{align}
        &\int^\infty_0\connf_L(r;a,b)Q_d(r)\left(\sinh r\right)^{d-1}\dd r\nonumber\\
        &\hspace{1cm}\leq L\int^{R_\varepsilon}_0 f\left(s_L(r)\right)\left(\sinh r\right)^{d-1}\dd r + L\int^\infty_{R_\varepsilon}\frac{f\left(s_L(r)\right)}{f(r)}\connf(r;a,b)f\left(r\right)\left(\sinh r\right)^{d-1}\dd r\nonumber\\
        &\hspace{1cm}\leq L\int^{R_\varepsilon}_0 f\left(s_L(r)\right)\left(\sinh r\right)^{d-1}\dd r + \varepsilon L\int^\infty_{0}\connf(r;a,b)f\left(r\right)\left(\sinh r\right)^{d-1}\dd r.
    \end{align}
    By Lemma~\ref{lem:TwoToTwoNormBound}, the triangle inequality, and the homogeneity of the operator norm,
    \begin{equation}
        \norm*{\Opconnf_L}_{2\to 2} \leq L\mathfrak{S}_{d-1}\int^{R_\varepsilon}_0 f\left(s_L(r)\right)\left(\sinh r\right)^{d-1}\dd r + \varepsilon L \norm*{\Opconnf^{(f)}}_{2\to 2}.
    \end{equation}
    Therefore 
    \begin{equation}
        \limsup_{L\to\infty} \frac{\norm*{\Opconnf_L}_{2\to 2}}{L} \leq \varepsilon\norm*{\Opconnf^{(f)}}_{2\to 2}.
    \end{equation}
    Since this norm on the right hand side is finite, taking $\varepsilon\searrow 0$ gives the result.

\end{proof}

\section{Proof of the Main Results}
\label{sec:MainProof}

In the first subsection, the non-uniqueness result Theorem~\ref{thm:NonUniqueness} is proven. Then Theorem~\ref{thm:meanfield} is proven in Section~\ref{sec:ProofCritExponents}.

\subsection{Proof of Non-Uniqueness}
\label{sec:NonUniqueness}

First we relate the uniqueness critical threshold to the $L^p\to L^p$ critical thresholds. The only special feature of $\HypDim$ that the following lemma uses is the existence of an infinite transitive family of transformations, such as the translations. To reflect this generality we generalize the definition of $\theta_\lambda$ like we did for $\chi_\lambda$ in \eqref{eqn:susceptibilityGeneral}: for $\mathbf{x}\in\HypDim\times \Ecal$
\begin{equation}
    \theta_\lambda\left(\mathbf{x}\right) = \mathbb{P}_{\lambda}\left(\#\C\left(\mathbf{x},\xi^{\mathbf{x}}\right) = \infty\right).
\end{equation}

\begin{lemma}
\label{lem:Uniquenesstoqtoq}
    For all $p\in\left[1,\infty\right]$ and $L>0$,
    \begin{equation}
        \lambda_u(L) \geq \lambda_{p\to p}(L).
    \end{equation}
\end{lemma}

\begin{proof}
    If there exists a unique infinite cluster in $\xi$, denote its vertex set by $\mathcal{I}\subset\eta$. Choose $\lambda>\lambda_{\mathrm{u}}$ such that $\pla\left(\exists! \text{ infinite cluster}\right)>0$. Then ergodicity implies that there is almost surely an infinite cluster and therefore
    \begin{equation}
        \tlam\left(\mathbf{x},\mathbf{y}\right)\geq \pla\left(\mathbf{x}\in\mathcal{I} \text{ in }\xi^{\mathbf{x}}, \mathbf{y}\in\mathcal{I} \text{ in }\xi^{\mathbf{y}}\right) \geq \theta_\lambda\left(\mathbf{x}\right)\theta_\lambda\left(\mathbf{y}\right),
    \end{equation}
    where we have used the FKG inequality (Lemma~\ref{lem:FKG} below). Since $\lambda>\lambda_u\geq \lambda_c$, there exist $\varepsilon>0$ and $E\subset \X$ such that $\theta_\lambda(\mathbf{x})\geq\varepsilon$ for all $\mathbf{x}\in E$, and $\nu\left(E\right)>0$. From the hyperbolic translation invariance, $\theta_\lambda(\mathbf{x})\geq \varepsilon$ on an infinite measure set and therefore $\norm{\theta_\lambda}_p=\infty$. Given $f\in L^p\left(\X\right)$ such that $f\geq 0$, we then have 
    \begin{equation}
        \int_\X \tlam(\mathbf{x},\mathbf{y})f(\mathbf{y})\nu\left(\dd \mathbf{y}\right) \geq \theta_\lambda(\mathbf{x})\int_\X\theta_\lambda(\mathbf{y})f(\mathbf{y})\nu\left(\dd \mathbf{y}\right),
    \end{equation}
    and
    \begin{equation}
        \norm{\Optlam f}_{p} \geq \norm{\theta_\lambda}_p\int_\X\theta_\lambda(\mathbf{y})f(\mathbf{y})\nu\left(\dd \mathbf{y}\right).
    \end{equation}

    Now let $f=\Id_E$. Then $\norm*{f}_p=\nu\left(E\right)^{\frac{1}{p}}$ and 
    \begin{equation}
        \int_\X\theta_\lambda(\mathbf{y})f(\mathbf{y})\nu\left(\dd \mathbf{y}\right) \geq \varepsilon\nu\left(E\right).
    \end{equation}
    Therefore
    \begin{equation}
        \norm*{\Optlam}_{p\to p} \geq \frac{\norm*{\Optlam f}_{p}}{\norm*{f}_p}\geq \norm*{\theta_\lambda}_{p}\varepsilon \nu\left(E\right)^{1-\frac{1}{p}} = \infty.
    \end{equation}
    Therefore $\lambda_{p\to p}\geq \lambda \geq \lambda_\mathrm{u}$.
\end{proof}

\begin{remark}
\label{rem:2to2isminimal}
    The following observation was made in an analogous case by \cite{hutchcroft2019percolation}, and applies equally well here. The symmetry of $\connf$ implies that $\norm{\Opconnf}_{p\to p} = \norm{\Opconnf}_{\frac{p}{p-1}\to\frac{p}{p-1}}$, and the Riesz--Thorin Theorem implies that $p\mapsto \log \norm{\Opconnf}_{\frac{1}{p}\to\frac{1}{p}}$ is a convex function on $p\in\left[0,1\right]$. Hence $p\mapsto\norm{\Opconnf}_{p\to p}$ is decreasing on $\left[1,2\right]$ and increasing on $\left[2,\infty\right]$. Hence the choice $p=2$ gives the optimal version of the bound in Lemma~\ref{lem:Uniquenesstoqtoq}.
\end{remark}

\begin{lemma}[FKG Inequality]
\label{lem:FKG}
    Given two increasing events $E,F$,
    \begin{equation}
        \pla\left(\xi\in E\cap F\right) \geq \pla\left(\xi\in E\right)\pla\left(\xi\in F\right).
    \end{equation}
\end{lemma}
\begin{proof}
    This is proven in \cite[Section~2.3]{HeyHofLasMat19}, building upon the FKG inequality for point processes (for example from \cite[Theorem~20.4]{LasPen17}).
\end{proof}

For the next two (similar) lemmas, we will make use of two subsets of $\HypDim$ that we now construct.

Select an arbitrary direction which we will refer to by the geodesic ray $\gamma_0$ emanating from $\orig$. Then, given $\theta\in\left[0,\pi\right]$, define
    \begin{equation}
        \Lcal_*\left(\theta\right) := \arcosh\left(\frac{1-\cos\theta\cos\frac{\theta}{2}}{\sin\theta\sin\frac{\theta}{2}}\right).
    \end{equation}
    The geometric construction of this length can be seen in Figure~\ref{fig:separattion_length}. The hyperbolic quadrilateral has internal angles $\theta$, $\pi-\theta$, $\pi-\theta$, and $0$ (the ideal point on the boundary of the hyperbolic plane). By splitting this into two triangles along the geodesic between the $\theta$ angle and the $0$ angle, the formula for $\Lcal_*\left(\theta\right)$ follows by applying the cosine rule for hyperbolic triangles.

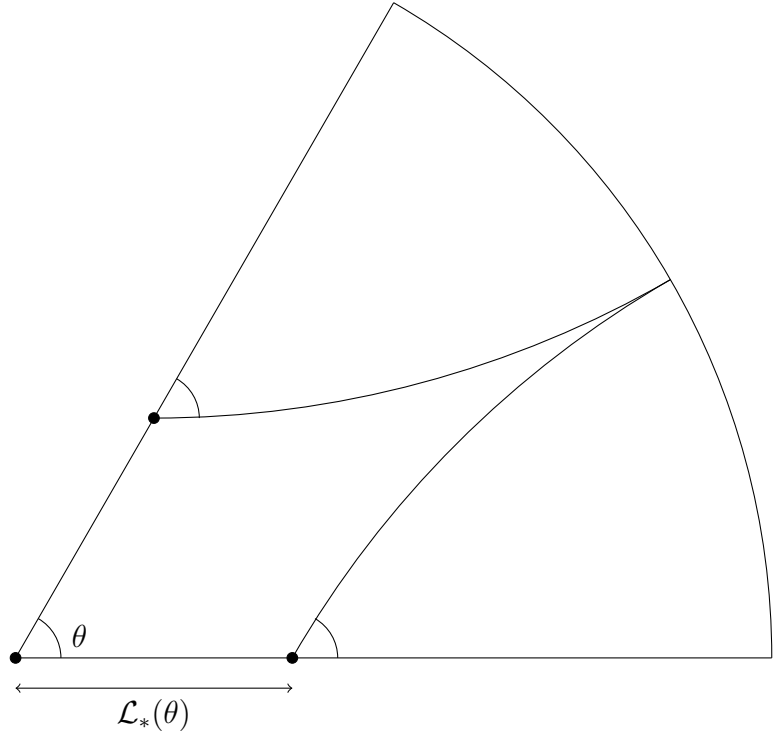
\begin{figure}
    \centering
    \begin{tikzpicture}[scale=2]
        \draw (5,0) arc (0:60:5);
        \draw (5,0) -- (0,0) -- (2.5,4.33);
        \draw (4.33,2.5) arc(120:150:6.83);
        \draw (4.33,2.5) arc(300:270:6.83);
        \filldraw (0,0) circle (1pt);
        \filldraw (1.83,0) circle (1pt);
        \filldraw (0.915,1.585) circle (1pt);
        \draw (0.3,0) node[above right]{$\theta$} arc(0:60:0.3);
        \draw (2.13,0) arc(0:60:0.3);
        \draw (1.215,1.585) arc(0:60:0.3);
        \draw[<->] (0,-0.2) -- (1.83,-0.2);
        \draw (0.915,-0.2) node[below]{$\Lcal_*(\theta)$};
    \end{tikzpicture}
    \caption{Construction of the length $\Lcal_*\left(\theta\right)$ in the Poincar{\'e} disc model.}
    \label{fig:separattion_length}
\end{figure}
    Given $x\in\HypDim$, let $t_x$ denote the unique translation isometry that maps $\orig\mapsto x$. Let $x_0,x_1\in\HypDim$ be distinct points satisfying $\dist{x_0,\orig},\dist{x_1,\orig}>\Lcal_*\left(\theta\right)$, where $\theta$ is the angle subtended at $\orig$ by $x_0$ and $x_1$. Also let $\left(k_1,k_2,\ldots\right)\in \left\{0,1\right\}^\N$ and $\left(l_1,l_2,\ldots\right)\in \left\{0,1\right\}^\N$ be two distinct sequences. Then it was proven in \cite{dickson2024hyperbolicrandomconnectionmodels} that there exists $\varepsilon>0$ such that 
    \begin{equation}
        \dist{t_{x_{k_n}}t_{x_{k_{n-1}}}\ldots t_{x_{k_1}}x_0,t_{x_{l_n}}t_{x_{l_{n-1}}}\ldots t_{x_{l_1}}x_1}>\varepsilon
    \end{equation}
    for all $n\geq 1$ such that $\left(k_1,k_2,\ldots,k_n\right)\ne \left(l_1,l_2,\ldots,l_n\right)$. Imprecisely, this means that in the binary tree generated by $x_0$ and $x_1$ the vertices do not get close to each other. The reasoning is that if the length of the edge labelled $\Lcal_*(\theta)$ in Figure~\ref{fig:separattion_length} was in fact longer, then the branching edges would not be able to meet by the time they reached the boundary.

    Now let $\angle\left(x,\orig,y\right)$ denote the angle subtended at $\orig$ by $x$ and $y$. Fix distinct $x_0,x_1\in\HypDim$ such that $\dist{x_0,\orig},\dist{x_1,\orig}>2\Lcal_*\left(\theta\right)$. Then for $\varepsilon>0$ we define the two frusta (cones without the point)
    \begin{align}
        \Vcal_0 &:= \left\{y\in\HypDim\colon \angle\left(x_0,\orig,y\right)<\varepsilon, \dist{y,\orig}>2\Lcal_*\left(\theta\right)\right\}\\
        \Vcal_1 &:= \left\{y\in\HypDim\colon \angle\left(x_1,\orig,y\right)<\varepsilon, \dist{y,\orig}>2\Lcal_*\left(\theta\right)\right\}.
    \end{align}
    The $d=2$ versions of these are sketched in Figure~\ref{fig:VsetConstruction}.
    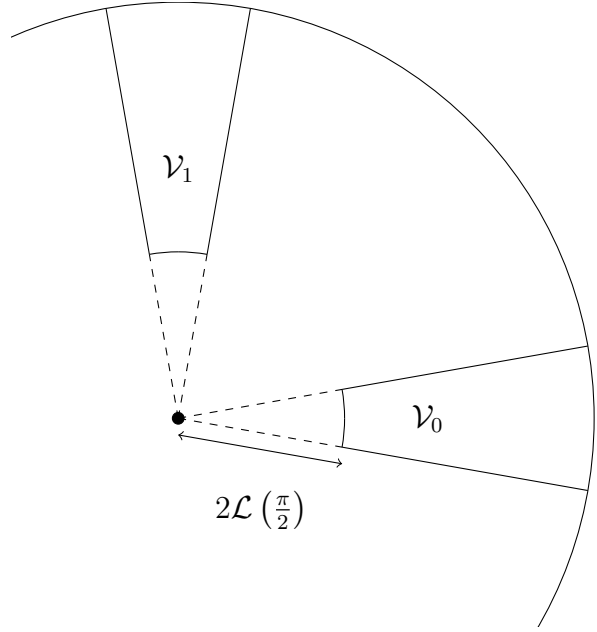
\begin{figure}
        \centering
        \begin{tikzpicture}[scale=1.1]
        \clip (-2,-2.5) rectangle + (7,7.5);
        \draw (5,0) arc (0:360:5);
        \draw (0,2) arc (90:100:2);
        \draw (0,2) arc (90:80:2);
        \draw (2,0) arc (0:10:2);
        \draw (2,0) arc (0:-10:2);
        \draw[dashed] (0,0) -- (1.970,0.347);
        \draw (1.970,0.347) -- (4.924,0.868);
        \draw[dashed] (0,0) -- (1.970,-0.347);
        \draw (1.970,-0.347) -- (4.924,-0.868);
        \draw[dashed] (0,0) -- (0.347,1.970);
        \draw (0.347,1.970) -- (0.868,4.924);
        \draw[dashed] (0,0) -- (-0.347,1.970);
        \draw (-0.347,1.970) -- (-0.868,4.924);
        \node at (0,3) {$\Vcal_1$};
        \node at (3,0) {$\Vcal_0$};
        \draw (5.5,0) arc (0:10:5.5);
        \draw (5.5,0) arc (0:-10:5.5);
        \node at (5.8,0) {$\varepsilon$};
        \draw[<->] (0,-0.2) -- (1.970,-0.347-0.2);
        \node at (1,-1.1) {$2\Lcal\left(\frac{\pi}{2}\right)$};
        \filldraw (0,0) circle (2pt);
    \end{tikzpicture}
        \caption{Sketch of the construction of the sets $\Vcal_0$ and $\Vcal_1$ in $\HypTwo$.}
        \label{fig:VsetConstruction}
    \end{figure}
    Then for each $n\geq 1$ and $\mathbf{k}=\left(k_1,\ldots,k_n\right)\in\left\{0,1\right\}^n$, we can define
    \begin{equation}
        \Vcal^{(n)}_{\mathbf{k}} := \left\{y\in\HypDim\colon \exists \left\{z_i\right\}^n_{i=1}\subset \prod^n_{i=1}\Vcal_{k_i} \text{ s.t. } y = t_{z_n}t_{z_{n-1}}\ldots t_{z_{2}}t_{z_{1}}\orig\right\}.
    \end{equation}
    Then for $\varepsilon$ sufficiently small, for all $n\geq 1$, $\mathbf{k},\mathbf{m}\in\left\{0,1\right\}^n$, $\mathbf{k}\ne \mathbf{m}$ implies $\Vcal^{(n)}_{\mathbf{k}}\cap \Vcal^{(n)}_{\mathbf{m}} = \emptyset$. Furthermore, if $n_1<n_2$ and $\left(k_1,\ldots,k_{n_1}\right)\ne \left(m_1,\ldots,m_{n_1}\right)$, then $\Vcal^{(n_1)}_{\mathbf{k}}\cap \Vcal^{(n_2)}_{\mathbf{m}} = \emptyset$.
    
\begin{lemma}
\label{lem:criticaltoOnetoOne-finite}
    If $\#\Ecal<\infty$ and \eqref{eqn:EigenValueRatio} holds, then for all $p,q\in\left[1,\infty\right]$ there exists $C_d=C_d(\connf,p,q)<\infty$ such that for $L$ sufficiently large,
    \begin{equation}
        \lambda_\mathrm{c}(L) \leq \frac{C_d}{\norm*{D_L}_{p\to q}}.
    \end{equation}
\end{lemma}

\begin{proof}
    Since $\#\Ecal<\infty$, for each $L$ there exist $a^*_L,b^*_L\in\Ecal$ such that
    \begin{equation}
        D_L\left(a^*_L,b^*_L\right) = \max_{a,b\in\Ecal}D_L(a,b),
    \end{equation}
    where we permit this maximum to be infinite. The symmetry of $\connf$ also implies that $D_L\left(b^*_L,a^*_L\right) = \max_{a,b\in\Ecal}D_L(a,b)$. Observe that since $\Ecal$ is finite, $\max_{a,b\in\Ecal}D_L(a,b)$ and all $\norm*{D_L}_{p\to q}$ are equivalent norms on $D_L$. Furthermore, the finiteness of $\Ecal$ and the convention that $\Pcal\left(a\right)>0$ for all $a\in\Ecal$ imply that $\liminf_{L\to \infty}\Pcal\left(a^*_L\right)>0$ and $\liminf_{L\to \infty}\Pcal\left(b^*_L\right)>0$.

    Now starting from $\origin{a^*_L}$, we aim to extract an infinite tree from $\C\left(\origin{a^*_L},\xi^\origin{a^*_L}\right)$. By Mecke's formula, the expected number of neighbours of $\origin{a^*_L}$ in $\Vcal_0\times\left\{b^*_L\right\}$ is given by
    \begin{multline}
        \E_{\lambda,L}\left[\#\left\{y\in\eta\cap \Vcal_0\times\left\{b^*_L\right\}\colon \adja{\origin{a^*_L}}{y}{\xi^{\origin{a^*_L}}}\right\}\right] = \lambda\Pcal\left(b^*_L\right)\int_{\Vcal_0}\connf_L\left(x;a^*_L,b^*_L\right)\dd x \\= \lambda\Pcal\left(b^*_L\right)c_d\int^\infty_{2\Lcal_*\left(\theta\right)}\connf_L\left(r;a^*_L,b^*_L\right)\left(\sinh r\right)^{d-1}\dd r,
    \end{multline}
    where $c_d=c_d(\varepsilon)$ is the $(d-1)$-Lebesgue volume of the spherical cap $\left\{y\in \partial \mathbb{B}\colon \angle\left(y,\orig,x_0\right)<\varepsilon\right\}$. Then by using \eqref{eqn:EigenValueRatio}, for $L$ sufficiently large
    \begin{multline}
        \int^\infty_{2\Lcal_*\left(\theta\right)}\connf_L\left(r;a^*_L,b^*_L\right)\left(\sinh r\right)^{d-1}\dd r \geq \frac{1}{2}\int^\infty_{0}\connf_L\left(r;a^*_L,b^*_L\right)\left(\sinh r\right)^{d-1}\dd r \\= \frac{1}{2\mathfrak{S}_{d-1}}D_L\left(a^*_L,b^*_L\right).
    \end{multline}
    Therefore
    \begin{equation}
        \E_{\lambda,L}\left[\#\left\{y\in\eta\cap \Vcal_0\times\left\{b^*_L\right\}\colon \adja{\origin{a^*_L}}{y}{\xi^{\origin{a^*_L}}}\right\}\right]\geq  \lambda\Pcal\left(b^*_L\right)\frac{c_d}{2\mathfrak{S}_{d-1}} D_L\left(a^*_L,b^*_L\right)
    \end{equation}
    for sufficiently large $L$.
    
    Since the number of neighbours of $\origin{a^*_L}$ is a Poisson random variable,
    \begin{align}
        &\mathbb{P}_{\lambda,L}\left(\exists y\in\eta\cap \Vcal_0\times\left\{b^*_L\right\}\colon \adja{\origin{a^*_L}}{y}{\xi^{\origin{a^*_L}}}\right)\nonumber\\
        &\hspace{4cm}= 1 - \exp\left(-\E_{\lambda,L}\left[\#\left\{y\in\eta\cap \Vcal_0\times\left\{b^*_L\right\}\colon \adja{\origin{a^*_L}}{y}{\xi^{\origin{a^*_L}}}\right\}\right]\right)\nonumber\\
        &\hspace{4cm} \geq 1 - \exp\left(-\frac{1}{2}\lambda\Pcal\left(b^*_L\right)\frac{c_d}{\mathfrak{S}_{d-1}} D_L\left(a^*_L,b^*_L\right)\right)
    \end{align}
    for $L$ sufficiently large. Similarly, 
    \begin{align}
        \mathbb{P}_{\lambda,L}\left(\exists y\in\eta\cap \Vcal_1\times\left\{b^*_L\right\}\colon \adja{\origin{a^*_L}}{y}{\xi^{\origin{a^*_L}}}\right) &\geq 1 - \exp\left(-\frac{1}{2}\lambda\Pcal\left(b^*_L\right)\frac{c_d}{\mathfrak{S}_{d-1}} D_L\left(a^*_L,b^*_L\right)\right)\\
        \mathbb{P}_{\lambda,L}\left(\exists y\in\eta\cap \Vcal_0\times\left\{a^*_L\right\}\colon \adja{\origin{b^*_L}}{y}{\xi^{\origin{b^*_L}}}\right) &\geq 1 - \exp\left(-\frac{1}{2}\lambda\Pcal\left(a^*_L\right)\frac{c_d}{\mathfrak{S}_{d-1}} D_L\left(a^*_L,b^*_L\right)\right)\\
        \mathbb{P}_{\lambda,L}\left(\exists y\in\eta\cap \Vcal_1\times\left\{a^*_L\right\}\colon \adja{\origin{b^*_L}}{y}{\xi^{\origin{b^*_L}}}\right) &\geq 1 - \exp\left(-\frac{1}{2}\lambda\Pcal\left(a^*_L\right)\frac{c_d}{\mathfrak{S}_{d-1}} D_L\left(a^*_L,b^*_L\right)\right).
    \end{align}

    We now construct our tree. Starting from the $0^{th}$ generation $\origin{a^*_L}$, we let our offspring be the nearest neighbour in the $\Vcal_0\times\left\{b^*_L\right\}$ branch and the nearest neighbour in the $\Vcal_1\times\left\{b^*_L\right\}$ branch (if either does not exist, then that branch produces no offspring). $\Vcal_0$ and $\Vcal_1$ are disjoint, so these offspring are independent. Furthermore, the vertex set that is further from $\origin{a^*_L}$ than these nearest neighbours is also independent of these neighbours. If a vertex $\left(x,b^*_L\right)$ is one of these offspring (i.e. in generation $1$), then it selects its offspring to be its nearest neighbour in $t_{x}\left(\Vcal_0\right)\times\left\{a^*_L\right\}$ and its nearest neighbour in $t_{x}\left(\Vcal_1\right)\times\left\{a^*_L\right\}$. This is then repeated iteratively, alternating between $a^*_L$ and $b^*_L$ in each generation. From the construction of $\Vcal_0$ and $\Vcal_1$, each offspring is independent from each other offspring in its generation and the offspring in previous generations (except its ancestral line).

    This random tree in $\HypDim\times\Ecal$ does not go extinct with a positive probability if the probability of each offspring is strictly greater than $\frac{1}{2}$. Therefore if
    \begin{equation}
        \frac{1}{2}\lambda\min\left\{\Pcal\left(a^*_L\right),\Pcal\left(b^*_L\right)\right\}\frac{c_d}{\mathfrak{S}_{d-1}} D_L\left(a^*_L,b^*_L\right) > \log 2,
    \end{equation}
    then the tree survives with a positive probability. Since this tree is a subset of $\C\left(\origin{a^*_L},\xi^\origin{a^*_L}\right)$, the cluster is infinite with positive probability. Since $\Pcal\left(a^*_L\right)>0$, we therefore have
    \begin{equation}
        \lambda_{\mathrm{c}}(L) \leq \frac{2\mathfrak{S}_{d-1}\log 2}{c_d\min\left\{\Pcal\left(a^*_L\right),\Pcal\left(b^*_L\right)\right\} D_L\left(a^*_L,b^*_L\right)}.
    \end{equation}
    Then from $\liminf_{L\to \infty}\Pcal\left(a^*_L\right)>0$ and $\liminf_{L\to \infty}\Pcal\left(b^*_L\right)>0$, and the equivalence of norms, the result follows.
\end{proof}

There are two main difficulties in applying the above argument to cases with infinitely many marks. The first is that we no longer have equivalence on norms on kernels on $\Ecal^2$. The second (and more serious) is that it is now possible that the significant marks move around infinitely many sets in ways that are very specific to the model in question and are difficult to control. Both of these issues can be solved by using the volume-linear scaling function. The property that
\begin{equation}
    D_L(a,b) = LD(a,b)
\end{equation}
for all $a,b\in\Ecal$ and $L$ ensures that the same marks are the significant ones for all scales. This - with the homogeneity of the norms - also implies that all $\norm*{D_L}_{p\to q} = L\norm*{D}_{p\to q}$. Therefore if $\norm*{D}_{p\to q}<\infty$, then the $L$ in the denominator in Lemma~\ref{lem:criticaltoOnetoOne-SpecialScaling} could be replaced with $\norm*{D_L}_{p\to q}$ like in Lemma~\ref{lem:criticaltoOnetoOne-finite}.

\begin{lemma}
\label{lem:criticaltoOnetoOne-SpecialScaling}
    If $\sigma_L(r)$ is volume-linear, then there exists $C_d=C_d(\connf)<\infty$ such that for $L$ sufficiently large,
    \begin{equation}
        \lambda_\mathrm{c}(L) \leq \frac{C_d}{L}.
    \end{equation}
\end{lemma}

\begin{proof}
    We will proceed similarly to the proof of the finitely many mark case, but instead of the two single marks $a^*_L$ and $b^*_L$, we will use sets $E,F\subset\Ecal$ and now our choice of $s_L$ will ensure that we do not need to vary these sets as $L$ changes.

    Since $\connf\not\equiv 0$, we have $\esssup_{a,b\in\Ecal}D(a,b)>0$. Therefore there exists $\varepsilon>0$ such that we can choose non-$\Pcal$-null sets $E,F\subset\Ecal$ such that 
    \begin{equation}
    \label{eqn:lowerboundonD}
        \essinf_{a\in E,b\in F}D(a,b) \in\left(\varepsilon,\infty\right).
    \end{equation}

    Now let $a\in E$. By Mecke's formula, the expected number of neighbours of $\origin{a}$ in $\Vcal_0\times F$ is given by
    \begin{multline}
        \E_{\lambda,L}\left[\#\left\{y\in\eta\cap \Vcal_0\times F\colon \adja{\origin{a}}{y}{\xi^{\origin{a}}}\right\}\right] = \lambda\int_F\int_{\Vcal_0}\connf_L\left(x;a,b\right)\dd x\Pcal\left(\dd b\right) \\= \lambda c_d\int_F \int^\infty_{2\Lcal_*\left(\theta\right)} \connf_L\left(r;a,b\right)\left(\sinh r\right)^{d-1}\dd r\Pcal\left(\dd b\right).
    \end{multline}
    Since $\connf\leq 1$ we have
    \begin{multline}
        \E_{\lambda,L}\left[\#\left\{y\in\eta\cap \Vcal_0\times F\colon \adja{\origin{a}}{y}{\xi^{\origin{a}}}\right\}\right] 
        \\\geq \frac{\lambda c_d}{\mathfrak{S}_{d-1}} \int_F D_L\left(a,b\right)\Pcal\left(\dd b\right) - \lambda c_d\Pcal\left(F\right) \mathbf{V}_d\left(2\Lcal_*\left(\theta\right)\right).
    \end{multline}
    Since our scaling function is volume-linear we have $D_L(a,b)= L D(a,b)$ for all $a,b\in\Ecal$, and therefore \eqref{eqn:lowerboundonD} implies that for sufficiently large $L$ we have
    \begin{equation}
        \E_{\lambda,L}\left[\#\left\{y\in\eta\cap \Vcal_0\times F\colon \adja{\origin{a}}{y}{\xi^{\origin{a}}}\right\}\right] \geq \lambda \frac{c_d}{2\mathfrak{S}_{d-1}}L\varepsilon\Pcal\left(F\right).
    \end{equation}
    Repeating this argument in all our required cases gives
        \begin{align}
        \essinf_{a\in E}\mathbb{P}_{\lambda,L}\left(\exists y\in\eta\cap \Vcal_0\times F\colon \adja{\origin{a}}{y}{\xi^{\origin{a}}}\right) &\geq 1 - \exp\left(-\lambda \frac{c_d}{2\mathfrak{S}_{d-1}}L\varepsilon\Pcal\left(F\right)\right)\\
        \essinf_{a\in E}\mathbb{P}_{\lambda,L}\left(\exists y\in\eta\cap \Vcal_1\times F\colon \adja{\origin{a}}{y}{\xi^{\origin{a}}}\right) &\geq 1 - \exp\left(-\lambda \frac{c_d}{2\mathfrak{S}_{d-1}}L\varepsilon\Pcal\left(F\right)\right)\\
        \essinf_{a\in F}\mathbb{P}_{\lambda,L}\left(\exists y\in\eta\cap \Vcal_0\times E\colon \adja{\origin{a}}{y}{\xi^{\origin{a}}}\right) &\geq 1 - \exp\left(-\lambda \frac{c_d}{2\mathfrak{S}_{d-1}}L\varepsilon\Pcal\left(E\right)\right)\\
        \essinf_{a\in F}\mathbb{P}_{\lambda,L}\left(\exists y\in\eta\cap \Vcal_1\times E\colon \adja{\origin{a}}{y}{\xi^{\origin{a}}}\right) &\geq 1 - \exp\left(-\lambda \frac{c_d}{2\mathfrak{S}_{d-1}}L\varepsilon\Pcal\left(E\right)\right).
    \end{align}
    Constructing the tree in the same way, replacing the singletons $a^*_L$ and $b^*_L$ with the sets $E$ and $F$ then tells us that
    \begin{equation}
        \lambda_{\mathrm{c}}(L) \leq \frac{2\mathfrak{S}_{d-1}\log 2}{c_d\min\left\{\Pcal\left(E\right),\Pcal\left(F\right)\right\} L \varepsilon}
    \end{equation}
    for $L$ sufficiently large.
\end{proof}

\begin{proof}[Proof of Theorem~\ref{thm:NonUniqueness}]
    Both assumptions~\ref{assump:finitelymany} and \ref{assump:specialscale} with Lemma~\ref{lem:limitQd} imply that the integral $\int^\infty_0\connf_L\left(r;a,b\right)Q_d\left(r\right)\e^{\left(d-1\right)r}\dd r<\infty$ for $\Pcal$-almost every $a,b\in\Ecal$ and $\norm*{\widetilde{\Opconnf}_L(0)}_{2\to 2}<\infty$ for all $L$ sufficiently large. Therefore by applying Lemma~\ref{lem:RatioofNorms} to the adjacency function $\connf_L$, we see that $\norm*{\Opconnf_L}_{2\to 2} = o\left( \norm*{D_L}_{2\to 2}\right)$ as $L\to \infty$. 
    
    Under Assumption~\ref{assump:finitelymany}, Lemma~\ref{lem:criticaltoOnetoOne-finite} then implies that $\lambda_c(L) =o\left( \norm*{\Opconnf_L}_{2\to 2}^{-1}\right)$ as $L\to \infty$. If $\norm*{D}_{p\to p}<\infty$ for any $p\in\left[1,\infty\right]$, then Lemma~\ref{lem:criticaltoOnetoOne-SpecialScaling} proves the same asymptotic behaviour under Assumption~\ref{assump:specialscale}. If $\norm*{D}_{2\to 2}=\infty$ (and therefore $\norm*{D}_{p\to p}=\infty$ for all $p\in\left[1,\infty\right]$ by the reasoning in Remark~\ref{rem:2to2isminimal}), then Lemma~\ref{lem:NormSublinear} is also required to show $\lambda_\mathrm{c}(L) \leq \frac{C_d}{L} \ll \frac{1}{\norm*{\Opconnf_L}_{2\to 2}}$ as $L\to\infty$.

    Finally Lemma~\ref{lem:Uniquenesstoqtoq} tells us that $\frac{1}{\norm*{\Opconnf_L}_{2\to 2}} \leq \lambda_{u}(L)$ for all $L$, and therefore for sufficiently large $L$ we have $\lambda_c(L)<\lambda_u(L)$.

\end{proof}

\subsection{Proof of Critical Exponents}
\label{sec:ProofCritExponents}

\begin{remark}
    If we were to consider Bernoulli bond percolation on a given graph, then usually the finiteness of the triangle diagram ($\triangle_\lambda$ in our case) at criticality would be sufficient to derive mean-field behaviour for various critical exponents. This is because finiteness would show that an \emph{open triangle condition} would hold, and it is this that is actually used to derive critical exponents. For Bernoulli bond percolation on $\Z^d$ with nearest neighbour edges the fact that the triangle condition implied the open triangle condition was proven in \cite{AizBar91}, and it was proven in \cite{kozma2011triangle} for any transitive graph. Both of these use the fact that the two-point function is of positive type on the set of vertices. However, the differentiability (and therefore continuity) of the two-point function for RCMs on $\Rd$ shown by \cite[Lemma~2.2]{HeyHofLasMat19} and the equality $\tau_0\left(\mathbf{x},\mathbf{y}\right) = \connf(\mathbf{x},\mathbf{y})$ shows that this need not be true for RCMs for which the adjacency function is not itself of positive type (such as the Boolean disc model).
\end{remark}

\begin{lemma}
\label{lem:TriangleIsSmall}
    If Assumption~\ref{assump:finitelymany} holds, or Assumption~\ref{assump:specialscale} and
    \begin{equation}
    \label{eqn:supintegralsquaredbound}
        \esssup_{a\in\Ecal} \int_{\Ecal}\int_{\HypDim}\connf_L(x;a,b)^2\mu\left(\dd x\right)\Pcal\left(\dd b\right)<\infty
    \end{equation}
    hold for sufficiently large $L$, then
    \begin{equation}
        \lim_{L\to\infty}\triangle_{\lambda_{\mathrm{c}}(L)}= 0.
    \end{equation}
\end{lemma}

\begin{proof}
    From Lemma~\ref{lem:TriangleFinite} and having $\lambda_{2\to 2}(L)>\lambda_{\mathrm{c}}(L)$ for sufficiently large $L$, the condition \eqref{eqn:supintegralsquaredbound} implies that $\triangle_{\lambda_{\mathrm{c}}(L)}<\infty$ for $L$ sufficiently large. We now need to show that it is in fact ``small".

    Repeating the argument of the proof of Lemma~\ref{lem:TriangleFinite}, we can show that
    \begin{multline}
        \triangle_{\lambda_{\mathrm{c}}(L)} \leq \left(\lambda_{\mathrm{c}}(L)\norm*{\OptlamCritL}_{2\to 2} + 2\lambda_{\mathrm{c}}(L)^2\norm*{\OptlamCritL}_{2\to 2}^2 + \lambda_{\mathrm{c}}(L)^3\norm*{\OptlamCritL}_{2\to 2}^3\right)\\
        \times\lambda_{\mathrm{c}}(L)\esssup_{a\in\Ecal} \int_{\Ecal}\int_{\HypDim}\connf_L(x;a,b)^2\mu\left(\dd x\right)\Pcal\left(\dd b\right).
    \end{multline}
    
    Now we can use $\connf_L\leq 1$ to get
    \begin{align}
        \lambda_{\mathrm{c}}(L)\esssup_{a\in\Ecal}\int_{\Ecal}\int_{\HypDim}\connf_L(x;a,b)^2\mu\left(\dd x\right)\Pcal\left(\dd b\right) &\leq \lambda_{\mathrm{c}}(L)\esssup_{a\in\Ecal}\int_{\Ecal}D_L(a,b)\Pcal\left(\dd b\right)\nonumber\\
        &= \lambda_{\mathrm{c}}(L)\norm*{D_L}_{1\to 1}\nonumber\\
        &\leq C_d,
    \end{align}
    where $C_d$ is the constant arising in Lemma~\ref{lem:criticaltoOnetoOne-finite} or Lemma~\ref{lem:criticaltoOnetoOne-SpecialScaling} as appropriate. We therefore only need to show that $\lambda_{\mathrm{c}}(L)\norm*{\OptlamCritL}_{2\to 2}\to 0$ as $L\to\infty$.

    By using the bound \eqref{eqn:TwoPointBoundGreen} arising from bounding $\tlam$ with $\Greenlam$, we get
    \begin{equation}
        \lambda_{\mathrm{c}}(L)\norm*{\OptlamCritL}_{2\to 2} \leq \sum_{k=1}^\infty\lambda_{\mathrm{c}}(L)^k\norm*{\Opconnf_L}_{2\to 2}^k.
    \end{equation}
    Since we have $\lambda_{\mathrm{c}}(L)\ll\frac{1}{\norm*{\Opconnf_L}_{2\to 2}}$ as $L\to\infty$, we therefore have $\lambda_{\mathrm{c}}(L)\norm*{\OptlamCritL}_{2\to 2}\to 0$ and $\triangle_{\lambda_{\mathrm{c}}(L)}\to 0$ as $L\to\infty$.
\end{proof}

We now want to use results from \cite{caicedo2023criticalexponentsmarkedrandom} to show that a form of triangle condition does indeed imply that certain critical exponents take their mean-field values. Let us introduce some notation so we can use their result. Given $\lambda>0$, $a\in\Ecal$, and $L>0$, define
\begin{align}
    \Ical_{\lambda,a}(L) &:= \left(\sup_{k\geq 1}\left(\frac{\lambda}{1+\lambda}\right)^k\essinf_{b\in\Ecal}D_L^{(k)}(a,b)\right)^{-1},\\
    \Jcal_{\lambda,a}(L) &:= \left(\sup_{k\geq 1}\left(\frac{\lambda}{2\left(1+\lambda\right)}\right)^k\essinf_{b\in\Ecal}D_L^{(k)}(a,b)\right)^{-1},\\
    \cbar{\lambda}(L) &:= 1+\lambda\esssup_{a,b,c\in\Ecal}D_L(a,b)\Jcal_{\lambda,c}(L),\\
    \ConstantTriangle(L) &:= \min\left\{\frac{1}{\left(1+\lambda_T(L)\esssup_{a,b,c\in\Ecal}D_L(a,b)\Ical_{\lambda_T(L),c}(L)\right)^2},\right.\nonumber\\
    &\hspace{3cm}\left.\frac{1}{\cbar{\lambda_T(L)}(L)}\frac{\lambda_T(L)^2\left(\essinf_{a\in\Ecal} \int D_L(a,b)\Pcal(\dd b)\right)^2}{1 + 2\lambda_T(L)\esssup_{a\in\Ecal}\int D_L(a,b)\Pcal(\dd b)}\right\}.
\end{align}

Following the naming from \cite{caicedo2023criticalexponentsmarkedrandom}, we have three conditions:
\paragraph{Conditions:}
\begin{enumerate}[label=\textbf{(D.\arabic*)}]
\item \label{Assump:BoundExpectedDegree}
    \begin{equation}
        \esssup_{a,b\in\Ecal}D_L(a,b) < \infty, \label{eqn:supsupbound}
    \end{equation}

\item \label{Assump:AllReachablebySome}
    \begin{equation}
        \esssup_{a\in\Ecal}\sup_{k\geq 1}\essinf_{b\in\Ecal}D_L^{(k)}(a,b) >0,\label{eqn:supinfbound}
    \end{equation}
\end{enumerate}
\begin{enumerate}[label=\textbf{(T)}]
    \item
    \begin{equation}
    \triangle_{\lambda_T(L)} < \ConstantTriangle(L).
\end{equation}\label{TriangleCondition_Assumption}
\end{enumerate}

The following proposition is an amalgamation of Theorems~1.7, 1.8, 1.9, 1.11, and 1.12, and Corollary~1.10, all from \cite{caicedo2023criticalexponentsmarkedrandom}. The notation $\asymp$ is used to denote asymptotic equality ``in the bounded ratio sense."
\begin{prop}\label{prop:summaryCD24}
    If Conditions~\ref{Assump:BoundExpectedDegree}, \ref{Assump:AllReachablebySome}, and \ref{TriangleCondition_Assumption} all hold for some $L$, then
    \begin{equation}
    \begin{array}{rll}
        \norm*{\chi_{\lambda,L}}_p &\asymp \left(\lambda_{\mathrm{c}}(L)-\lambda\right)^{-1} &\text{as }\lambda\nearrow\lambda_{\mathrm{c}}(L),\\
        \norm*{\theta_{\lambda,L}}_p &\asymp \lambda-\lambda_{\mathrm{c}}(L) &\text{as }\lambda\searrow\lambda_{\mathrm{c}}(L),\\
        \mathbb{P}_{\lambda_\mathrm{c}(L),L}\left(\#\C\left(\origin{a},\xi^{\origin{a}}\right) = n\right) &\asymp n^{-1-\frac{1}{2}} &\text{as }n\to\infty,
    \end{array}
    \end{equation}
    for $\Pcal$-almost every $a\in\Ecal$, for all $p\in\left[1,\infty\right]$, and that $L$ in question. In particular, $\lambda_T(L)=\lambda_\mathrm{c}(L)$.
\end{prop}
While the argument in \cite{caicedo2023criticalexponentsmarkedrandom} is phrased for RCMs on $\Rd\times\Ecal$, the argument follows in exactly the same way for RCMs on $\HypDim\times\Ecal$.

Now if Conditions~\ref{Assump:BoundExpectedDegree}, \ref{Assump:AllReachablebySome}, and \ref{TriangleCondition_Assumption} can all be shown to hold, Theorem~\ref{thm:meanfield} will be proven. Conditions \ref{Assump:BoundExpectedDegree} and \ref{Assump:AllReachablebySome} follow fairly directly from the extra ingredients of Assumptions \ref{assump:finitelymanyPlus} and \ref{assump:specialscalePlus}. In Assumption~\ref{assump:finitelymanyPlus}, \eqref{eqn:maxmaxOverminsum} implies $\max_{a,b\in\Ecal} D_L(a,b)<\infty$ for all $L$ (since $D_L(a,b)$ is non-decreasing in $L$), and \eqref{eqn:kstepRatio} implies $\max_{a\in\Ecal}\sup_{k\geq 1}\min_{b\in\Ecal} D^{(k)}_L(a,b)>0$ for all $L$ sufficiently large. Therefore Assumption~\ref{assump:finitelymanyPlus} implies that the Conditions~\ref{Assump:BoundExpectedDegree} and \ref{Assump:AllReachablebySome} of \cite{caicedo2023criticalexponentsmarkedrandom} both hold. 

In Assumption~\ref{assump:specialscalePlus}, the conditions \eqref{eqn:bounddegreedensity} and \eqref{eqn:strongIrreducibilityCondition} are precisely the Conditions~\ref{Assump:BoundExpectedDegree} and \ref{Assump:AllReachablebySome} of \cite{caicedo2023criticalexponentsmarkedrandom} for the model with the \emph{reference} adjacency function, $\connf$, and \eqref{eqn:infIntegralDegree} makes it feasible for Condition~\ref{TriangleCondition_Assumption} to hold. The choice of a volume-linear scaling then means that these values change in an explicit way, and we can show that all of \ref{Assump:BoundExpectedDegree}, \ref{Assump:AllReachablebySome}, and \ref{TriangleCondition_Assumption} hold for the \emph{scaled} versions of the model.

This leaves only Condition~\ref{TriangleCondition_Assumption} to address.

\begin{lemma}
\label{lem:FiniteCase-ConstantPositive}
    If $\#\Ecal<\infty$, and
    \begin{align}
        \liminf_{L\to\infty}\sup_{k\geq 1}\frac{\min_{a,b\in\Ecal} D^{(k)}_L(a,b)}{\left(\max_{a\in\Ecal}\sum_{b\in\Ecal} D_L(a,b)\Pcal\left(b\right)\right)^k} &>0,
        \label{eqn:FiniteIrreducibility-ish}\\
        \limsup_{L\to\infty}\frac{\max_{a,b\in\Ecal} D_L(a,b)}{\min_{a\in\Ecal}\sum_{b\in\Ecal} D_L(a,b)\Pcal\left(b\right)}&<\infty, \label{eqn:FiniteSameOrder}
    \end{align}
    both hold, then there exists $\varepsilon>0$ such that
    \begin{equation}
        \ConstantTriangle(L) \geq \varepsilon
    \end{equation}
    for $L$ sufficiently large.
\end{lemma}

\begin{proof}
    For succinctness, let us denote
    \begin{align}
        \overline{A}_\star(L) &:= \max_{a,b\in\Ecal} D_L(a,b)\\
        \overline{A}(L) &:= \max_{a\in\Ecal}\sum_{b\in\Ecal} D_L(a,b)\Pcal\left(b\right)\\
        \underline{A}(L) &:= \min_{a\in\Ecal}\sum_{b\in\Ecal} D_L(a,b)\Pcal\left(b\right)\\
        \underline{A}^{(k)}_\star(L) &:= \min_{a,b\in\Ecal} D^{(k)}_L(a,b).
    \end{align}
    First observe that 
    \begin{multline}
        \max_{a\in\Ecal}\Ical_{\lambda_T(L),a}(L) = \left(\min_{a\in\Ecal}\sup_{k\geq 1}\left(\frac{\lambda_T(L)}{1+\lambda_T(L)}\right)^k\min_{b\in\Ecal}D_L^{(k)}(a,b)\right)^{-1} \\ = \left(\sup_{k\geq 1}\left(\frac{\lambda_T(L)}{1+\lambda_T(L)}\right)^k\underline{A}^{(k)}_\star(L)\right)^{-1}.
    \end{multline}
    Then from Lemmata~\ref{lem:qtoqintensity_bound} and \ref{lem:criticaltoOnetoOne-finite} there exist constants $c_1,c_2\in\left(0,\infty\right)$ such that
    \begin{equation}
        \frac{c_1}{\overline{A}(L)} \leq \lambda_T(L) \leq \frac{c_2}{\overline{A}(L)}.
    \end{equation}
    This means that 
    \begin{equation}
        \max_{a\in\Ecal}\Ical_{\lambda_T(L),a}(L) \leq \left(\sup_{k\geq 1}\left(\frac{c_1}{c_1 + \overline{A}(L)}\right)^k\underline{A}^{(k)}_\star(L)\right)^{-1}.
    \end{equation}
    Then \eqref{eqn:FiniteIrreducibility-ish} implies that there exists a constant $C<\infty$ such that
    \begin{align}
        \max_{a\in\Ecal}\Ical_{\lambda_T(L),a}(L) &\leq C\\
        \max_{a\in\Ecal}\Jcal_{\lambda_T(L),a}(L) &\leq C
    \end{align}
    for all $L\geq 1$. We also have
    \begin{equation}
        \limsup_{L\to\infty}\frac{\overline{A}_\star(L)}{\overline{A}(L)}\leq \limsup_{L\to\infty}\frac{\overline{A}_\star(L)}{\underline{A}(L)}<\infty,
    \end{equation}
    which then implies that there exists a constant $C'<\infty$ such that
    \begin{align}
        \cbar{\lambda_T(L)}(L) &\leq 1 + C',\\
        \left(1+\lambda_T(L)\max_{a,b,c\in\Ecal}
        D_L(a,b)\Ical_{\lambda_T(L),c}(L)\right)^{-2} &\geq \left(1+C'\right)^{-2}.
    \end{align}
    Therefore
    \begin{align}
        &\frac{1}{\cbar{\lambda_T(L)}(L)}\frac{\lambda_T(L)^2\left(\min_{a\in\Ecal} \sum_{b\in\Ecal} D_L(a,b)\Pcal( b)\right)^2}{1 + 2\lambda_T(L)\max_{a\in\Ecal}\sum_{b\in\Ecal} D_L(a,b)\Pcal( b)} \nonumber\\
        &\hspace{7cm}
        \geq \frac{1}{1 + C'}\frac{\lambda_T(L)^2\underline{A}(L)^2}{1 + 2\lambda_T(L)\overline{A}(L)}\nonumber\\
        &\hspace{7cm}
        \geq \frac{1}{1 + C'}\frac{c_1^2}{1 + 2c_2}\left(\frac{\underline{A}(L)}{\overline{A}(L)}\right)^2 \nonumber\\
        &\hspace{7cm}
        \geq \frac{1}{1 + C'}\frac{c_1^2}{1 + 2c_2}\left(\frac{\overline{A}_\star(L)}{\underline{A}(L)}\right)^{-2}.
    \end{align}
    The condition \eqref{eqn:FiniteSameOrder} then produces the result.
\end{proof}

\begin{lemma}
\label{lem:ScalingCase-ConstantPositive}
    If $\sigma_L$ is volume-linear, and
    \begin{align}
        \esssup_{a,b\in\Ecal} D(a,b)&<\infty \label{eqn:Dnormonetoone}\\
        \essinf_{a\in\Ecal}\int_\Ecal D(a,b)\Pcal\left(\dd b\right)&>0\\
        \esssup_{a\in\Ecal}\sup_{k\geq 1}\essinf_{b\in\Ecal}D^{(k)}(a,b)&>0 \label{eqn:StrongIrreducibility}
    \end{align}
    all hold, then there exists $\varepsilon>0$ such that
    \begin{equation}
        \ConstantTriangle(L) \geq \varepsilon
    \end{equation}
    for $L$ sufficiently large.
\end{lemma}

\begin{proof}
    First observe that the choice of $\sigma_L=s_L$ means that
    \begin{equation}
        D_L(a,b) = LD(a,b)
    \end{equation}
    for all $a,b\in\Ecal$ and $L>0$. This then means that 
    \begin{align}
        \esssup_{a\in\Ecal}\int D_L(a,b)\Pcal(\dd b) &= L\esssup_{a\in\Ecal}\int D(a,b)\Pcal(\dd b),\\
        \essinf_{a\in\Ecal}\int D_L(a,b)\Pcal(\dd b) &= L\essinf_{a\in\Ecal}\int D(a,b)\Pcal(\dd b),\\
        \essinf_{b\in\Ecal}D_L^{(k)}(a,b) &= L^k\essinf_{b\in\Ecal}D^{(k)}(a,b) \qquad\forall a\in\Ecal,k\geq 1,
    \end{align}
    for all $L>0$.

    From Lemma~\ref{lem:qtoqintensity_bound}, we then have
    \begin{equation}
        \lambda_T(L) = \lambda_{1\to 1}(L) \geq \frac{1}{\norm*{\Opconnf_L}_{1\to 1}} = \frac{1}{\norm*{D_L}_{1\to 1}} = \frac{1}{L\norm*{D}_{1\to 1}}>0,
    \end{equation}
    where this positivity follows from \eqref{eqn:Dnormonetoone} implying $\norm*{D}_{1\to 1}<\infty$. Therefore for $L\geq 1$,
    \begin{multline}
        \esssup_{a\in\Ecal}\sup_{k\geq 1}\left(\frac{\lambda_T(L)}{1+\lambda_T(L)}\right)^k\essinf_{b\in\Ecal}D_L^{(k)}(a,b) \\\geq \esssup_{a\in\Ecal}\sup_{k\geq 1}\left(\frac{1}{\norm*{D}_{1\to 1}+1}\right)^k\essinf_{b\in\Ecal}D^{(k)}(a,b).
    \end{multline}
    In particular, this bound is $L$-independent and \eqref{eqn:StrongIrreducibility} implies that it is strictly positive. This then means that there exists $c<\infty$ such that
    \begin{align}
        \esssup_{a\in\Ecal}\Ical_{\lambda_T(L),a}(L) &\leq c,\\
        \esssup_{a\in\Ecal}\Jcal_{\lambda_T(L),a}(L) &\leq c.
    \end{align}
    Therefore
    \begin{equation}
        \cbar{\lambda_T(L)}(L) \leq 1 + cL\lambda_T(L)\esssup_{a,b\in\Ecal}D(a,b) \leq 1 + \frac{c}{\norm*{D}_{1\to 1}}\esssup_{a,b\in\Ecal}D(a,b).
    \end{equation}

    Also recall from Lemma~\ref{lem:criticaltoOnetoOne-SpecialScaling} that
    \begin{equation}
        \lambda_T(L) \leq \frac{C_d}{\norm*{D}_{1\to 1}}.
    \end{equation}
    
    In summary, we have
    \begin{equation}
        \left(1+\lambda_T(L)\esssup_{a,b,c\in\Ecal}D_L(a,b)\Ical_{\lambda_T(L),c}(L)\right)^{-2} \geq \left(1 + \frac{c}{\norm*{D}_{1\to 1}}\esssup_{a,b\in\Ecal}D(a,b)\right)^{-2} >0,
    \end{equation}
    and
    \begin{multline}
        \frac{1}{\cbar{\lambda_T(L)}(L)}\frac{\lambda_T(L)^2\left(\essinf_{a\in\Ecal} \int D_L(a,b)\Pcal(\dd b)\right)^2}{1 + 2\lambda_T(L)\esssup_{a\in\Ecal}\int D_L(a,b)\Pcal(\dd b)} \\\geq \frac{1}{\left(\norm*{D}_{1\to 1} + c\esssup_{a,b\in\Ecal}D(a,b)\right)\left(\norm*{D}_{1\to 1}+2C_d\esssup_{a,b\in\Ecal}D(a,b)\right)}>0.
    \end{multline}
    Therefore we have the required strictly positive and $L$-independent lower bound for $\ConstantTriangle(L)$.
\end{proof}

\section{Proofs for Specific Models}
\label{sec:ProofSpecificModels}

\subsection{Boolean Disc Model}
\label{sec:ProofBooleanModel}

\begin{proof}[Proof of Corollary~\ref{lem:BooleanCorollary}]
    For part~\ref{enum:BooleanCorPart0}, first note that we have 
    \begin{equation}
    \label{eqn:BooleanPtAitem1}
        \lambda_\mathrm{u}(L) \geq \frac{1}{\norm*{\Opconnf_L}_{2\to 2}}
    \end{equation}
    through Lemmata~\ref{lem:qtoqintensity_bound} and \ref{lem:Uniquenesstoqtoq}. In general we cannot directly apply Lemma~\ref{lem:RatioofNorms} to our model, but the argument needs only a slight modification for our case. Fix $R\in\left(0,\infty\right)$. Since $D^{(\leq R)}_L(a,b)\leq \mathfrak{S}_{d-1}\mathbf{V}_d\left(R\right)$ for all $L$ and all $a,b\in\Ecal$,
    \begin{equation}
        \norm*{D^{(\leq R)}_L}_{2\to 2} \leq \mathfrak{S}_{d-1}\mathbf{V}_d\left(R\right)
    \end{equation}
    for all $L$. Now fix $\varepsilon>0$ such that $\Pcal\left(\left(\varepsilon,\infty\right)\right)>0$ and consider $f(a):= \Pcal\left(\left(\varepsilon,\infty\right)\right)^{-\frac{1}{2}}\Id_{\left\{a>\varepsilon\right\}}$. We have $\norm*{f}_2=1$ and
    \begin{align}
        \norm*{D_L}_{2\to 2}^2 \geq \norm*{D_Lf}_2^2 &= \frac{\mathfrak{S}^2_{d-1}}{\Pcal\left(\left(\varepsilon,\infty\right)\right)}\int_{\Ecal}\left(\int_{\Ecal\cap\left(\varepsilon,\infty\right)}\mathbf{V}_d\left(\sigma_L\left(a+b\right)\right)\Pcal\left(\dd a\right)\right)^2\Pcal\left(\dd b\right)\nonumber\\
        &\geq \frac{\mathfrak{S}^2_{d-1}}{\Pcal\left(\left(\varepsilon,\infty\right)\right)}\int_{\Ecal}\mathbf{V}_d\left(\sigma_L\left(\varepsilon\right)\right)^2\Pcal\left(\left(\varepsilon,\infty\right)\right)^2\Pcal\left(\dd b\right)\nonumber\\
        & = \mathfrak{S}^2_{d-1}\Pcal\left(\left(\varepsilon,\infty\right)\right)\mathbf{V}_d\left(\sigma_L\left(\varepsilon\right)\right)^2 \to\infty
    \end{align}
    as $L\to\infty$. Therefore $\norm*{D^{(\leq R)}_L}_{2\to 2} = o\left(\norm*{D_L}_{2\to 2}\right)$ as $L\to\infty$ and the argument in Lemma~\ref{lem:RatioofNorms} proceeds in the same way to show
    \begin{equation}
    \label{eqn:BooleanPtAitem2}
        \lim_{L\to\infty}\frac{\norm*{\Opconnf_L}_{2\to 2}}{\norm*{D_L}_{2\to 2}} = 0.
    \end{equation}
    
    In general we also cannot directly apply Lemma~\ref{lem:criticaltoOnetoOne-finite} or Lemma~\ref{lem:criticaltoOnetoOne-SpecialScaling}. However, we can bound $\lambda_\mathrm{c}(L)$ for our model by the critical intensity for another model that we can apply Lemma~\ref{lem:criticaltoOnetoOne-finite} for. Recall $R^*:= \esssup\Ecal\in\left(0,\infty\right)$, and let $\lambda^*_\mathrm{c}(L)$ denote the percolation critical intensity for the same scaled Boolean disc model but with the finite mark space $\Ecal^*=\left\{0,R^*\right\}$ and $\Pcal^*\left(\left\{0\right\}\right) = \Pcal\left(\left(0,R^*\right)\right)$ and $\Pcal^*\left(\left\{R^*\right\}\right) = \Pcal\left(\left\{R^*\right\}\right)$. This can be coupled to the original scaled Boolean disc model in such a way that the graph resulting from the new model is a sub-graph of the graph resulting from the original model. Therefore
    \begin{equation}
        \lambda^*_\mathrm{c}\left(L\right) \geq \lambda_\mathrm{c}\left(L\right).
    \end{equation}
    We can now apply Lemma~\ref{lem:criticaltoOnetoOne-finite} to this new model. From the equivalence of norms on finite dimensional spaces, 
    \begin{equation}
        \norm*{D^*_L}_{2\to 2} \asymp \mathbf{V}_d\left(\sigma_L\left(2R^*\right)\right)
    \end{equation}
    as $L\to \infty$. On the other hand, bounding by the Hilbert-Schmidt norm gives
    \begin{multline}
        \norm*{D_L}_{2\to 2} \leq \HSNorm{D_L} = \mathfrak{S}_{d-1}\left(\int_{\left(0,R^*\right]}\int_{\left(0,R^*\right]}\mathbf{V}_d\left(\sigma_L\left(a+b\right)\right)^2\Pcal\left(\dd a\right)\Pcal\left(\dd b\right)\right)^\frac{1}{2}\\
        \leq \mathfrak{S}_{d-1}\mathbf{V}_d\left(\sigma_L\left(2R^*\right)\right).
    \end{multline}
    Therefore there exists a constant $C>0$ such that for sufficiently large $L$
    \begin{equation}
    \label{eqn:BooleanPtAitem3}
        \lambda_\mathrm{c}(L) \leq \frac{C}{\norm*{D_L}_{2\to 2}}.
    \end{equation}
    Taken together, \eqref{eqn:BooleanPtAitem1}, \eqref{eqn:BooleanPtAitem2}, and \eqref{eqn:BooleanPtAitem3} are enough to prove part~\ref{enum:BooleanCorPart0}.
    
    To prove part~\ref{enum:BooleanCorPart1}, directly use Theorem~\ref{thm:NonUniqueness} with Assumption~\ref{assump:specialscale}. This requires that the $L^2\left(\Ecal\right)\to L^2\left(\Ecal\right)$ operator norm of the operator constructed from the function
    \begin{equation}
        \left(a,b\right)\mapsto \int^{a+b}_0r\exp\left(\frac{1}{2}\left(d-1\right)r\right)\dd r
    \end{equation}
    is finite. Let us denote this operator by $\Bcal$ (for ``Boolean"). We can bound the $L^2(\Ecal)\to L^2(\Ecal)$ operator norm by the Hilbert-Schmidt norm:
    \begin{equation}
        \norm*{\Bcal}_{2\to 2} \leq \HSNorm{\Bcal} = \left(\int^\infty_0\int^\infty_0\left(\int^{a+b}_0r\exp\left(\frac{1}{2}\left(d-1\right)r\right)\dd r\right)^2\Pcal\left(\dd a\right)\Pcal\left(\dd b\right)\right)^\frac{1}{2}.
    \end{equation}
    There exists a constant $C_d<\infty$ such that
    \begin{multline}
        \int^{a+b}_0r\exp\left(\frac{1}{2}\left(d-1\right)r\right)\dd r = \frac{4}{\left(d-1\right)^2}\left(1 - \left(1-\frac{1}{2}\left(d-1\right)\left(a+b\right)\e^{\frac{1}{2}\left(d-1\right)\left(a+b\right)}\right)\right)\\
        \leq
        C_d\left(\left(1\vee a\right)+\left(1\vee b\right)\right)\e^{\frac{1}{2}\left(d-1\right)\left(a+b\right)},
    \end{multline}
    so that
    \begin{align}
        \norm*{\Bcal}_{2\to 2}^2 &\leq C_d^2\int^\infty_0\int^\infty_0 \left(\left(1\vee a\right)+\left(1\vee b\right)\right)^2\e^{\left(d-1\right)\left(a+b\right)}\Pcal\left(\dd a\right)\Pcal\left(\dd b\right)\nonumber\\
        & =  C_d^2\int^\infty_0\int^\infty_0 \left(\left(1\vee a\right)^2+\left(1\vee b\right)^2\right)\e^{\left(d-1\right)\left(a+b\right)}\Pcal\left(\dd a\right)\Pcal\left(\dd b\right) \nonumber\\ &\hspace{4cm}+ 2C_d^2 \int^\infty_0\int^\infty_0 \left(1\vee a\right)\left(1\vee b\right)\e^{\left(d-1\right)\left(a+b\right)}\Pcal\left(\dd a\right)\Pcal\left(\dd b\right)\nonumber\\
        & = 2C_d^2\left(\int^\infty_0\e^{\left(d-1\right)b}\Pcal\left(\dd b\right)\int^\infty_0\left(1\vee a\right)^2\e^{\left(d-1\right)a}\Pcal\left(\dd a\right)\right.\nonumber\\
        &\hspace{4cm}+ \left.\left(\int^\infty_0\left(1\vee a\right)\e^{\left(d-1\right)a}\Pcal\left(\dd a\right)\right)^2\right)\nonumber\\
        & \leq \left(2C_d \int^\infty_0\left(1\vee r^2\right)\e^{\left(d-1\right)r}\Pcal\left(\dd r\right)\right)^2.
    \end{align}
    This bound is finite precisely when $\int^\infty_0r^2\e^{\left(d-1\right)r}\Pcal\left(\dd r\right)<\infty$, and therefore the condition \eqref{eqn:FiniteExpectedVolume} is sufficient to apply Theorem~\ref{thm:NonUniqueness}.

    For part \ref{enum:BooleanCorPart2}, let us first consider the $L=1$ version. This version still has the interpretation in which a vertex with mark $a$ is assigned a ball of radius $a$ that is centred on it, and two vertices are adjacent in the RCM if their balls overlap. For the associated model on $\Rd$, \cite{hall1985continuum} gives an argument that every point in $\Rd$ is covered by a ball if and only if the expected size of the random ball is infinite. The same argument applies for the $\HypDim$ models when $L=1$, and the expected size of the balls is finite exactly when \eqref{eqn:FiniteExpectedVolume} holds. Therefore if this expectation is infinite the union of all the balls almost surely equals $\HypDim$, and the connectedness and infinite-ness of this set then implies $\lambda_{\mathrm{u}}(1)=\lambda_{\mathrm{c}}(1)=0$. For $L>1$, observe that this family of RCMs is monotone: there is a coupling of the $L=1$ model with the $L>1$ model such that if an edge exists in the $L=1$ model then it exists in the $L>1$ model. Therefore if almost every vertex is in the same cluster in the $L=1$ model, then the same is true for the $L>1$ model.

    For part \ref{enum:BooleanCorPart3}, use Theorem~\ref{thm:meanfield} with Assumption~\ref{assump:specialscalePlus}. The form of $\connf$ and $\Pcal\left(\left\{0\right\}\right)<1$ imply that \eqref{eqn:infIntegralDegree} and \eqref{eqn:strongIrreducibilityCondition} hold, and it is simple to show that the boundedness of the support of $\Pcal$ implies the condition \eqref{eqn:bounddegreedensity} holds.
\end{proof}

\subsection{Weight-Dependent Hyperbolic Random Connection Models}
\label{sec:ProofWDRCM}

\begin{proof}[Proof of Corollary~\ref{cor:WDRCMGeneral}]
    First we aim to use Theorem~\ref{thm:NonUniqueness} with Assumption~\ref{assump:specialscale}. The main observation is that
    \begin{equation}
    \label{eqn:SeparateProfileandKernel}
        \int^\infty_0\connf\left(r;a,b\right) r\exp\left(\frac{1}{2}\left(d-1\right)r\right)\dd r = \kappa\left(a,b\right)\int^\infty_0\rho(r)r\exp\left(\frac{1}{2}\left(d-1\right)r\right)\dd r.
    \end{equation}
    Therefore condition \eqref{eqn:WDRCMprofileCondition} ensures that $\int^\infty_0\connf\left(r;a,b\right) r\exp\left(\frac{1}{2}\left(d-1\right)r\right)\dd r<\infty$ for $\Pcal$-almost every $a,b\in\Ecal$, and \eqref{eqn:L2normVolumeScale} becomes the requirement that
    \begin{equation}
        \norm*{\Kcal}_{2\to 2}\int^\infty_0\rho(r)r\exp\left(\frac{1}{2}\left(d-1\right)r\right)\dd r<\infty.
    \end{equation}
    This is then clearly satisfied if $\norm*{\Kcal}_{2\to 2}<\infty$.
\end{proof}

\begin{proof}[Proof of Corollary~\ref{cor:WDRCMSpecific}]
    This follows directly from Corollary~\ref{cor:WDRCMGeneral} once one has evaluated $\norm*{\Kcal}_{2\to 2}$ for each model. These are indeed evaluated in Lemmata~\ref{lem:ProdKernel}--\ref{lem:PAkernel} below.
\end{proof}

\begin{lemma}
\label{lem:ProdKernel}
\begin{equation}
    \norm*{\Kcal^{\mathrm{prod}}}_{2\to 2} = \begin{cases}
    \frac{1}{1-2\zeta} &: \zeta\in\left(0,\frac{1}{2}\right)\\
    \infty &: \zeta\geq \frac{1}{2}.
    \end{cases}
\end{equation}
\end{lemma}

\begin{proof}
We first consider $\zeta<\frac{1}{2}$ and evaluate $\HSNorm{\Kcal^{\mathrm{prod}}}$. We find
\begin{equation}
    \HSNorm{\Kcal^{\mathrm{prod}}} = \left(\int^1_0 \int^1_0 a^{-2\zeta} b^{-2\zeta} \dd a\dd b\right)^\frac{1}{2} = \int^1_0 a^{-2\zeta}\dd a = \frac{1}{1-2\zeta}.
\end{equation}
Therefore $\norm*{\Kcal^{\mathrm{prod}}}_{2\to 2}\leq \frac{1}{1-2\zeta}$ for $\zeta<\frac{1}{2}$. Now let us consider the normalized $L^2$-function
\begin{equation}
    f_1(a) := \left(1-2\zeta\right)^\frac{1}{2}a^{-\zeta}.
\end{equation}
Then
\begin{equation}
    \left(\Kcal^{\mathrm{prod}}f_1\right)(a) = \frac{1}{\left(1-2\zeta\right)^{\frac{1}{2}}}a^{-\zeta},
\end{equation}
and
\begin{equation}
    \norm*{\Kcal^{\mathrm{prod}}f_1}_2 = \frac{1}{1-2\zeta}.
\end{equation}
Therefore $\norm*{\Kcal^{\mathrm{prod}}}_{2\to 2}= \frac{1}{1-2\zeta}$ for $\zeta<\frac{1}{2}$.

For $\zeta\geq \frac{1}{2}$, let us consider $f_2\equiv 1$. It is clear that $f_2\in L^2$ with $\norm*{f_2}_2=1$. We then find
\begin{equation}
    \left(\Kcal^{\mathrm{prod}}f_2\right)(a) = \frac{1}{1-\zeta}a^{-\zeta}.
\end{equation}
Note that for $\zeta\geq \frac{1}{2}$, $\norm*{\Kcal^{\mathrm{prod}}f_2}_2=\infty$ and $\Kcal^{\mathrm{prod}}f_2\not\in L^2$. Therefore $\norm*{\Kcal^{\mathrm{prod}}}_{2\to 2}= \infty$ for $\zeta\geq\frac{1}{2}$.
\end{proof}

\begin{lemma}
\label{lem:StrongKernel}
    \begin{equation}
        \norm*{\Kcal^{\mathrm{strong}}}_{2\to 2}<\infty \iff \zeta<\frac{1}{2}
    \end{equation}
\end{lemma}

\begin{proof}
    For $\zeta<\frac{1}{2}$, we first show that by the symmetry of the kernel
    \begin{equation}
        \HSNorm{\Kcal^{\mathrm{strong}}}^2 = 2\int^1_0\left(\int^b_0 a^{-2\zeta} \dd a\right)\dd b = \frac{1}{\left(1-2\zeta\right)\left(1-\zeta\right)}.
    \end{equation}
    This then bounds $\norm*{\Kcal^{\mathrm{strong}}}_{2\to 2}<\infty$.

    For $\zeta\geq \frac{1}{2}$, let us consider $f\equiv 1$. We find that for $a\in\left(0,1\right)$
    \begin{equation}
        \left(\Kcal^{\mathrm{strong}}f\right)(a) = a^{-\zeta}\left(1-a\right) + \int^a_0b^{-\zeta}\dd t = 
        \begin{cases}
            a^{-\zeta}\left(1-a\right) + \frac{1}{1-\zeta}a^{1-\zeta} &\colon \zeta<1\\
            \infty &\colon \zeta\geq 1.
        \end{cases}
    \end{equation}
    Clearly $\norm*{\Kcal^{\mathrm{strong}}f}_2=\infty$ if $\zeta\geq \frac{1}{2}$, and therefore $\norm*{\Kcal^{\mathrm{strong}}}_{2\to 2}=\infty$.
\end{proof}

While the product and strong (and therefore sum) kernels do indeed have parameter regimes for which $\norm*{\Kcal}_{2\to 2}<\infty$, we can now show that the weak and preferential attachment kernels do not, and therefore our results do not apply to them.

\begin{lemma}
\label{lem:WeakKernel}
    For all $\zeta>0$, $\norm*{\Kcal^{\mathrm{weak}}}_{2\to 2} = \infty$.
\end{lemma}

\begin{proof}
    For $\zeta>0$, let us consider $f(a)= a^{\zeta -\frac{1}{2}}$. We have $f\in L^2$ since $\norm*{f}_2 = \frac{1}{\sqrt{2\zeta}}<\infty$. Now for $a\in\left(0,1\right)$
    \begin{equation}
        \left(\Kcal^{\mathrm{weak}}f\right)(a) = a^{-1-\zeta}\int^a_0b^{\zeta -\frac{1}{2}}\dd b + \int^1_ab^{-\frac{3}{2}}\dd b = \frac{1}{\zeta + \frac{1}{2}}a^{-\frac{1}{2}} + 2\left(a^{-\frac{1}{2}}-1\right).
    \end{equation}
    Therefore $\norm*{\Kcal^{\mathrm{weak}}f}_2=\infty$, and $\norm*{\Kcal^{\mathrm{weak}}}_{2\to 2} = \infty$.
\end{proof}

\begin{lemma}
\label{lem:PAkernel}
    For all $\zeta>0$, $\norm*{\Kcal^{\mathrm{pa}}}_{2\to 2} = \infty$.
\end{lemma}

\begin{proof}
    First we consider $\zeta<\frac{1}{2}$ and the family of functions $f_\varepsilon(a) = a^{\zeta-1}\Id_{a>\varepsilon}$ for $\varepsilon>0$. We have
    \begin{equation}
        \norm*{f_\varepsilon}^2_2 = \int^1_\varepsilon a^{2\zeta-2}\dd a = \frac{1}{1-2\zeta}\left(\varepsilon^{2\zeta-1}-1\right).
    \end{equation}
    Note $\norm*{f_\varepsilon}_2<\infty$ for all $\varepsilon>0$, but that $\norm*{f_\varepsilon}_2\to\infty$ as $\varepsilon\to 0$.

    Now for $a\leq\varepsilon$ we have
    \begin{equation}
        \left(\Kcal^{\mathrm{pa}}f_\varepsilon\right)(a) = \frac{a^{-\zeta}}{1-2\zeta}\left(\varepsilon^{2\zeta-1}-1\right),
    \end{equation}
    and for $a>\varepsilon$ we have
    \begin{equation}
        \left(\Kcal^{\mathrm{pa}}f_\varepsilon\right)(a) = \frac{a^{-\zeta}}{1-2\zeta}\left(a^{2\zeta-1}-1\right) + a^{\zeta-1}\log\frac{a}{\varepsilon}.
    \end{equation}
    Since $\zeta<\frac{1}{2}$, we can identify the dominant term and find
    \begin{equation}
        \norm*{\Kcal^{\mathrm{pa}}f_\varepsilon}^2_2 \asymp \int^1_\varepsilon a^{2\zeta-2}\left(\log a\right)^2\dd a
    \end{equation}
    as $\varepsilon\to 0$. Since $\left(\log a\right)^2\to\infty$ as $a\to 0$,
    \begin{equation}
        \norm*{\Kcal^{\mathrm{pa}}f_\varepsilon}_2 \gg \norm*{f_\varepsilon}_2
    \end{equation}
    as $\varepsilon\to 0$. Therefore $\norm*{\Kcal^{\mathrm{pa}}}_{2\to 2}=\infty$ for $\zeta<\frac{1}{2}$.

    For $\zeta\geq \frac{1}{2}$, consider $f\equiv 1$. Then
    \begin{equation}
        \left(\Kcal^{\mathrm{pa}}f\right)(a) = a^{-\zeta}\int^1_a b^{\zeta-1}\dd b + a^{\zeta-1}\int^a_0 b^{-\zeta}\dd b = \begin{cases}
            \frac{1}{\zeta}a^{-\zeta} - \frac{1}{\zeta} + \frac{1}{1-\zeta} &\colon \zeta<1\\
            \infty &\colon \zeta\geq 1.
        \end{cases}
    \end{equation}
    Therefore $\norm*{\Kcal^{\mathrm{pa}}f}_2=\infty$ for $\zeta\geq \frac{1}{2}$, and $\norm*{\Kcal^{\mathrm{pa}}}_{2\to 2}=\infty$.
\end{proof}

\begin{proof}[Proof of Proposition~\ref{prop:nonperturb}]
    First note that by Mecke's formula, the expected degree of a vertex with mark $a$ is given by
    \begin{multline}
        \lambda\mathfrak{S}_{d-1}\int^1_0\left(\int^\infty_0\rho\left(s^{-1}_{\kappa\left(a,b\right)}\left(r\right)\right)\left(\sinh r\right)^{d-1}\dd r\right) \dd b \\= \lambda\mathfrak{S}_{d-1} \int^1_0 \kappa\left(a,b\right)\left(\int^\infty_0\rho\left(r\right)\left(\sinh r\right)^{d-1}\dd r \right)\dd b.
    \end{multline}
    Therefore since $\kappa\left(a,b\right)>0$ for a positive measure of $\left(a,b\right)\in\left(0,1\right)^2$, the expected degree of a vertex with mark $a$ equals $\infty$ for a positive measure of $a$. In particular this occurs for any $\lambda>0$. Since the expected cluster size is bounded below by the expected degree, $\theta_\lambda(a)=1$ for $\lambda>0$ and a positive measure of $a$, and $\lambda_\mathrm{c}=0$.

    To show that $\lambda_{\mathrm{u}}>0$, we utilize Lemma~\ref{lem:Uniquenesstoqtoq} to show $\lambda_{\mathrm{u}}\geq \lambda_{2\to 2}$ and Lemma~\ref{lem:qtoqintensity_bound} to show $\lambda_{2\to 2}\geq \frac{1}{\norm*{\Opconnf}_{2\to 2}}$. Then the observation \eqref{eqn:SeparateProfileandKernel} means that \eqref{eqn:ProfileL2bound} implies that
    \begin{equation}
        \int^\infty_0\connf(r;a,b) Q_d\left(r\right)\exp\left(\left(d-1\right)r\right)\dd r<\infty
    \end{equation}
    for $\Pcal$-almost every $a,b\in\Ecal$, and Lemma~\ref{lem:TwoToTwoNormBound} shows that $\norm*{\Opconnf}_{2\to 2} = \norm*{\widetilde{\Opconnf}(0)}_{2\to 2}$ if the latter is finite. Since $\int^\infty_0\rho(r) Q_d(r)\left(\sinh r\right)^{d-1}\dd r<\infty$ if \eqref{eqn:ProfileL2bound} holds, $\norm*{\widetilde{\Opconnf}(0)}_{2\to 2}<\infty$ if $\norm*{\Kcal}_{2\to 2}<\infty$. Therefore $\lambda_{2\to 2}>0=\lambda_{\mathrm{c}}$.
\end{proof}

\begin{appendix}

\section{Scaling Functions}
\label{app:ScalingFunctions}
The volume-linear scaling $s_L$ appearing in Assumptions~\ref{assump:specialscale} and \ref{assump:specialscalePlus} is special because it transforms radial integrals in a predictable way. Up to an $r$-independent constant, $\mathbf{V}_d(r)$ equals the volume of a hyperbolic ball with radius $r$, and $s_L$ can therefore be viewed as a conjugation between lengths and volumes. Therefore for any measurable function $\phi\colon \R_+\to \R_+$, we have
    \begin{equation}
    \label{eqn:volumeScalingProperty}
        \int_\HypDim\phi\left(s_L^{-1}\left(\dist{x,\orig}\right)\right)\mu\left(\dd x\right) = L\int_\HypDim\phi\left(\dist{x,\orig}\right)\mu\left(\dd x\right).
    \end{equation}
    Also observe that $L\mapsto s_L$ is a homomorphism: for all $L,M>0$ and $r>0$
    \begin{equation}
    \label{eqn:scalingHomomorphism}
        s_M\left(s_L\left(r\right)\right) = \mathbf{V}_d^{-1}\left(M \mathbf{V}_d\left(\mathbf{V}_d^{-1}\left(L \mathbf{V}_d(r)\right)\right)\right) = \mathbf{V}_d^{-1}\left(ML \mathbf{V}_d(r)\right) = s_{ML}(r).
    \end{equation}
    In particular this means $s_L^{-1} = s_{1/L}$. To give a picture of what the function $s_L(r)$ actually looks like, it is elementary to evaluate the asymptotics of $s_L(r)$. For sequences $\left\{r_n\right\}_{n\in\N}$ and $\left\{L_n\right\}_{n\in\N}$ such that $L_n\to\infty$,
    \begin{equation}
        s_{L_n}(r_n) \sim \begin{cases}
            L_n^\frac{1}{d}r_n &\text{if } r_n\ll L_n^{-\frac{1}{d}},\\
            \frac{1}{d-1}\log L_n &\text{if } L_n^{-\frac{1}{d}}\ll r_n\ll \log L_n,\\
            r_n &\text{if } r_n\gg \log L_n,
        \end{cases}
    \end{equation}
    as $n\to\infty$.
    
    Recall that we refer to the family of functions $\left\{\sigma_L\right\}_{L>0}$ as a scaling function if all $\sigma_L\colon\R_+\to \R_+$ are increasing bijections such that
\begin{itemize}
    \item $\sigma_1(r)=r$,
    \item for all $r>0$, $L\mapsto \sigma_L(r)$ is increasing and $\lim_{L\to\infty}\sigma_L(r)=\infty$,
\end{itemize} Also recall that Assumption~\ref{assump:finitelymany} required of the scaling and adjacency function that
    \begin{equation}
    \label{eqn:EigenValueRatioV2}
        \lim_{L\to\infty}\frac{\max_{a,b\in\Ecal}\int^R_0\connf_L\left(r;a,b\right) \left(\sinh r\right)^{d-1}\dd r}{\max_{a,b\in\Ecal}\int^\infty_0\connf_L\left(r;a,b\right) \left(\sinh r\right)^{d-1}\dd r} = 0
    \end{equation}
    for all $R<\infty$.
    \begin{lemma}
    \label{lem:VolumeLinearNice}
        Suppose $\#\Ecal<\infty$. For all $\connf$ such that $\max_{a,b\in\Ecal}\int^\infty_0\connf\left(r;a,b\right)\dd r>0$, the volume-linear scaling function satisfies \eqref{eqn:EigenValueRatioV2}.
    \end{lemma}

    \begin{proof}
        Our bound for the numerator in \eqref{eqn:EigenValueRatioV2} is independent of the scaling function. For all $a,b\in \Ecal$ and $R\in\left[0,\infty\right]$,
        \begin{equation}
            \int^R_0\connf_L\left(r;a,b\right)\left(\sinh r\right)^{d-1}\dd r \leq \mathbf{V}_d\left(R\right).
        \end{equation}
        
        On the other hand, for all $a,b\in\Ecal$ the volume-linear scaling function means
        \begin{equation}
            \int^\infty_0\connf_L\left(r;a,b\right)\left(\sinh r\right)^{d-1}\dd r = L\int^\infty_0\connf\left(r;a,b\right)\left(\sinh r\right)^{d-1}\dd r.
        \end{equation}

        Therefore, if $\max_{a,b\in\Ecal}\int^\infty_0\connf\left(r;a,b\right)\dd r>0$ then $\max_{a,b\in\Ecal}\int^\infty_0\connf\left(r;a,b\right)\left(\sinh r\right)^{d-1}\dd r>0$ and
        \begin{multline}
            \limsup_{L\to\infty}\frac{\max_{a,b\in\Ecal}\int^R_0\connf_L\left(r;a,b\right) \left(\sinh r\right)^{d-1}\dd r}{\max_{a,b\in\Ecal}\int^\infty_0\connf_L\left(r;a,b\right) \left(\sinh r\right)^{d-1}\dd r} \\\leq \frac{\mathbf{V}_d\left(R\right)}{\max_{a,b\in\Ecal}\int^\infty_0\connf\left(r;a,b\right)\left(\sinh r\right)^{d-1}\dd r}\lim_{L\to\infty}\frac{1}{L}=0.
        \end{multline}
    \end{proof}

An alternative natural option for a scaling function is the length-linear scaling function $\sigma_L(r) = L r$. However, the volume scaling effects of this simple length-linear scaling are rather complicated. In Euclidean $\Rd$, if we scale lengths by $L$ then we scale volumes by $L^d$. Specifically, if we have a measurable function $\phi\colon \R_+\to \R_+$ then
    \begin{equation}
        \int_{\Rd}\phi\left(\frac{\abs*{x}}{L}\right)\dd x = L^d \int_{\Rd}\phi\left(\abs*{x}\right)\dd x.
    \end{equation}
    On the other hand, there is no function $f\colon \R_+\to \R_+$ such that 
    \begin{equation}
        \int_\HypDim\phi\left(\frac{\dist{x,\orig}}{L}\right)\mu\left(\dd x\right) = f(L)\int_\HypDim\phi\left(\dist{x,\orig}\right)\mu\left(\dd x\right)
    \end{equation}   
    for all $\phi$. 
    
    \begin{lemma}
    \label{lem:LenthLinearNice}
        Suppose $\#\Ecal<\infty$. For all $\connf$ such that $\max_{a,b\in\Ecal}\int^\infty_0\connf\left(r;a,b\right)\dd r>0$, the length-linear scaling function satisfies \eqref{eqn:EigenValueRatioV2}.
    \end{lemma}

    \begin{proof}
        As in the proof of Lemma~\ref{lem:VolumeLinearNice}, for all $a,b\in \Ecal$ and $R\in\left[0,\infty\right]$,
        \begin{equation}
            \int^R_0\connf_L\left(r;a,b\right)\left(\sinh r\right)^{d-1}\dd r \leq \mathbf{V}_d\left(R\right).
        \end{equation}

        On the other hand, for all $a,b\in\Ecal$ the length-linear scaling function means
    \begin{multline}
        \int^\infty_0\connf_L\left(r;a,b\right)\left(\sinh r\right)^{d-1}\dd r = L\int^\infty_0\connf\left(r;a,b\right)\left(\sinh L r\right)^{d-1}\dd r \\\geq L\int^\infty_0\connf\left(r;a,b\right)\left(\sinh r\right)^{d-1}\dd r
    \end{multline}
    for $L\geq 1$. The proof then concludes similarly as for Lemma~\ref{lem:VolumeLinearNice}.
    \end{proof}

    It is also worth noting that there are also some adjacency functions for which \eqref{eqn:EigenValueRatioV2} holds for any scaling function.
    \begin{lemma}
    \label{lem:adjacencyNearOrigin}
        Suppose $\#\Ecal<\infty$ and that 
        \begin{equation}
        \label{eqn:adjacencyNearOrigin}
            \liminf_{r\searrow 0}\max_{a,b\in\Ecal}\connf\left(r;a,b\right)>0,
        \end{equation}
    Then for any scaling function \eqref{eqn:EigenValueRatioV2} holds.
    \end{lemma}

    \begin{proof}
        As in the previous lemmata, the numerator in \eqref{eqn:EigenValueRatioV2} is bounded by $\mathbf{V}_d\left(R\right)$, and we aim to show the denominator diverges. 

        Since $\lim_{L\to\infty}\sigma_L(r)=\infty$ for any $r>0$, $\lim_{L\to\infty}\sigma^{-1}_L(r)=0$ for any $r>0$ and 
        \begin{equation}
            \liminf_{L\to\infty}\connf_L\left(r_\star;a,b\right) = \liminf_{r\to0}\connf\left(r;a,b\right)
        \end{equation}
        for all $a,b\in\Ecal$ and $r_*>0$. Then by Fatou's lemma and our suppositions,
        \begin{multline}
            \liminf_{L\to\infty}\max_{a,b\in\Ecal}\int^\infty_0\connf_L\left(r;a,b\right)\left(\sinh r\right)^{d-1}\dd r \\\geq \int^\infty_0\left(\liminf_{L\to\infty}\max_{a,b\in\Ecal}\connf_L\left(r;a,b\right)\right)\left(\sinh r\right)^{d-1}\dd r = \infty
        \end{multline}
        as required.
    \end{proof}
    
    \begin{lemma}
        Suppose $\#\Ecal<\infty$, $\max_{a,b\in\Ecal}\int^\infty_0\connf\left(r;a,b\right)\dd r>0$, and that there exists $\varepsilon>0$ such that
            \begin{equation}
        \max_{a,b\in\Ecal}\esssup_{r<\varepsilon}\connf\left(r;a,b\right) = 0.
    \end{equation}
    Then for any scaling function \eqref{eqn:EigenValueRatioV2} holds.
    \end{lemma}

    \begin{proof}
        Given $R<\infty$, let $L_0=L_0(R)$ satisfy $\sigma_{L_0}\left(\varepsilon\right)> R$. Then for all $L\geq L_0$
        \begin{equation}
            \max_{a,b\in\Ecal}\int^R_0\connf_L\left(r;a,b\right)\left(\sinh r\right)^{d-1}\dd r = 0.
        \end{equation}
        The assumption $\max_{a,b\in\Ecal}\int^\infty_0\connf\left(r;a,b\right)\dd r>0$ then ensures that the denominator of \eqref{eqn:EigenValueRatioV2} is strictly positive for all $L>0$, and therefore the result follows.
    \end{proof}

One may initially (and incorrectly) expect that because scaling functions turn short edges into long edges the expected degree of a vertex will diverge. While this is true for the volume-linear and length-linear scaling functions, and for adjacency functions satisfying the condition \eqref{eqn:adjacencyNearOrigin} in Lemma~\ref{lem:adjacencyNearOrigin}, the following example shows that it is not true in general.

\begin{example}\label{expl:annulus}
    Let $\Ecal$ be a singleton (so we can neglect it from the notation) and consider the adjacency function
    \begin{equation}
        \connf\left(r\right) = \Id_{\left\{1<r<2\right\}}
    \end{equation}
    and scaling function
    \begin{equation}
        \sigma_L\left(r\right) = \begin{cases}
            Lr &\colon r\leq 1,\\
            L + a_L(r-1) &\colon 1< r\leq \frac{L}{1-a_L},\\
            r &\colon r>\frac{L}{1-a_L},
        \end{cases}
    \end{equation}
    where $a_L = o\left(\e^{-L\left(d-1\right)}\right)$.

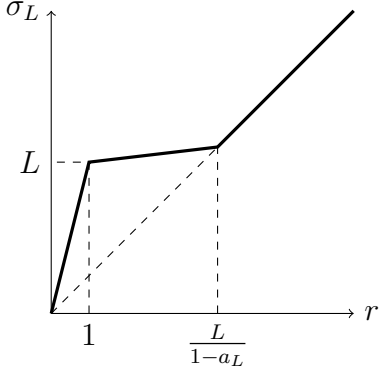
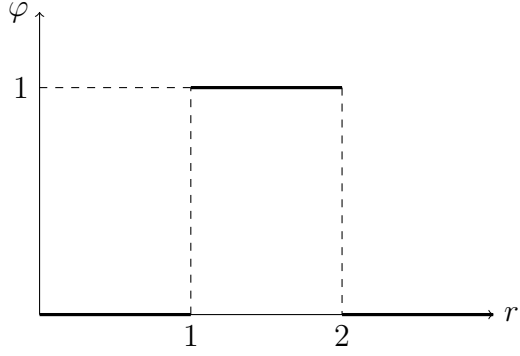
\begin{figure}
    \centering
    \begin{subfigure}{0.33\textwidth}
    \centering
    \begin{tikzpicture}
        \draw[->] (0,0) -- (0,4) node[left]{$\sigma_L$};
        \draw[->] (0,0) -- (4,0) node[right]{$r$};
        \draw[dashed] (0,0)--(4,4);
        \draw[very thick] (0,0) -- (0.5,2) -- (2.2,2.2) -- (4,4);
        \draw[dashed] (0.5,0) node[below]{$1$} -- (0.5,2) -- (0,2) node[left]{$L$};
        \draw[dashed] (2.2,0) node[below]{$\frac{L}{1-a_L}$} -- (2.2,2.2);
    \end{tikzpicture}
    \caption{Example scaling function}
    \end{subfigure}
    \hfill
    \begin{subfigure}{0.65\textwidth}
    \centering
    \begin{tikzpicture}[scale=2]
        \draw[->] (0,0) -- (0,2) node[left]{$\connf$};
        \draw[->] (0,0) -- (3,0) node[right]{$r$};
        \draw[very thick] (1,1.5) -- (2,1.5);
        \draw[very thick] (0,0) -- (1,0);
        \draw[very thick] (2,0) -- (3,0);
        \draw[dashed] (1,1.5) -- (0,1.5) node[left]{$1$};
        \draw[dashed] (1,0) node[below]{$1$} -- (1,1.5);
        \draw[dashed] (2,1.5) -- (2,0) node[below]{$2$};
    \end{tikzpicture}
    \caption{Example adjacency function}
    \end{subfigure}
    \caption{Sketches of the example scaling and adjacency functions in Example~\ref{expl:annulus}.}
    \label{fig:ExampleAnnulus}
\end{figure}

    \begin{lemma}
        For this example,
        \begin{equation}
            \lim_{L\to\infty}\mathbb{E}_{\lambda,L}\left[\#\left\{x\in\eta\colon \adja{\orig}{x}{\xi^\orig}\right\}\right] = 0.
        \end{equation}
    \end{lemma}

    \begin{proof}
    By Mecke's formula and the substitution $r=\sigma_L(y)$,
    \begin{multline}
        \mathbb{E}_{\lambda,L}\left[\#\left\{x\in\eta\colon \adja{\orig}{x}{\xi^\orig}\right\}\right] = \lambda\mathfrak{S}_{d-1}\int^\infty_0\connf_L\left(r\right)\left(\sinh r\right)^{d-1}\dd r \\= \lambda\mathfrak{S}_{d-1}\int^2_1\sigma'_L\left(y\right)\left(\sinh \sigma_L(y)\right)^{d-1}\dd y.
    \end{multline}
    The scaling function $\sigma_L$ is simple to differentiate (almost everywhere) and find
    \begin{equation}
        \sigma'\left(r\right) = \begin{cases}
            L &\colon r\leq 1,\\
            a_L &\colon 1< r\leq \frac{L}{1-a_L},\\
            1 &\colon r>\frac{L}{1-a_L},
        \end{cases}
    \end{equation}
    Therefore
    \begin{equation}
        \int^2_1\sigma'_L\left(y\right)\left(\sinh \sigma_L(y)\right)^{d-1}\dd y = a_L\int^2_1\left(\sinh \left(L + a_L(y-1)\right)\right)^{d-1}\dd y \leq 2a_L\e^{L\left(d-1\right)}\to 0,
    \end{equation}
    where the inequality holds for sufficiently large $L$ (so $a_L\leq \frac{\log 2}{d-1}$).
    \end{proof}
\end{example}

The following example demonstrates that condition \eqref{eqn:EigenValueRatioV2} does not follow from $\sigma_L$ being a scaling function, and it is indeed a required part of Assumption~\ref{assump:finitelymany}.

\begin{example}\label{expl:manyAnnulii}
    Let $\Ecal$ be a singleton (so we can neglect it from the notation) and consider the adjacency function
    \begin{equation}
        \connf\left(r\right) = \sum^\infty_{i=0}\Id_{\left\{2^{-2i-1}<r<2^{-2i}\right\}}.
    \end{equation}
    Before defining our scaling function, we define a  `dummy' (continuous) scaling function $\overline{\sigma}_L$ by its derivative. Fix $R\in\left(0,\infty\right)$ and for $L\geq R$ let
    \begin{equation}
        \overline{\sigma}'_L\left(r\right) = \begin{cases}
        L &\colon r<\frac{R}{L},\\
        1 &\colon \exists i\in\N_0 \colon r\in\left(2^{-2i-1},2^{-2i}\right)\cap\left(\frac{R}{L},1\right),\\
        2^{2i+2} &\colon \exists i\in\N_0 \colon r\in\left(2^{-2i-2},2^{-2i-1}\right)\cap\left(\frac{R}{L},1\right),\\
        1 &\colon r>1.
        \end{cases}
    \end{equation}
    Then we can choose $a_L\in\left(0,1\right)$ such that
    \begin{equation}
    \label{eqn:SizeofA_L}
        a_L = o\left(\frac{1}{\int^1_{\frac{1}{3}}\left(\sinh \overline{\sigma}_L(r)\right)^{d-1}\dd r}\right)
    \end{equation}
    as $L\to\infty$. Now for $L\geq R$ define the true (continuous) scaling function $\sigma_L$ by its derivative:
    \begin{equation}
        \sigma'_L\left(r\right) = \begin{cases}
        L &\colon r<\frac{R}{L},\\
        a_L &\colon \exists i\in\N_0 \colon r\in\left(2^{-2i-1},2^{-2i}\right)\cap\left(\frac{R}{L},1\right),\\
        2^{2i+2} &\colon \exists i\in\N_0 \colon r\in\left(2^{-2i-2},2^{-2i-1}\right)\cap\left(\frac{R}{L},1\right),\\
        1 &\colon r>1.
        \end{cases}
    \end{equation}
    Observe that $\overline{\sigma}_L(r) \geq \sigma_L(r)$ for all $r\geq 0$. The idea behind this scaling function is that for $r<\sigma_L^{-1}\left(R\right)$ the scaling function is just the normal length-linear scaling function. Then for larger $r$ the scaling function is a ``ladder" that is very shallow when $\connf(r)>0$ and very steep when $\connf(r)=0$ (the derivative for $r> 1$ is unimportant other than being $\geq 1$). The precise steepness of the steep segments is chosen so that the total increase in $\sigma_L$ on each of these segments is exactly $1$, and the dummy scaling function was introduced so that we know how shallow the shallow segments needed to be.

    Let us check that $\sigma_L$ is indeed a scaling function. It is clearly strictly increasing in $r$ and diverges as $r\to\infty$. As $L$ increases, the transition point $\frac{R}{L}$ decreases towards $0$, and therefore more and more shallow intervals are revealed. Since each steep interval adds $1$ to $\sigma_L(r)$, we have $\sigma_L(r)\to\infty$ as $L\to\infty$ for all $r>0$. The same argument also shows that $\overline{\sigma}_L$ is a scaling function and in particular that $\lim_{L\to\infty} \overline{\sigma}_L(r)=\infty$ for all $r>0$. This then shows that necessarily $a_L\to0$ as $L\to\infty$.

\begin{figure}
    \centering
    \includegraphics[width=\linewidth]{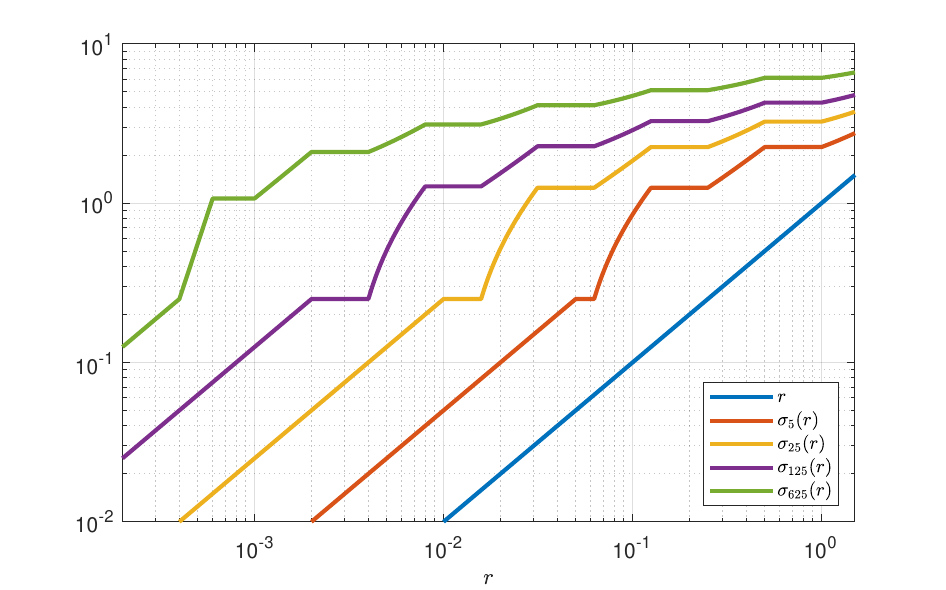}
    \caption{Plot of the scaling function $\sigma_L$ in Example~\ref{expl:manyAnnulii} for $R=0.25$, using MATLAB.}
    \label{fig:ExampleManyAnnulus}
\end{figure}

    Now let $\connf_L(r)$ be the scaled adjacency function using the true scaling function $\sigma_L$ (not the dummy scaling function $\overline{\sigma}_L$). 

    \begin{lemma}
    \label{lem:BadExample}
        For this example, \eqref{eqn:EigenValueRatioV2} does \emph{not} hold.
    \end{lemma}

    \begin{proof}
        By changing variables $y=\sigma_L(r)$ (and back for the last equality) we have
        \begin{multline}
            \int^R_0\connf_L(r)\left(\sinh r\right)^{d-1}\dd r = L\int^\frac{R}{L}_0\connf(y)\left(\sinh Ly\right)^{d-1}\dd y \\
             \geq L\int^\frac{R}{3L}_0\left(\sinh Ly\right)^{d-1}\dd y = \mathbf{V}_d\left(\frac{R}{3}\right).
        \end{multline}
        In the above inequality we have used that $\left(\sinh Ly\right)^{d-1}$ is increasing in $y$ to move the intervals on which $\connf(y)> 0$ towards $0$. The total Lebesgue measure of these intervals intersected with $\left(0,\frac{R}{L}\right)$ is always greater than or equal to $\frac{R}{3L}$ (achieved when $\frac{R}{L}$ coincides with the lower bound of one of these positive intervals), and therefore $\frac{R}{3L}$ appears in the integral bound.

        We now prove that \eqref{eqn:EigenValueRatioV2} does not hold by showing that
        \begin{equation}
            \int^\infty_R\connf_L(r)\left(\sinh r\right)^{d-1}\dd r \to 0
        \end{equation}
        as $L\to \infty$. The same change of variables as above shows that
        \begin{multline}
            \int^\infty_R\connf_L(r)\left(\sinh r\right)^{d-1}\dd r = \int^\infty_\frac{R}{L}\connf(y)\sigma'_L(y)\left(\sinh \sigma_L\left(y\right)\right)^{d-1}\dd y \\\leq a_L \int^1_\frac{R}{L}\connf(y)\left(\sinh \overline{\sigma}_L\left(y\right)\right)^{d-1}\dd y,
        \end{multline}
        where in the inequality we have used that on $r>\frac{R}{L}$ the property $\connf(y)>0$ implies $\sigma'_L(y)=a_L$, that $\overline{\sigma}_L(y)\geq \sigma_L(y)$, and that $\connf(y)=0$ for $y>1$. Then since $\int^1_0\connf(y)\dd y = \frac{2}{3}$ and $\left(\sinh \overline{\sigma}_L\left(y\right)\right)^{d-1}$ is increasing in $y$, we can move the intervals of $\connf$ towards $1$ to get
        \begin{equation}
            a_L \int^1_\frac{R}{L}\connf(y)\left(\sinh \overline{\sigma}_L\left(y\right)\right)^{d-1}\dd y \leq a_L \int^1_\frac{1}{3}\left(\sinh \overline{\sigma}_L\left(y\right)\right)^{d-1}\dd y.
        \end{equation}
        The bound on $a_L$ in \eqref{eqn:SizeofA_L} then gives us the result.
    \end{proof}
\end{example}

\end{appendix}

\vspace{5mm}
\begin{acks}
The author would like to thank the Isaac Newton Institute for Mathematical Sciences, Cambridge, for support and hospitality during the programme \emph{Stochastic systems for anomalous diffusion}, where work on this paper was undertaken. This work was supported in parts by EPSRC grant EP/Z000580/1 and by NSERC of Canada. Thanks are also due to Petr Kosenko for guiding me to the spherical transform, Markus Heydenreich for their discussions about RCMs and percolation on hyperbolic spaces, and Gordon Slade and a referee for their advice on presenting the paper.
\end{acks}

\bibliography{bibliography}{}
\bibliographystyle{alpha}

\end{document}